\documentclass[11pt]{article}
\usepackage{macros}
\date{}

\title{One-dimensional Chern-Simons theory and the $\hat{A}$ genus}
\author{Owen Gwilliam\footnote{University of California, Berkeley, \url{gwilliam@math.berkeley.edu}} \; and Ryan Grady\footnote{Boston University, \url{regrady@bu.edu}}}

\begin{document}

\maketitle

\abstract{We construct a Chern-Simons gauge theory for dg Lie and L-infinity algebras on any one-dimensional manifold and quantize this theory using the Batalin-Vilkovisky formalism and Costello's renormalization techniques. Koszul duality and derived geometry allow us to encode topological quantum mechanics, a nonlinear sigma model of maps from a 1-manifold into a cotangent bundle $T^* X$, as such a Chern-Simons theory. Our main result is that the effective action of this theory is naturally identified with the $\hat{A}$ class of $X$.  From the perspective of derived geometry, our quantization constructs a projective volume form on the derived loop space $\cL X$ that can be identified with the $\hat{A}$ class. }

\tableofcontents

\section{Introduction}

The Atiyah-Singer index theorem and the mathematics around it --- the theory of elliptic and pseudodifferential operators, K-theory, cobordism, and so on --- has a long-standing relationship with quantum field theory \cite{Wit82}, \cite{Windey}, \cite{Getzler}, \cite{IAS}. In particular, the index theorem for Dirac operators appears naturally when one studies supersymmetric quantum mechanics on a Riemannian spin manifold. Our main object of study in this paper is a 1-dimensional quantum field theory that bears a strong resemblance to supersymmetric quantum mechanics, and our project, whose first product is this paper, aims to explore how much of the mathematics around the index theorem arises from this QFT.

In \cite{Cos2}, \cite{Cos3}, Kevin Costello constructed a 2-dimensional field theory with source manifold an elliptic curve and with target the cotangent bundle $T^* X$ of a K\"ahler manifold $X$. His theory recovers the elliptic genus of $X$, allowing the source manifold to vary over the moduli of elliptic curves. Inspired by this work, we sought to construct the analogous 1-dimensional field theory, which is a nonlinear sigma model of maps from a circle into a smooth manifold $T^* X$. Our main result is that the quantized theory recovers the $\hat{A}$ genus of $X$. We now state precisely what we accomplish in this paper.\footnote{Subsequent to the writing of this paper, Costello revised substantially the original draft of \cite{Cos3}. In particular, he developed a beautiful formalism for derived smooth geometry and a precise notion of a projective volume form on well-behaved derived spaces. We have {\em not} revised our work in light of his improvements. See \cite{Grady2} for a statement of our results in this new, elegant formalism.}

In parallel with Costello's work, there are two stages: 
\begin{enumerate}

\item[(1)] we construct a class of 1-dimensional field theories known as 1-dimensional Chern-Simons theories, where the input data is a (possibly curved) $\L8$ algebra $\fg$, and we compute the quantum observables using the Batalin-Vilkovisky formalism;

\item[(2)] we reinterpret a smooth manifold as an $\L8$ algebra, which is an exercise in derived geometry (more accurately, a smooth manifold is described by a sheaf of curved $\L8$ algebras).

\end{enumerate}
Thus in Part I of the paper, we review the definition of a quantum field theory in the formalism of \cite{Cos1} and exhibit how one-dimensional Chern-Simons provides a beautiful example. The main theorem, Theorem \ref{inftytermBg}, characterizes the effective action\footnote{We use Costello's notion of an effective field theory, and ``effective action" here means the action functional for the quantum field theory, i.e., the local functional depending on $\hbar$ that satisfies the quantum master equation and agrees with the classical action functional modulo $\hbar$. In \cite{CG}, this notion is shown to have an interpretation as a ``partition function."} as an invariant, the ``$\hat{A}$ class," of the $\L8$ algebra, but the bulk of the work is in carefully examining the Feynman diagrams of the theory. In Part II we explain the formal geometry and derived geometry that leads to a Lie-theoretic description of smooth geometry. The primary challenge in this part is to identify the invariant of Theorem \ref{inftytermBg} with the usual $\hat{A}$-class of a smooth manifold $X$.

We state our main theorem first in the case associated to a complex manifold $X$, so that it is clear how it parallels Costello's work on the Witten genus. Costello's ICM lecture \cite{Cos2} places our result in context with work of Bressler-Nest-Tsygan \cite{BNT} and his own.

Working in the Batalin-Vilkovisky (BV) formalism, we find that the classical observables of the theory are naturally quasi-isomorphic to the regraded holomorphic de Rham forms $\Omega_{hol}^{-\ast}(T^\ast X)$. This result can be interpreted as a version of the HKR theorem, since the fields are related to the loop space of $T^\ast X$. It may not come as a surprise that the {\em quantum observables} involve a deformation of the HKR isomorphism (indeed, the deformation associated to differential operators on $X$).

\begin{theorem}\label{CplxManifoldThm}
There exists a BV quantization of a nonlinear sigma model from the circle $S^1$ into $T^* X$, where $X$ is a compact K\"ahler manifold, with the following properties:
\begin{itemize}
\item only 1-loop Feynman diagrams appear in the quantization, and 
\item the quantization produces a quasi-isomorphism from the global quantum observables to a deformation of the regraded holomorphic de Rham forms of $T^*X$:
\[
(\Omega^{-*,*}(T^*X)[[\hbar]], \hbar L_\pi + \hbar\{\log (e^{-c_1(X)/2} \Td(X)), -\}).
\]
\end{itemize}
\end{theorem}

Here $L_\pi = [d,\iota_\pi]$ denotes the Lie derivative with respect to the canonical Poisson bivector $\pi$ on $T^* X$, and it is thus a degree $1$ operator. The bracket $\{-,-\}$ denotes the associated degree $1$ Poisson bracket on $\Omega^{-\ast,\ast}(T^\ast X)$, given explicitly by the formula
\[
\{a,b\} = L_\pi(ab) - L_\pi(a) b - (-1)^{|a|} a L_\pi b
\]
with $a$ and $b$ Dolbeault forms. In other words, the quantum observables are a $\hbar$-weighted version of the holomorphic Poisson homology, also known as Koszul-Brylinski homology.

Theorem \ref{CplxManifoldThm} follows from a more general theorem about a gauge theory for $\L8$ algebras. We construct a gauge theory on 1-dimensional manifolds that works for any curved $\L8$ algebra $\fg$. In analogy to the usual 3-dimensional Chern-Simons theory, where our Lie algebra needs a nondegenerate invariant pairing, we need an $\L8$ algebra with a nondegenerate invariant pairing of cohomological degree -2. We use the simplest possible class of such $\L8$ algebras: given $\fg$, let $\DD\fg$ denote the central extension of curved $\L8$ algebras 
\[
0 \to \fg^\vee[-2] \to \DD\fg \to \fg \to 0
\]
where $\fg$ acts on the extension by the shifted coadjoint action. The evaluation pairing induces the desired shifted pairing. Much of the work in the paper is devoted to showing that there is a quantization of this theory that only involves connected Feynman diagrams with at most one loop.

Let $M$ will be a one-dimensional manifold and $\fg$ a curved $\L8$ algebra. For our field theory, the equation of motion (or Euler-Lagrange equation) is the Maurer-Cartan equation for a flat connection on the trivial principal $\DD\fg$-bundle on $M$. (This theory arises by the AKSZ procedure \cite{AKSZ}, an aspect we discuss in describing our results from the perspective of QFT.) 

\begin{theorem}\label{LieThm}
There exists a quantization of this theory on $\RR$, invariant under translation along $\RR$ and under dilation of the $\fg^\vee[-2]$ factor in $\DD\fg$, with the following properties:
\begin{itemize}
\item  only 1-loop Feynman diagrams appear in the quantization, and 
\item the quantization produces a quasi-isomorphism from the global quantum observables of this theory on $S^1$ to
\[
\left( \bigoplus_{n \geq 0} C^*(\fg, \Sym^{n}(\fg^\vee[n]) )[[\hbar]], \hbar L_\pi + \hbar\{\log (\hat{A}(\fg)), -\} \right).
\]
\end{itemize}
\end{theorem}

There is a natural dictionary between geometric and $\L8$ constructions: 
\begin{itemize}
\item an $\L8$ algebra $\fg$ corresponds to a ``space" $B\fg$ whose functions are the Chevalley-Eilenberg cochain complex $C^* \fg$; 
\item the $\fg$-module $\fg[1]$ under the shifted adjoint action corresponds to $C^*(\fg, \fg[1])$, viewed as the vector fields on this space $B\fg$; and
\item the $\fg$-module $\fg^\vee[-1]$ with the shifted coadjoint action corresponds to $C^*(\fg, \fg^\vee[-1])$, viewed as the 1-forms on $B\fg$.
\end{itemize}
Under this Koszul duality correpondence, there are $\L8$ analogs $L_\pi$ and $\{-,-\}$ of those defined on $T^* X$ above, and $\hat{A}(\fg)$ denotes the ``$\hat{A}$ characteristic class" for $B\fg$. In other words, we interpret this theorem as giving a deformation of the Hochschild homology of the derived loop space of $B\fg$. (See \cite{CalaqueRossi} for a beautiful exposition of this Lie-geometry correspondence and many other techniques of relevance to this paper. There is a wealth of literature about the underlying Koszul duality between commutative and Lie algebras.) 

\begin{remark}
We should emphasize that this quantization is essentially unique. In constructing the quantization, we will make choices but we will also show that the space of such choices is contractible. The natural invariant output of the quantized theory is the deformed complex given above. Note that it is defined directly in terms of $\fg$. Moreover, the quasi-isomorphism arises by ``integrating out the nonzero modes" of the theory, which mathematically amounts to removing all dependence on the choices made in constructing the quantization.
\end{remark}

A central goal of this paper is to apply this theorem when the target is a smooth manifold $X$, but several challenges appear. As a result, the main theorem has a more complicated statement: instead of working with a manifestation of Hochschild homology --- the regraded de Rham forms --- we need to work with a version of negative cyclic homology, for reasons that we explain after the statement of the theorem.\footnote{The K\"ahler case is easier than the smooth case because, in essence, the Hodge-to-de Rham spectral sequence collapses. We are witnessing here the cyclic analog of this fact.} 

The classical field theory consists of maps of $S^1$ into $T^* X$ equipped with the action functional 
\[
\Maps(S^1 ,T^*X) \ni \gamma \overset{S}{\longrightarrow} \int_{S^1} \gamma^* \lambda,
\] 
where $\lambda$ is the canonical {\it aka} Liouville 1-form on $T^* X$. Again, using the BV formalism, we find that the classical observables of the theory --- namely, functions on the {\em derived} critical locus of the action functional above --- are naturally quasi-isomorphic to $\Omega^{-\ast}(T^\ast X)$. It may not come as a surprise that the {\em quantum observables} are then related to the negative cyclic homology of $T^*X$, which we identify with $(\Omega^{-*}(T^*X)[[u]], u d)$, where $u$ is a formal variable of cohomological degree 2 and $d$ denotes the exterior derivative with cohomological degree $-1$. In fact, we obtain a twisted version of this complex. Our deformation of the differential involves the $\hat{A}$ class of $X$ in a form modified to work with the negative cyclic homology: let $\hat{A}_u(X)$ denote the element in negative cyclic homology obtained by replacing $ch_{k}(X)$ by $u^k ch_{k}(X)$ wherever it appears in the usual $\hat{A}$ class.

\begin{theorem}\label{ManifoldThm}
There exists a BV quantization of a nonlinear sigma model from the circle $S^1$ into $T^* X$, where $X$ is a smooth manifold, with the following properties:
\begin{itemize}
\item only 1-loop Feynman diagrams appear in the quantization, and 
\item the quantization produces a quasi-isomorphism from the $S^1$-invariant global quantum observables to the following deformation of the negative cyclic homology of $T^*X$:
\[
(\Omega^{-*}(T^*X)[[u]][[\hbar]], ud + \hbar L_\pi + \hbar\{\log (\hat{A}_u(X)), -\}).
\]
\end{itemize}
\end{theorem}

One step in proving this theorem is to reduce to the theorem for curved $\L8$ algebras. To do this, we use the Koszul duality between dg commutative algebras and dg Lie algebras which allows us to identify a smooth manifold $X$, as a space over its de Rham space $X_{dR}$, with the classifying space $B\fg_X$ of a sheaf of curved $\L8$ algebras over $X_{dR}$. (See appendix \ref{DGmanifolds} for an introduction to these spaces.) Essentially, we replace smooth functions $C^\infty_X$ by the de Rham complex of jets of smooth functions. This kind of construction is sometimes known as Gelfand-Kazhdan formal geometry or Fedosov resolutions.

\begin{remark}\label{whythecircleaction}
It may appear strange that we only use cyclic homology in the case of a smooth manifold, but the reason is simple. (The cyclic version works, of course, for an arbitrary $B\fg$.) In the setting of K\"ahler manifolds, the $\hat{A}$ class defined by Atiyah classes lives in the ``backbone" $\oplus_k \Omega^{k,k}$ of the Dolbeault complex and hence {\em does} have degree 0 in the regraded holomorphic de Rham complex. In the setting of smooth manifolds, two separate things go wrong. First, the corresponding scalar Atiyah classes, as defined, say, by Calaque and Van den Bergh \cite{CVDB} (see also \cite{Kap}), vanish. Second, ignoring that issue, the $\hat{A}$ class via Atiyah classes is concentrated in degree 0 and thus cannot agree with the usual $\hat{A}$ class.

To deal with the first issue, we develop an Atiyah class version of the Chern-Weil construction of characteristic classes. To deal with the second, we move to cyclic homology, where the extra grading allows us to obtain a cyclic version of the usual $\hat{A}$ class.

More accurately, we encode the smooth manifold $X$ as an $\L8$ algebra object $B\fg_X$, where Gelfand-Kazhdan formal geometry provides the curved $\L8$ algebra $\fg_X$. As a result, our construction of the global observables involves a complex quasi-isomorphic to (shifted) de Rham forms, and the characteristic classes $ch_{k}(B\fg_X)$ all manifestly have cohomological degree 0 in this construction. The final difficulty is in identifying $ch_{k}(B\fg_X)$ with $ch_k(X)$, and working with negative cyclic homology accomplishes this identification.
\end{remark}

Finally, we remark on the next step in our project. We have shown here that the global quantum observables are quasi-isomorphic to complexes that usually appear as the Hochschild or cyclic cohomology of familiar algebras, but these algebras have not appeared in our discussion thus far. Indeed, the theorems here are one half of a more interesting theorem, which we will prove in a followup paper. In the holomorphic setting, where $X$ is a complex manifold, Bressler, Nest, and Tsygan \cite{BNT} constructed a quasi-isomorphism between the Hochschild homology of the Rees algebra of holomorphic differential operators $\Diff^\hbar(X)$ and $(\Omega^{-*}(T^*X)[[\hbar]], \hbar L_\pi + \hbar\{e^{-c_1(X)/2}\log(\Td(X), -\})$. In our next paper, we will construct the factorization algebra of observables for our 1-dimensional theory and show that it is equivalent to $\Diff^\hbar(X)$, the Rees algebra of differential operators on a complex or smooth manifold $X$. We will then use the formalism of factorization algebras to compute the global observables on the circle, which is equivalent to the Hochschild homology of $\Diff^\hbar(X)$. These two descriptions of the global observables are quasi-isomorphic, and hence we will recover the Bressler-Nest-Tsygan theorem, as well as a smooth analog. Similarly, our results can be interpreted as a path integral derivation of Fedosov's trace map in his approach to deformation quantization \cite{FedDQ}, \cite{FedBook}.

\subsection{Our results from the perspective of derived geometry}

A field theory, classical or quantum, is a geometric construction, and it is useful to pinpoint what our construction means in the language of geometry. Before describing our constructions, we introduce a bit of terminology. Throughout the paper, we use dg manifolds,\footnote{See appendix \ref{DGmanifolds} for a quick introduction to these spaces.} a rather concrete and primitive version of derived geometry well-suited to the explicit computations of field theory. Essentially, a dg manifold is a ringed space where the underlying space is a smooth manifold and the structure sheaf is a sheaf of commutative dg algebras. The key spaces that appear in our field theory are the classifying space of an $\L8$ algebra $B\fg$, the de Rham space $X_{dR}$ of a smooth manifold $X$ (this space is essentially the quotient of $X$ where nearby points are identified), and the derived loop space $\cL M$ of a dg manifold $M$.

Classical field theory fits easily into the language of geometry: a space of fields is simply a mapping space (or space of sections of some bundle) and a classical theory picks out a subspace satisfying some system of equations. Fix an $\L8$ algebra $\fg$ and let $\DD \fg = \fg \oplus \fg^\vee[-2]$ denote the split, square-zero extension of $\fg$. Our classical Chern-Simons theory picks out a space with two equivalent descriptions: 
\begin{itemize}
\item the formal neighborhood of the trivial connection in the moduli space of flat $\DD \fg$ connections on $S^1$ or, equivalently, 
\item the formal completion around the constant maps in the mapping space $\Maps(S^1_{dR}, T^* B\fg)$.
\end{itemize} 
We denote this {\em derived loop space} by $\cL T^*B\fg$.\footnote{There are several spaces that could reasonably be called the derived loop space, but this one is the most relevant for our purposes. In \cite{GGonL8}, we use Costello's formalism for derived smooth geometry to discuss these different options.} The global classical observables for this theory are precisely the functions on this space; it is well-known that the functions on a derived loop space $\cL X$ are quasi-isomorphic to the Hochschild homology complex of the functions on $X$, which is also quasi-isomorphic to the regraded de Rham forms on $X$. Hence, the classical observables are
\[
\sO( \cL T^* B\fg) \simeq \Omega^{-*} (T^* B\fg) = \sO(T[-1]T^*B\fg).
\]
When we quantize, we deform this complex over $\RR[[\hbar]]$.\footnote{The relationship between the derived loop space, Hochschild and cyclic homology, and the Chern character has been explored extensively \cite{BZN}, \cite{TV}.}

There is an appealing interpretation of our main theorem, rooted in the idea that quantization amounts to taking a path integral.\footnote{In his update of \cite{Cos3}, Costello provided a precise version of this idea via the notion of a ``projective volume form." Again, \cite{Grady2} explains how our result produces a projective volume form.} In other words, we should view our quantization as constructing a ``volume form'' on the space $\cL B \fg$ and thus allows us to define an integration map for functions on $\cL B \fg$ (i.e., ``the expected value" of any observable).\footnote{The attentive reader will have noted that the volume form lives on $\cL B\fg$ but the space of fields is $\cL T^*B\fg$. In fact, the space of fields is isomorphic to $T^*[-1] \cL B\fg$, so that functions on the fields are polyvector fields on $\cL B\fg$. These polyvector fields provide an obfuscated version of the de Rham complex fof $\cL B\fg$ and hence encode integration on that space. See \cite{Cos3} for more discussion.} That $\hat{A} (B \fg)$ shows up has a natural Lie-theoretic interpretation.  Recall that the power series $\hat{A}$ arises, speaking loosely, by comparing the Lebesgue measure on a Lie algebra to a Haar measure on its formal group $\hat{G}$.\footnote{For $G$ a compact Lie group, the derivative of the exponential map is
\[
d(\exp a) = \det \left(\frac{1 - e^{\ad(a)}}{\ad(a)} \right) da,
\]
where $a$ denotes a coordinate on $\fg$ \cite{BGV}.} Our quantization can then be thought of as pulling back the volume form, arising from quantization, on the ``formal group" $\cL B \fg$ to the ``Lie algebra" $T[-1]B \fg$ via an exponential map. 



Let us explain how Koszul duality allows us to phrase a smooth manifold as an $\L8$ algebra (in the process, we will explain the exponentiation remark from the preceding paragraph); we rely on the work of Kapranov and Costello. For a smooth manifold $X$, the tubular neighborhood theorem allows us to identify a small neighborhood of the diagonal $X \overset{\Delta}{\hookrightarrow} X \times X$ with a small neighborhood of the zero section of the tangent bundle $X \overset{i}{\hookrightarrow} TX$. Essentially, one chooses a metric on $X$ and then uses the induced exponential map to send a small ball around $0$ in each tangent fiber $T_x X$ to a small transverse slice to $(x,x) \in X \times X$. This argument works in the setting of formal geometry and says that we can identify a formal neighborhood of the diagonal in $X \times X$ with a formal neighborhood of the zero section of the tangent bundle $TX$. We will now provide a Lie-theoretic interpretation of this construction. 

Denote by $\hat{X}$ the formal neighborhood of the diagonal and by $T[0]X$ the formal neighborhood of the zero section in $TX$. Following Kapranov's work in the holomorphic setting \cite{Kap}, one shows that the Atiyah class of the tangent sheaf $\cT_X$ equips $\cT_X[-1]$ with the structure of a sheaf of $\L8$ algebras over $X$. There is then a fiberwise exponential map $\exp: T[-1]X \to \cL X$, which at each point $x \in X$ maps the $\L8$ algebra $T_x[-1]X$ to its formal group, the based derived loops $\Omega_x X$. By delooping, we obtain a map $B\exp: BT[-1]X \cong T[0]X \to B\cL X \cong \hat{X}$, which is precisely the kind of exponential map arising from the tubular neighborhood picture.

Building on Kapranov's picture, Costello \cite{Cos3} showed that in the holomorphic setting, this sheaf of $\L8$ algebras $\cT_X[-1]$ arises from a sheaf of curved $\L8$ algebras over the de Rham space $X_{dR}$. In the smooth setting, we have an analogous situation: we have a homotopy pullback diagram\footnote{This pullback diagram is a straightforward consequence of the fact that $X_{dR}$ can be presented as a groupoid $\hat{X} \rightrightarrows X$.}
\[
\begin{xymatrix}{
\hat{X} \ar[r] \ar[d] & X \ar[d]^\pi \\
X \ar[r]_\pi & X_{dR}
}
\end{xymatrix}
\]
and there exists a sheaf of curved $\L8$ algebras $\fg_X$ over $X_{dR}$ so that there is an isomorphism $B\exp: B\fg_X \to X$ over $X_{dR}$ and so that the pullback sheaf $\pi^*\fg_X$ over $X$ is isomorphic to $\cT_X[-1]$. 

Now we introduce the field theory. In studying a classical field theory, we focus on the derived critical locus of the action functional. In our case, the derived critical locus corresponds to the mapping space from $S^1_{dR}$ into $T^*[0]X$, the formal neighborhood of the zero section of the cotangent bundle. Thus, the classical field theory is simply the study of the derived loop space $\cL T^*[0]X \cong T^*[-1] \cL X$. In general, the BV formalism for quantization works cleanly with shifted cotangent bundles such as $T^*[-1] \cL X$.



\subsection{Our results from the perspective of quantum field theory}

An appealing and powerful aspect of the formalism for QFT developed by Costello \cite{Cos1} is that it naturally combines derived geometry and Feynman diagrammatics, which makes it straightforward to work with QFTs in the style of geometry: we can construct QFTs in a local-to-global fashion, build families of QFTs, and describe the obstructions, deformations, and automorphisms of quantizations via explicit cochain complexes. 

We study here a nice and rather simple example of this formalism. Of course, because we are working with one-dimensional spaces, the analytic aspects are well-behaved. Thus, we hope this paper will help those already familiar with QFT to see how to work with the other aspects of Costello's machine.

There are two topics that might be of especial interest from the point of view of QFT. First, we sketch in section \ref{large volume limit} how to recover our action functional by a two step process: first, take the infinite-volume limit of the usual action for a free particle wandering around a Riemannian manifold, and second, apply the Batalin-Vilkovisky formalism. These two steps together recover an AKSZ action functional. This process is a simple source of several beautiful theories, and it leads to the holomorphic Chern-Simons theory studied by Costello \cite{Cos3}. The second topic is a method for converting some nonlinear sigma models into gauge theories and, equivalently, interpreting some gauge theories as sigma models. In essence, there is a correspondence between commutative dg algebras and dg Lie algebras (or $\L8$ algebras) known as Koszul duality. Since perturbative field theory can be organized in the style of algebraic geometry, it should be no surprise that one might use Koszul duality to translate between sigma models and gauge theories. We give an example of this translation in Part II of the paper, where we encode a sigma model of a circle mapping into a smooth manifold $T^*X$ with a Chern-Simons theory on the circle with Lie algebra $\fg_X \oplus \fg_X^\vee[-2]$. Alternatively, one can view our procedure as a repackaging of Gelfand-Kazhdan formal geometry or Fedosov resolutions. Again, similar techniques are used for holomorphic geometry in \cite{Cos3}, where we learned these ideas.

Our work here clearly has a strong relationship with the vast literature on deformation quantization (notably, \cite{BNT}, \cite{Will}, \cite{Tsygan}, \cite{Dolg}, \cite{PPT}). In our next paper, where we construct the factorization algebra of observables, we hope to make these connections more precise.

\subsection{Acknowledgements}

Our original inspiration for this project was to better understand Costello's holomorphic Chern-Simons theory \cite{Cos3} by developing its one-dimensional analog. Many of the techniques and ideas in this paper are thus due to Kevin Costello, and we thank him for his tremendous generosity in discussing all aspects of this work. We have also benefited from conversations with Damien Calaque, Gr\'egory Ginot, John Francis, Theo Johnson-Freyd, David Nadler, Fr\'ed\'eric Paugam, Claudia Scheimbauer, Josh Shadlen, Yuan Shen, Mathieu Stienon, Stephan Stolz, Justin Thomas, Ping Xu, and Shilin Yu. 

During revisions of this work, OG was supported as a postdoctoral fellow by the National Science Foundation under Award DMS-1204826.

\part{One-dimensional Chern-Simons theories}

Our goal in this part of the paper is to construct a one-dimensional Batalin-Vilkovisky (BV) theory which we call {\it one-dimensional Chern-Simons theory}. As a perturbative gauge theory, it depends on a choice of Lie or $\L8$ algebra, which must possess an invariant inner product of cohomological degree $-2$. We can construct such an $\L8$ algebra from any finite rank $\L8$ algebra $\fg$: simply take $\fg \oplus \fg^\vee [-2]$ and use the evaluation pairing. In this case, the obstructions to BV quantization vanish and the quantized action functional has an interpretation in terms of characteristic classes. Over the course of Part I, we will introduce and explain all the terms appearing in the Theorem \ref{LieThm}.



We begin by reviewing the notion of a BV theory, define one-dimensional Chern-Simons theory, and then discuss renormalization and quantization after Costello \cite{Cos1}. Then we develop the language of characteristic classes in the setting of $\L8$ algebras to prove Theorem \ref{LieThm}.

\section{Defining the theory}

\subsection{Free theories in the BV formalism}

\begin{definition}
A {\em free field theory} consists of the following data:
\begin{itemize}

\item a manifold $M$ and a finite rank, $\ZZ$-graded (super)vector bundle $\pi: E \rightarrow M$ whose smooth sections are denoted $\sE$;

\item a degree $-1$ antisymmetric pairing on the bundle $\langle- , -\rangle_{loc}: E \ot E \rightarrow \Dens(M)$ that is fiberwise nondegenerate;\footnote{Note that this induces a pairing $\langle -, - \rangle$ on compactly supported sections of $\sE$ by integration.}

\item a degree $+1$ differential operator $Q: \sE \rightarrow \sE$ that is square-zero and skew-self-adjoint for the pairing;

\item a degree $-1$ differential operator $Q^\ast: \sE \rightarrow \sE$ that is square-zero, self-adjoint for the pairing, and whose commutator $[Q,Q^\ast]$ is a generalized Laplacian.

\end{itemize}
\end{definition}

A free field theory has the quadratic action functional $S_{free}: \phi \rightarrow \langle \phi, Q\phi\rangle$.

\begin{remark}
This definition differs from \cite{Cos1} by including the ``gauge-fixing operator" into the definition. From the point of view of Costello's formalism, this is unappealing, but in practice we'll fix an operator $Q^\ast$ once and for all and never worry about the space of such operators. Thanks to theorem 10.7.2 in chapter 5 of \cite{Cos1}, we know that this choice does not affect structural aspects of our theory since our space of gauge fixes is contractible.
\end{remark}

Given a free field theory, we can consider modifying $S_{free}$ by adding ``interaction" terms $I$. The kinds of functional $I$ that we allow will be motivated by physics, but we need some notation first.

\begin{definition}
The space of {\em functionals} on the fields $\sE$ is $\sO(\sE) := \csym(\sE^\vee)$.
\end{definition}

\begin{remark}
Whenever we work with these big vector spaces, like the fields, we work in the appropriate category of topological vector spaces and we always use the natural morphisms, tensor products, and so on, for that context. Here $\sE^\vee$ denotes the continuous dual to $\sE$ (hence, distributions) and $\csym$ denotes the completed symmetric algebra, where we construct this algebra using the continuous product and completed projective tensor product.
\end{remark}

Note that we use the completed symmetric algebra -- aka the ``formal power series" on fields -- because we are working perturbatively, and hence in a formal neighborhood of the classical solution.

Not every functional can serve as an action functional, however. A basic premise of field theory is that the physics must be {\em local} (so there is no ``spooky" action-at-a-distance). Here is a precise expression of that idea.

\begin{definition}
A functional $I \in \sO(\sE)$ is {\em local} if every homogeneous component $I_k \in \Sym^k \sE^\vee$ is of the form
\[
I_k(\phi) = \sum _{\alpha \in A} \int_{x \in M} \left(D_{\alpha, 1} \phi \big|_x \right) \cdots \left(D_{\alpha, k} \phi \big|_x \right) d\mu_\alpha (x),
\]
where each $D_{\alpha,i}$ is a differential operator from $\sE$ to $C^\infty(M)$, $d\mu_\alpha$ is a density on $M$, and the index set $A$ for the integrals is finite.
\end{definition}

We denote the space of local functionals by $\sO_{loc}(\sE)$.

This definition captures our intuition of locality because it says the functional only cares about the local behavior of the field point by point on the manifold $M$ (i.e., depends only on the Taylor series, or $\infty$-jet, of $\phi$). It doesn't compare the behavior of the field at separated points or regions. For example, it excludes functionals like $\phi(p)\phi(q)$, where $p$ and $q$ are distinct points. 

\begin{definition}
An {\em interaction term} $I$ is a local functional whose homogeneous components are cubic and higher. An {\em action functional} associated to a free field theory is a functional $S_{free} + I$, with $I$ an interaction term. 
\end{definition}

In the BV formalism, a classical theory is an action functional satisfying the {\em classical master equation}. The pairing $\langle -,-\rangle$ on fields $\sE$  induces a skew-symmetric pairing ${-,-}$ of degree $1$ on local functionals.\footnote{On a finite-dimensional graded vector space $V$, a pairing on $V$ induces a dual pairing on $V^\vee$, but we are working with infinite-dimensional vector spaces where analytic issues arise. The expected dual pairing is only defined on a subset of the dual space.} This pairing behaves like a Poisson bracket.

\begin{definition}\label{def:classicalBVtheory}
A {\em classical BV theory} consists of a a free BV theory and an interaction term $I \in \sO_{loc}(\sE)$ satisfying the classical master equation $QI + \frac{1}{2}\{I,I\} = 0$.
\end{definition}

\subsection{Perturbative Chern-Simons theory on a 1-manifold}

Although Chern-Simons theory typically refers to a gauge theory on a 3-manifold, the perturbative theory has analogues over a manifold of any dimension. The only modification is to use dg Lie algebras, or $\L8$ algebras, with an invariant pairing of the appropriate degree. In section \ref{large volume limit}, we will explain how the AKSZ formalism for nonlinear sigma models relates to the gauge theories described here, and so we defer a general discussion of the sigma model motivation to that section. Nonetheless, we hope the analogy to the usual Chern-Simons theory is transparent.

\subsubsection{The simplest example}

Our base space is $S^1$. Let $\fg = \bigoplus_n \fg_n$ be a graded Lie algebra with a nondegenerate invariant symmetric pairing $\langle-,-\rangle_\fg$ of degree -2.  Notice that this means $\fg[1]$ comes equipped with a nondegenerate skew-symmetric pairing of degree $0$, which we denote $\langle-,-\rangle_{\fg[1]}$. The space of fields is $\sE = \Omega^*(S^1) \ot \fg[1]$. The pairing on $\fg$ induces a symplectic form of degree -1 on $\sE$ by
\[
\langle \alpha, \beta \rangle = \int_{t \in S^1} \langle \alpha(t) \wedge \beta(t) \rangle_{\fg[1]}.
\]
More explicitly, let $\alpha = \sum_n A^0_n(t) + A^1_n(t) dt$ denote an element of $\sE$, where $A^0_n(t)$ and $A^1_n(t)$ are smooth functions on $S^1$ taking values in $\fg[1]_n$, and likewise for $\beta = \sum B^0_n(t) + B^1_n(t) dt$. Then
\[
\langle \alpha, \beta \rangle = \sum_n \int_{t \in S^1} \langle A^0_n(t), B^1_{-n}(t) \rangle_{\fg[1]} + \langle A^1_n(t), B^0_{-n}(t) \rangle_{\fg[1]} \, dt.
\]
We fix a metric on $S^1$ and let $Q = d$, the exterior derivative, and $Q^* = d^*$, its adjoint with respect to our metric. The action functional is
\[
S(\alpha) = \frac{1}{2} \langle \alpha, d \alpha \rangle + \frac{1}{6} \langle \alpha, [\alpha, \alpha ] \rangle.
\]

\subsubsection{The general case}

Let $\fg$ now denote a curved $\L8$ algebra over a commutative dg algebra $R$ (for the definition, see appendix \ref{app:L8}). Let the maps $\ell_n: \wedge^n \fg \rightarrow \fg$ denote the brackets (i.e., these are the Taylor components of the derivation $d_\fg$ defining the $\L8$ structure).\footnote{Note $\ell_1(r x) = d_R(r) x \pm r \ell_1(x)$, so it is not $R^\sharp$-linear.} We want an $\L8$ algebra that has a nondegenerate invariant symmetric pairing $\langle-,-\rangle$ of degree -2. Note that the sum $\fg \oplus \fg^\vee[-2]$ is equipped with an $\L8$ structure using the coadjoint action:
\[
[X + \lambda, Y + \mu] = [X,Y] + X \cdot \mu - Y \cdot \lambda,
\]
where $X,Y \in \fg$ and $\lambda, \mu \in \fg^\vee[-2]$. Moreover, $\fg \oplus \fg^\vee[-2]$ also has a natural pairing
\[
\langle X + \lambda, Y + \mu \rangle = \lambda(Y) \pm \mu(X),
\]
which is invariant by construction.

Our space of fields is
\[
\Omega_{S^1} \ot \left( \fg[1] \oplus \fg^\vee[-1] \right).
\]
Our action functional is
\[
S(\phi) = \frac{1}{2}\langle \phi, d \phi \rangle + \sum_{n = 0}^\infty \frac{1}{(n+1)!} \langle \phi, \ell_n(\phi^{\ot n}) \rangle.
\]
Note that when $\fg$ is just a graded Lie algebra, $\ell_2$ is the only nontrivial bracket and we recover the action functional from the simple example above.

\begin{remark}
In this setting, the action functional takes values in the graded algebra $R^\sharp$, not $\RR$ or $\CC$. (We are implicitly studying a family of theories with base ring $R^\sharp$. For a discussion of this notion, see chapter 2, section 13 of \cite{Cos1}.) We view the action $S$ as a sum of a free action functional $\frac{1}{2}\langle \phi, Q \phi \rangle$, where $Q = d + \ell_1$, and an interaction term $I_{CS} = \langle \phi, \ell_0 \rangle + \sum_{n = 2}^\infty \frac{1}{(n+1)!} \langle \phi, \ell_n(\phi^{\ot n}) \rangle.$ We group $\ell_1$ into the ``kinetic term" so that the interaction term $I$ is $R^\sharp$-linear. (We may occasionally, and abusively, slip into viewing $\ell_1$ as part of the interaction term.)
\end{remark}

\begin{lemma}
The Euler-Lagrange equation of $S$ is the Maurer-Cartan equation for the trivial $\fg \oplus \fg^\vee[-2]$-bundle on $S^1$.
\end{lemma}

\begin{proof}
Let $\phi + \epsilon \psi$ be a first-order deformation of $\phi$, i.e., $\epsilon^2 = 0$. Then 
\[
S(\phi + \epsilon \psi) - S(\phi) = \epsilon \left( \langle \psi, d\phi \rangle + \sum_{n=0}^\infty \frac{n+1}{(n+1)!} \langle \psi, \ell_n(\phi^{\ot n}) \rangle \right),
\]
by the $\fg$-invariance of the pairing $\langle -,-\rangle$ and integration by parts. For this integral to vanish for any choice of $\psi$, we need $\phi$ to satisfy 
\[
d\phi + \sum_{n=0}^\infty \frac{1}{n!} \ell_n(\phi^{\ot n}) = 0,
\]
the $\L8$-version of the Maurer-Cartan equation.
\end{proof}

\section{BV quantization and renormalization group flow}\label{section2}

In the previous section we defined the classical Chern-Simons action functional.  In this section we review the notions of a quantum field theory and quantization of a classical field theory in the framework of effective field theory developed in \cite{Cos1}. Constructing an effective field theory from a classical field theory consists of two stages, as described by the following figure.

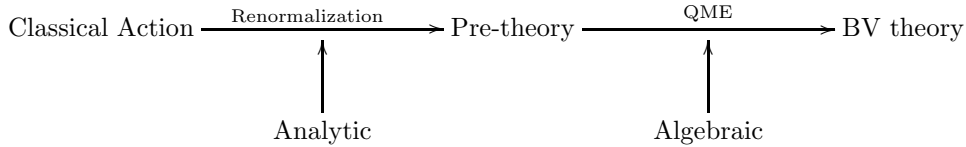
\[
\xymatrix{ \text{Classical Action} \ar[0,2]^{\text{Renormalization}} &&\text{Pre-theory} \ar[0,2]^{\text{QME}} && \text{BV theory} \\
&\text{Analytic} \ar[u] && \text{Algebraic} \ar[u]
}
\]

To begin, we assume that we have the data of a classical field theory $(\sE , \langle - , - \rangle_{loc}, Q, Q^\ast )$ with classical interaction $I \in \sO_{loc} (\sE)$.  The Feynman diagrams arising from this data typically lead to divergent integrals (as we are trying to multiply distributions), and we need some method of renormalization to resolve these analytic issues. In the framework of effective field theory, we introduce a parameter $L \in (0, \infty)$ called the {\it length scale} and work with families of action functionals $\{I[L]\}$ --- no longer local --- parametrized by $L$.  We require the functionals at different length scales to be related by ``integrating out the fields at intermediate length scales;" the notion of renormalization group flow (RG flow) provides a precise interpretation of this idea. All the Feynman diagrams appearing in such a family yield well-defined integrals. The first stage of quantizing our classical field theory consists of finding a family of functionals $\{I [L]\}$ with $I[L] \in \sO (\sE) [[\hbar]]$ such that
\[
\lim_{L \to 0} I [L] = I \text{ modulo } \hbar.
\]
The physical meaning of the above limit is that in the classical limit ($L \to 0$) of our quantum theory (determined by $\{I[L]\}$), the fields become fully local and hence interact at points. 

Even after the analytic issues are overcome, there is an algebraic aspect to address: we need our theory to satisfy the quantum master equation (QME), which, speaking casually, insures that our theory leads to a well-defined ``measure" on the space of fields. (For an overview of the QME and its meaning, we again direct the reader to \cite{Cos1}.) There is a cochain complex, determined just by the classical theory, that encodes all the algebraic aspects of BV quantization; we call it the {\em obstruction-deformation complex} for the classical theory. In particular, the obstruction-deformation complex of a theory $\sT$ describes the formal neighborhood of $\sT$ inside the space of classical BV theories with the same underlying free BV theory. It is a nontrivial result of \cite{Cos1} that the obstructions to solving the QME are cocycles in this complex.

\subsection{Locality}

In order for the classical limit of our theory to exist, modulo $\hbar$, we need some locality conditions on our functionals $\{I[L]\}$.  The scale $L$ interaction term $I[L] \in \sO (\sE) [[\hbar]]$ has a decomposition into homogeneous components
\[
I[L] = \sum_{i,j} \hbar^i I_{i,j} [L] ,
\]
with $I_{i,j}[L] \in \Sym^j(\sE^\vee)$. We then require that for each index $(i,j)$ there exists a small $L$ asymptotic expansion
\[
I_{i,j} [L] \simeq \sum_{k \in \ZZ_{\ge 0}} g_k (L) \Upsilon_k, 
\]
with $g_k \in C^\infty (0,\infty)$ a smooth function of $L$ and $\Upsilon_k \in \sO_{loc} (\sE )$ a local functional. This expansion must be a true asymptotic expansion in the weak topology on $\sO (\sE)$.

\subsection{The renormalization group flow}\label{RGflow}

Given an asymptotically local family of interactions $\{I[L]\}$, we next want them to satisfy the {\it renormalization group equation} (RGE).  The RGE expresses the notion that the interaction at length scale $L$ is related to interaction at length scale $\epsilon$ by integrating over all fields with wavelengths between $\epsilon$ and $L$.  Mathematically, we write the RGE as
\[
I[L] = W(P_\epsilon^L , I[\epsilon]) ,
\]
where $P_\epsilon^L $ is the propagator and $W$ is a weighted sum over Feynman graphs.  We now describe these operators.

Let $D = [Q,Q^\ast]$ be the generalized Laplacian associated to our classical field theory.  For $t \in \RR_{>0}$, let $K_t \in \sE \otimes \sE$ denote the heat kernel for $D$, where our convention for kernels is that for any $\phi \in \sE$,
\[
\int_M \langle K_t (x,y) , \phi (y) \rangle_{loc} = (e^{-tD} \phi ) (x) .
\]
Note that we use the symplectic pairing rather than the more conventional evaluation pairing.

\begin{definition}\label{propagator}
For a classical field theory $(\sE , \langle - , - \rangle_{loc}, Q, Q^\ast )$ the {\it propagator with ultra-violet cut off} $\epsilon$ {\it and infrared cut off} $L$ is given by
\[
P_\epsilon^L = \int_\epsilon^L (Q^\ast \otimes 1) K_t \, dt .
\]
For $\epsilon >0$, $P_\epsilon^L$ is a smooth section of $E \boxtimes E$.
\end{definition}

\begin{example}
For one dimensional Chern-Simons with values in the $\L8$-algebra $\fg[1] \oplus \fg^\vee [-1]$, we can write the propagator explicitly.  Let 
\[
\Cas_\fg = \text{Id}_\fg + \text{Id}_{\fg^\vee} \in \left( \fg[1] \otimes \fg^\vee [-1] \right) \oplus \left( \fg^\vee [-1]\otimes \fg [1] \right) 
\]
be the Casimir, where $\text{Id}_\fg \in \fg[1] \otimes \fg^\vee [-1]$ corresponds to the identity element of $\text{End} (\fg ) = \fg \otimes \fg^\vee$. In this setting $K_t$ is just the one-dimensional heat operator tensored with $\Cas_\fg$, and hence
\[
P_\epsilon^L = \int_\epsilon^L t^{-3/2} \lvert x_1 - x_2 \rvert e^{- \lvert x_1 - x_2 \rvert^2 / t} \Cas_\fg \, dt ,
\] 
up to some constants. In our case, the limit where $\epsilon$ goes to zero and $L$ goes to infinity is a Heaviside step function. In particular, 
\[
P_0^\infty = \pi \sign(x_1 - x_2) \Cas_\fg,
\]
where $\sign(x) = 1$ if $x > 0$ and $\sign(x) = -1$ if $x < 0$. The fact that small length scales are well-behaved insures that we avoid most of the usual analytic challenges in quantum field theory; this feature is one way in which the one-dimensional case is easier than the higher-dimensional analogues.
\end{example}

With propagator in hand, we proceed to define the {\it renormalization group flow operator} $W( P_\epsilon^L , -) : \sO (\sE)[[\hbar]] \to \sO (\sE)[[\hbar]]$.\footnote{There is a subtlety that this operator is really only defined on those functionals that are at least cubic modulo $\hbar$, but we suppress this requirement in the notation throughout.}  For $\gamma$ a stable graph\footnote{This means that each vertex $v$ has an ``internal genus" $g(v) \in \NN$. Moreover, a genus 0 vertex must have valence greater than 2, and a genus 1 vertex must have valence greater than 0. The genus and valence of a vertex picks out an associated homogeneous component $I_{i,j}$ of the action.} and interaction functional $I \in \sO (\sE )[[\hbar]]$, we define the Feynman graph weight
\[
W_\gamma (P_\epsilon^L , I) : \sE^{\otimes T(\gamma)} \to \CC ,
\]
where $T(\gamma)$ indicates the number of tails of $\gamma$, as follows: 
\begin{itemize}
\item Use the decomposition $I = \sum \hbar^i I_{i,j}$ to label vertices of $\gamma$: to a vertex $v$ with genus $i$ and valence $j$, assign $I_{i,j}$. 
\item Label each internal edge by the propagator $P_\epsilon^L$. 
\item Now contract these tensors to obtain the desired map $W_\gamma (P_\epsilon^L , I)$. 
\end{itemize}
For details see \cite{Cos1}.

\begin{definition}
The {\it renormalization group flow operator} from scale $\epsilon$ to scale $L$ is a map $\sO (\sE ) [[\hbar]] \to \sO (\sE) [[\hbar]]$ given by
\[
W(P_\epsilon^L , I) \overset{\text{def}}{=} \sum_\gamma \frac{ \hbar^{g(\gamma)} }{\lvert \text{Aut } \gamma \rvert} W_\gamma (P_\epsilon^L, I),
\]
where the sum is over all connected stable graphs $\gamma$.
\end{definition}

\begin{remark}\label{parametrixremark}
One could choose a different parametrix $\Phi$ for the operator $[Q,Q^\ast]$, i.e., a symmetric distributional section of $E \boxtimes E$ of cohomological degree $+1$ with proper support such that
\begin{itemize}
\item[(i)] $\Phi$ is closed with respect to $Q \otimes 1 + 1 \otimes Q$;
\item[(ii)] $( [ Q , Q^*] \otimes 1) \Phi - K_0$ is a smooth section of $E \boxtimes E$.
\end{itemize}
Given a parametrix $\Phi$, we have an associated propagator $P (\Phi) = (Q^\ast \otimes 1 ) \Phi .$ The renormalization group flow and BV formalism continue to make sense with respect to $\Phi$, see \cite{CG}.
\end{remark}

\begin{definition}
A {\em pre-theory} is an asymptotically-local family of interaction functionals $\{ I[L] \}$ satisfying the RGE
\[
I[L] = W(P_\epsilon^L, I[\epsilon])
\]
for all $0 < \epsilon < L < \infty$.
\end{definition}

\subsection{The quantum master equation}\label{QMEsection}

Let $K_L \in \sE \otimes \sE$ be the heat kernel at length scale $L$ as defined in the preceding section.  We define an operator $\Delta_L : \sO (\sE) \to \sO (\sE)$, called the {\it BV Laplacian}, as contraction with $K_L$. Two properties of this operator are that $\Delta_L^2 =0$ and $[Q, \Delta_L] =0$.  We define the {\it BV bracket} at scale $L$ 
\[
\{ - , - \}_L : \sO (\sE) \otimes \sO (\sE) \to \sO (\sE)
\]
by the formula
\[
\{ I, J \}_L = \Delta_L (I J) - (\Delta_L I ) J - (-1)^{\lvert I \rvert} I (\Delta_L J) .
\]
It follows that $\{ - , -\}_L$ is a derivation in each slot, satisfies the Jacobi identity, and that both $Q$ and $\Delta_L$ are derivations with respect to $\{-,-\}_L$.

The BV Laplacian and bracket have a nice (and equivalent) description in terms of Feynman graphs (see chapter 5, section 9 of \cite{Cos1}).  

\begin{definition}
A pre-theory $\{I[L]\}$ satisfies the {\it quantum master equation} (QME) if for each length scale $L$ we have
\[
Q I [L] + \hbar \Delta_L I[L] + \frac{1}{2} \{ I[L] , I[L] \}_L = 0.
\]
\end{definition}

The RG flow and BV structures interlock to insure that if a pre-theory $I[L]$ satisfies the QME at scale $L$, then $I[L']$ also satisfies the QME at scale $L'$. See Lemma 5.9.2.2 of \cite{Cos1}.

\subsection{Definition of quantization}\label{satisfiesCME}

With all these definitions in hand, we give the definition of a quantum BV theory from \cite{Cos1}. Recall the discussion preceding definition \ref{def:classicalBVtheory}. The BV bracket $\{-,-\}_0$, which is dual to the shifted symplectic pairing on fields, is not well-defined on all functionals, but it is well-defined on local functionals. Our interaction term $I_{CS}$ satisfies the classical master equation because $\sE[-1] = \Omega^* \ot \left( \fg \oplus \fg^\vee[-2] \right)$ is an $\L8$ algebra; in other words, $d + \{ I_{CS}, -\}$ makes $\sO(\sE)$ into the Chevalley-Eilenberg complex of $\sE[-1]$.

\begin{definition}
Let $I \in \sO_{loc} (\sE)$ be a local action functional defining a classical BV theory.  A {\it quantization} of $I$ is a family of effective interactions $\{I[L]\}$ with $I[L] \in \sO (\sE) [[\hbar]]$ such that
\begin{enumerate}
\item $\{I[L]\}$ satisfies the renormalization group equation;
\item $I[L]$ satisfies the locality condition (i.e., there is a small $L$ asymptotic expansion);
\item $I[L]$ satisfies the scale $L$ quantum master equation;
\item The classical limit of $\{I[L]\}$ is $I$, i.e., $\lim_{L \to 0} I [L] = I \text{ modulo } \hbar$.
\end{enumerate}
\end{definition}

\section{Quantizing Chern-Simons}

In the next two sections we give a quantization of our classical Chern-Simons action functional. As might be expected for a one-dimensional theory, no complications arise, such as analytic issues or obstructions to BV quantization.

\subsection{Taking the naive approach}

We begin by ignoring the analytic issues and explore what kind of Feynman diagrams would appear if we could simply run the RG flow from scale 0 to scale $L$. Since all these Feynman diagrams are in fact well-defined (see Proposition \ref{noctterms}), we will have a pre-theory $\{I[L]\}$ and it will remain to show that this theory satisfies the QME. It does, and so our naive approach leads to a quantization of Chern-Simons. 

Let $\phi = (\alpha, \beta) \in \left( \Omega^* (M) \ot \fg [1] \right) \oplus \left( \Omega^* (M) \ot \fg^\vee[-1] \right)$ be a field. Observe that our classical action functional becomes
\[
S(\phi) = \langle \beta, d\alpha \rangle + \sum_{n=0}^\infty \frac{1}{(n+1)!} \langle \beta, \ell_n (\alpha^{\ot n}) \rangle,
\]
because the brackets $\ell_n$ vanish when more than one $\beta$ appears and $\langle -,-\rangle$ is cyclically invariant. Thus the interaction term has homogeneous components $I_k$ where $I_k$ takes in $k-1$ copies of $\alpha$ and one copy of $\beta$. As a consequence, the vertices arising from our theory have the form
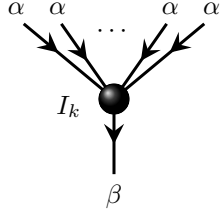
\begin{center}
\begin{tikzpicture}[decoration={markings,
   mark=at position 0.6cm with {\arrow[black,line width=.8mm]{stealth}}}];
\draw[postaction=decorate, line width=.4mm] (-1.2,1) -- (0,0);
\draw[postaction=decorate, line width=.4mm] (-0.7, 1) -- (0,0);
\draw[postaction=decorate, line width=.4mm] (0.7,1) -- (0,0);
\draw[postaction=decorate, line width=.4mm] (1.2,1) -- (0,0);
\draw[postaction=decorate, line width=.4mm] (0,0) -- (0,-1);
\draw[ball color=black]  (0,0) circle (.2);
\draw (-0.6,-0.1) node {$I_k$};
\draw (-1.3,1.2) node {$\alpha$};
\draw (-0.75,1.2) node {$\alpha$};
\draw (0.75,1.2) node {$\alpha$};
\draw (1.3,1.2) node {$\alpha$};
\draw (0,0.9) node {$\cdots$};
\draw (0,-1.3) node {$\beta$};
\end{tikzpicture}
\end{center}
where the direction of the tail indicates whether the input lives in $\Omega^* \ot \fg[1]$ or $\Omega \ot \fg^\vee[-1]$. Moreover, as our pairing $\langle -, - \rangle$ arises from the evaluation pairing between $\fg$ and $\fg^\vee$, the propagator for our theory
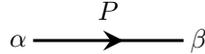
\begin{center}
\begin{tikzpicture}[decoration={markings,
   mark=at position 1.2cm with {\arrow[black,line width=.8mm]{stealth}}}];
\draw[postaction=decorate, line width=.4mm] (-1,0) -- (1,0);
\draw (-1.2,0) node {$\alpha$};
\draw (1.2,0) node {$\beta$};
\draw (0, 0.4) node {$P$};
\end{tikzpicture}
\end{center}
is a directed edge. 

Notice that the kind of connected, directed graphs we can construct from such vertices and edges is highly constrained: we can make trees, wheels, or wheels with trees attached. Here is an example of a wheel.
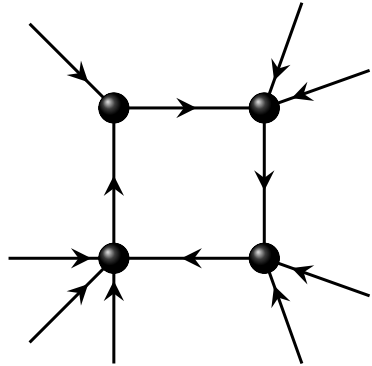
\begin{center}
\begin{tikzpicture}[decoration={markings,
   mark=at position 1.1cm with {\arrow[black,line width=.8mm]{stealth}}}];
\draw[postaction=decorate, line width=.4mm] (2,0) -- (0,0);
\draw[postaction=decorate, line width=.4mm] (0,0) -- (0,2);
\draw[postaction=decorate, line width=.4mm] (0,2) -- (2,2);
\draw[postaction=decorate, line width=.4mm] (2,2) -- (2,0);
\draw[postaction=decorate, line width=.4mm] (-1.4,0) -- (0,0);
\draw[postaction=decorate, line width=.4mm] (0,-1.4) -- (0,0);
\draw[postaction=decorate, line width=.4mm] (-1.12,-1.12) -- (0,0);
\draw[postaction=decorate, line width=.4mm] (-1.12,3.12) -- (0,2);
\draw[postaction=decorate, line width=.4mm] (3.4,2.5) -- (2,2);
\draw[postaction=decorate, line width=.4mm] (2.5,3.4) -- (2,2);
\draw[postaction=decorate, line width=.4mm] (3.4,-0.5) -- (2,0);
\draw[postaction=decorate, line width=.4mm] (2.5,-1.4) -- (2,0);
\draw[ball color=black]  (0,0) circle (.2);
\draw[ball color=black]  (0,2) circle (.2);
\draw[ball color=black]  (2,0) circle (.2);
\draw[ball color=black]  (2,2) circle (.2);
\end{tikzpicture}
\\A wheel with four vertices.
\end{center}
In particular, observe that
\begin{itemize}
\item every tree is ``rooted" by its solitary outward pointing tail (which takes in $\beta$);
\item every one-loop graph $\gamma$ only has inward pointing tails, so $W_\gamma$ is a functional only on $\Omega^*(M) \ot \fg[1]$;
\item the connected graphs have at most one loop.
\end{itemize}

\subsection{The naive quantization has no analytic issues}

It is a general fact that the weight $W_\gamma$ of a tree $\gamma$ is always well-defined. Hence, if we run the RG flow modulo $\hbar$ on the classical Chern-Simons action, we obtain a well-defined functional. 

The next step is to consider the weight of a one-loop graph. The following lemma is specific to one dimensional Chern-Simons, though similar computations hold true in other dimensions (compare 14.3.1 of \cite{Cos3}).

\begin{prop}\label{noctterms}
Let $I_{CS}$ denote classical interaction functional for one-dimensional Chern-Simons on the $\L8$ algebra $\fg \oplus \fg^\vee[-2]$.  For all connected graphs $\gamma$ with one loop,
\[
\lim_{\epsilon \to 0} W_\gamma (P_\epsilon^L, I_{CS})
\]
exists.
\end{prop}

A preliminary step in the proof is the following structural result for perturbative Chern-Simons theories on connections for the trivial bundle.

\begin{lemma}
The weight $W_\gamma (P_\epsilon^L, I_{CS})$ decomposes as a product
\[
W_\gamma (P_\epsilon^L, I_{CS}) = W^{\mathfrak{g}}_\gamma (P_\epsilon^L, I_{CS}) W^{an}_\gamma (P_\epsilon^L, I_{CS}),
\]
where $W^\mathfrak{g}$ arises from contracting tensors in $\mathfrak{g}$ and $W^{an}$ comes from contracting tensors in $C^\infty (M)$.
Further, $W^{\mathfrak{g}}_\gamma (P_\epsilon^L, I_{CS})$ is independent of $\epsilon$ or $L$.
\end{lemma}

\begin{proof}
The weight of a graph $W_\gamma (P_\epsilon^L, I_{CS})$ is given by contracting tensors in 
\[
\sE = \Omega^\ast (M) \otimes ( \mathfrak{g}[1] \oplus \mathfrak{g}^\vee [-1]).
\]
By considering the explicit presentation of the propagator (see section \ref{RGflow}), we see that for each interior edge we are just integrating $(Q^\ast \otimes 1)K_t$, where $K_t$ is the scalar heat kernel tensor the Casimir of the $\L8$-algebra. Hence we can contract in each factor separately.  Note that this is really a consequence of $[d , l_1 ] =0$, which tells us that $K_t$ the kernel for $D = [Q,Q^\ast]$ is just a simple tensor.
\end{proof}

Proposition \ref{noctterms} now follows from a Feynman diagram computation which we have relegated to Appendix \ref{nocttermproof}. In a nutshell, the analytic weight of a wheel leads to an integral that is well-defined as $\epsilon \to 0$, as is shown by some explicit if tedious calculus.

A consequence of proposition \ref{noctterms} is that we obtain an effective field theory (although it remains to show that it satisfies the QME).

\begin{definition}\label{naivecor}
The {\em naive quantization} of $I_{CS}$ is the family of functionals $I_{naive}[L] = I_{naive}^{(0)} + \hbar I_{naive}^{(1)}$, where
\[
I_{naive}^{(0)} = \sum_{\gamma \in \text{Trees}} \frac{1}{|\Aut \gamma|}W_\gamma(P_0^L, I_{CS})
\]
and
\[
I_{naive}^{(1)} = \sum_{\gamma \in \text{One-loop graphs}} \frac{1}{|\Aut \gamma|}W_\gamma(P_0^L, I_{CS}).
\]
By construction, $I_{naive}^{(1)}$ is only a functional on $\Omega^* \ot \fg[1]$.
\end{definition}

\subsection{A symmetry of this theory}\label{actionsection}

The simplicity of this quantization is striking, as {\it a priori} one might expect Feynman diagrams with arbitrarily many loops to appear in the quantization. We provide here a kind of structural explanation for this fortuitous simplicity, as it provides insight both into the theory under consideration and into the question of how to construct classical theories with one-loop quantizations.

Essentially, we only get one-loop graphs because the classical action functional of Chern-Simons only depends linearly on $\Omega^*_M \ot \fg^\vee[-1]$.\footnote{Alternatively, we can view our theory as a sigma model with target $T^*B\fg$. Our action functional then depends linearly on rescaling of the cotangent fibers.\\} Moreover, the action of $\GG_m$\footnote{We use $\GG_m$ because we can work with $\fg$ over $\RR$ or $\CC$, and we don't want to muddle the notation.\\} by rescaling $\Omega^*_M \ot \fg^\vee$ is compatible with the RG flow and the BV structure. Hence we can ask for quantizations that have the same $\GG_m$ action as the classical action functional.

Recall that 
\[
\sE = \Omega^\ast_M \otimes \fg[1] \oplus \Omega^\ast_M \otimes \fg^\vee[-1].
\]
Let $\GG_m$ act on $\sE$ via
\[
z \cdot (\alpha + \beta) = \alpha + z^{-1} \beta .
\]
Define an action of $\GG_m$ on $\sO (\sE)$ with $\mu (z) : \sO (\sE) \to \sO (\sE)$ given by
\[ 
(\mu (z) F )(\phi) = F ( z^{-1} \cdot \phi ) .
\]
Notice that with this action of $\GG_m$, the classical action functional has weight one.  Indeed, $\mathrm{Sym}^n (\Omega^\ast_M \otimes \fg[1])$ has weight zero for all $n$, and $\mathrm{Sym}^n (\Omega^\ast_M \otimes \fg^\vee [-1])$ has weight $-n$. We extend $\mu (z) $ to an action on $\sO (\sE) [[\hbar]]$ by declaring $\hbar$ to have weight one. This weight is a natural consequence of the desire that the path integral be $\GG_m$-invariant: heuristically, the integrand is $\exp(S/\hbar)$. Since the classical action has weight $1$, we scale $\hbar$ to compensate. The following lemma, borrowed from \cite{Cos3}, is then a straightforward computation.

\begin{lemma}\label{actionlemma}
The following operations are $\GG_m$-invariant.
\begin{enumerate}
\item[(1)] The renormalization group flow operator $W(P_\epsilon^L, -) : \sO (\sE)[[\hbar]] \to \sO (\sE)[[\hbar]].$
\item[(2)] The differential $Q :  \sO (\sE)[[\hbar]] \to \sO (\sE)[[\hbar]].$ 
\item[(3)] The quantized differential $\widehat{Q}_L = Q + \hbar \Delta_L$, where $\Delta_L$ is the BV Laplacian.
\end{enumerate}
\end{lemma}

Additionally we have the following.
\begin{lemma}
The BV bracket $\{ - , - \}_L : \sO (\sE) [[\hbar]] \otimes_{\CC [[\hbar]]} \sO (\sE) [[\hbar]] \to \sO (\sE) [[\hbar]]$ is of weight -1.  Hence, $\{I_{CS},-\}$ is of weight zero.
\end{lemma}

A quantization is $\GG_m$-invariant when $I[L]$ has weight 1 with respect to the action $\mu$. We can then ask what a $\GG_m$-invariant quantization would look like.  By the following proposition, if one exists, then only tree-level and one-loop Feynman diagrams appear in the quantized action functional.

\begin{prop}\label{oneloop}
Consider one-dimensional Chern-Simons on a circle $S^1$. If $\{I[L]\}$ is a $\GG_m$-invariant quantization then for each
\[
I[L] = \sum I^{(i)} [L] \hbar^i \in \sO (\sE)[[\hbar]],
\]
$I^{(i)} =0$ for $i>1$ and further $I^{(1)}$ lies in the subspace
\[
\sO (\Omega^\ast_{S^1} \otimes \fg [1]) \subset \sO (\sE).
\]
In other words, a $\GG_m$-invariant quantization only has one-loop terms.
\end{prop}

\begin{proof}
If $I[L]$ has weight one then $I^{(i)}$ must be of weight $1-i$.  There are no negative weight spaces of $\sO (\sE)$, hence $I^{(i)} =0$ for $i>1$.  Lastly, $I^{(1)}$ must be of weight zero, so indeed $I^{(1)} \in \sO(\Omega^\ast_{S^1} \otimes \fg[1])$.
\end{proof}

\begin{remark}
This proposition works for the analogous Chern-Simons theory on arbitrary compact $n$-manifolds.
\end{remark}

%

\section{The obstructions to satisfying the QME}

We have found a $\GG_m$-invariant one-loop quantization $\{ I_{naive}[L]\}$, but this quantization does not necessarily satisfy the quantum master equation (QME), as described in section \ref{QMEsection}.  There is also an action of $\RR$ on the domain by translation (or rotation, for a circle), and we are interested in quantizations invariant under translation as well. By definition, the obstruction to satisfying the QME at scale $L$ is 
\[
O[L] = \hbar^{-1} \left( Q I_{naive}[L] + \frac{1}{2} \{ I_{naive} [L] , I_{naive} [L] \}_L + \hbar \Delta_L I_{naive}[L]  \right), 
\]
where $\{ - , - \}_L$ and $\Delta_L$ are the scale $L$ BV bracket and Laplacian respectively. We will show in this section that this obstruction vanishes, and hence the naive quantization gives a quantum BV theory.

\subsection{Reminder on obstructions}

The space of local functionals $\sO_{loc}(\sE)$ is a graded vector space, and the operator $\{S_{CS}, - \} = d + \{I_{CS},-\}$ makes it into a cochain complex (since $S_{CS}$ satisfies the classical master equation).  We want to restrict attention to translation-invariant local functionals, so from hereon we will only work with the cochain complex $\sO_{loc}(\sE)^\RR$, where the superscript indicates invariance with respect to translation.\footnote{More generally, if we put a group or Lie algebra as a superscript, we mean the invariant subspace.}

As shown in \cite{Cos1}, the obstruction element $O[L]$ for any putative quantization of a classical BV theory has the following properties: it is compatible with the RG flow, its limit as $L \to 0$ exists, and this limit is a local functional. We denote the $L \to 0$ limit by
\[
O \in \sO_{loc} (\Omega^\ast (M) \otimes \fg [1] )^{\RR} \subset \sO_{loc}(\sE)^\RR.
\]
Our obstruction $O$ is an element of cohomological degree 1 and is closed with respect to the differential $d + \{ I_{CS} , - \}$. 

In order to find a quantization which satisfies the QME we need to find a trivialization for $O$. Typically that entails finding an element $J$, where 
\[
J \in \sO_{loc} (\Omega^\ast (M) \otimes \fg [1] )^\RR
\]
is of degree 0 such that $Q J + \{ I_{CS} , J \} = O$.  However, in our setting, we find that the obstruction $O$ vanishes in cohomology and no such $J$ is necessary.

\subsection{The obstruction-deformation complex}


Recall that the BV bracket $\{-,-\}$ is actually a Poisson bracket of degree $+1$ and since our action functional satisfies the classical master equation $\{S,S\} =0$, we obtain a differential graded Lie algebra $(\sO_{loc} (\sE)[-1] , \{S,-\}, \{-,-\})$ by shifting the obstruction-deformation complex down by one.  It is proven in \cite{CG} that this dg Lie algebra encodes a formal deformation problem: how to deform this classical BV theory to infinitesimally nearby classical BV theories.  In particular, first-order deformations of our action functional $S$ are classified by $H^0 (\sO_{loc}(\sE) [-1])$, and $H^{-1}(\sO_{loc} (\sE)[-1])$ describes the infinitesimal automorphisms of the theory (e.g., conserved quantities). As remarked in the preceding paragraph, the obstruction to BV quantization lives in $H^1(\sO_{loc}(\sE) [-1])$, a non-obvious but helpful fact.\footnote{What makes this fact interesting is that it produces a relationship between two distinct moduli problems. Every quantum BV theory has an associated classical BV theory by taking the $\hbar^0$ term of the action functional. Thus, there is a map from the moduli functor of quantum BV theories to the moduli of classical BV theories. The dg Lie algebra $\sO_{loc}(\sE)[-1]$ describes the moduli of {\em classical} BV theories, but it knows about trying to lift to {\em quantum} BV theories.}

In this paper, for consistency with the conventions of \cite{Cos1}, we work with the obstruction-deformation complex $(\sO_{loc} (\sE) , \{S, -\})$ as just a cochain complex. This is justified as our primary aim is to show the vanishing of the obstruction.  The deformations of our classical theory will be studied in future work.

We now compute the  obstruction-deformation complex for our one-dimensional Chern-Simons theory. Note that the computation doesn't depend on the choice of $L_\infty$-algebra, but does depend on the dimension of the domain (i.e., it depends on the fact that our fields are forms on $\RR$ with values in an $L_\infty$-algebra).

The  obstruction-deformation complex for us is
\[
(\sO_{loc} (\Omega^\ast (\RR)\otimes \fg [1])^{\RR} , Q + \{ I_{CS} , - \}),
\]
as we only want to consider action functionals that are translation-invariant and $\GG_m$-invariant. Since a local functional consists of a ``Lagrangian" (i.e., a function on the infinity-jet of a field) and a density on the base manifold, a translation-invariant local functional must be constructed from a translation-invariant Lagrangian and a translation-invariant density. On $\RR$, there is only a one-dimensional space of such densities, namely $\RR \, dx$, where $dx$ is the standard Lebesgue measure. Moreover, a translation-invariant Lagrangian is determined by its behavior at one point in $\RR$. As the $\infty$-jet of a field at a point can be viewed as an element of the space $\fg [[x, dx]]$, it is also easy to describe the space of such Lagrangians.

It should thus come as no surprise that the  obstruction-deformation complex is quasi-isomorphic to a smaller complex given as the translation invariant forms on $\RR$ tensored with the reduced Chevalley-Eilenberg complex of the $\L8$-algebra $\fg$.

\begin{prop}\label{obscpx}
Let $\fg$ be an $L_\infty$ algebra. There is a quasi-isomorphism
\[
(\sO_{loc} (\Omega^\ast (\RR)\otimes \fg [1])^{\RR} , d + \{ I_{CS} , - \}) \simeq \Omega^\ast (\RR)^\RR [1] \otimes_\RR C_{red}^\ast (\fg),
\]
where the $\RR$ action arises from translation on the base manifold $\RR$. In sum, the  obstruction-deformation complex is quasi-isomorphic to
\[
C_{red}^\ast (\fg) \oplus C_{red}^\ast (\fg)[1].
\]
\end{prop}

\begin{proof} 
A local functional is given by integrating a function on the $\infty$-jets against a density. Indeed, by Lemma 6.7.1 of Chapter 5 in \cite{Cos1}, we have a quasi-isomorphism
\[
(\sO_{loc} (\Omega^\ast (\RR)\otimes \fg [1])/C^\infty)^{\RR} \simeq (\Dens_\RR)^\RR \otimes^\LL_{\RR[\partial/\partial x]} \sO (J(\Omega^\ast (\RR)\otimes \fg )_0)/\RR ,
\]
where $J(\Omega^\ast (\RR)\otimes \fg )_0$ indicates jets at $0 \in \RR$ and $C^\infty$ is short hand notation for ``constant functions" in $ \sO (J(\Omega^\ast (\RR)\otimes \fg ))$ (i.e., functionals on jets that are independent of the jets themselves). 

The rightmost factor $\sO (J(\Omega^\ast (\RR)\otimes \fg )_0)/\RR$ can be identified with the reduced Chevalley-Eilenberg complex for $\fg$ via the Poincar\'e lemma, as follows.  At the origin $0 \in \RR$, we know that the fiber of jets $J(\Omega^*(\RR) \ot \fg)$ can be identified with the $\L8$ algebra $\fg [[x, dx]]$, where we include the exterior derivative as part of the differential. That is, 
\[
d(Y x^n) = (dY) x^n + (-1)^{|Y|} Y x^{n-1} dx
\] 
for any $Y \in \fg$. Hence $\sO(J(\Omega^*(\RR) \ot \fg)_0) \cong C^*(\fg [[x, dx]])$. The Poincar\'e lemma on $\RR[[x,dx]]$ then implies that $C^*(\fg [[x, dx]]) \simeq C^*(\fg)$. Alternatively, if we view $\fg$ as a trivial $\RR[\partial / \partial x]$ module, then the inclusion $\fg \hookrightarrow \fg[[x,dx]]$ is an $\RR [\partial / \partial x]$-linear quasi-isomorphism of $\L8$ algebras.

The only translation-invariant densities on $\RR$ are of the form $r\, dx$ for $r \in \RR$, so we find
\[
(\sO_{loc} (\Omega^\ast (\RR)\otimes \fg [1])/C^\infty)^{\RR} \simeq \RR \, dx \otimes^\LL_{\RR[ \partial/\partial x]} C^\ast_{red} (\fg) .
\]
We compute this derived tensor product by resolving $\RR \, dx$ as a right $\RR[\partial / \partial x]$-module:
\[
\begin{array}{rll}
{\displaystyle \RR \otimes_\RR \RR \left [\frac{\partial}{\partial x}\right ] } &\xrightarrow{\delta}& {\displaystyle \RR \, dx \otimes_\RR \RR \left [ \frac{\partial}{\partial x}\right ]}, \\[2ex]
{\displaystyle r \left ( \frac{\partial}{\partial x} \right )^k } & \mapsto & {\displaystyle r \; dx \left ( \frac{\partial}{\partial x} \right )^{k+1}}.
\end{array}
\]
Let $(\cR^\ast , \delta)$ denote this resolution. Then $\cR^\ast \otimes_{\RR[ \partial/\partial x]} C^\ast_{red} (\fg)$ is equal to $C^\ast_{red} (\fg)[1] \oplus C^\ast_{red} (\fg)$, as $C^\ast_{red} (\fg)$ has the trivial $\RR[ \partial/\partial x]$ action.
\end{proof}

\begin{cor}\label{obscpxcor}
The $\RR$-invariant  obstruction-deformation complex is quasi-isomorphic to $\Omega^1_{cl} (B \fg) \oplus \Omega^1_{cl} (B \fg) [1]$.
\end{cor}

\begin{remark}
By the closed 1-forms $\Omega^*_{cl}(B\fg)$, we mean the complex
\[
\Omega^1 \overset{d}{\to} \Omega^2 \overset{d}{\to} \Omega^3 \to \cdots,
\]
i.e., the truncated de Rham complex.
\end{remark}

\begin{proof}
What remains is to make explicit the quasi-isomorphism $C^\ast_{red} (\fg) \simeq \Omega^1_{cl} (B \fg )$.  Note that $C^\ast_{red} (\fg)$ is given by the two term complex
\[
R [1] \to \sO (B \fg) ,
\]
where we denote by $R$ the commutative dg algebra over which $\fg$ is defined. Consider the augmented de Rham complex
\[
\Omega^*_{aug}(B\fg) := R[1] \to \sO(B\fg) \to \Omega^1(B\fg) \to \Omega^2(B\fg) \to \cdots,
\]
which is acyclic.\footnote{In this setting, the de Rham complex can be viewed as a double complex, since the terms $R$, $\sO(B\fg)$, and so on, are themselves cochain complexes. If we filter by this ``internal grading," we get a spectral sequence whose initial page is simply the de Rham complex over the {\em graded} algebra $R^\#$ of $\csym(\fg^\vee[-1])$, without any internal differential. We can then apply the usual retraction to see that this first page is acyclic. If we are working over a dg manifold --- as we will later --- then we are working sheaf-theoretically, so we apply this same argument on small, contractible opens.} There is a projection map $\Omega^*_{aug}(B\fg) \to C^\ast_{red} (\fg)$ of the form
\[
\begin{xymatrix}{
 R \ar[d]_{id} \ar[r] & \Omega^0 \ar[d]_{id} \ar[r]^d & \Omega^1 \ar[d]_{id} \ar[r]^d & \Omega^2 \ar[d]_{id} \ar[r] & \cdots \\
 R \ar[r] & \Omega^0 \ar[r] & 0 \ar[r] & 0 \ar[r] & \cdots
}
\end{xymatrix}
\]
whose kernel is precisely $\Omega^1_{cl}$. Thus we have an exact triangle of complexes $\Omega^1_{cl}[-1] \to \Omega^*_{aug} \to C^*_{red}$ where the middle term is acyclic. Hence we have an isomorphism $C^*_{red} \to \Omega^1_{cl}$ by rotating the triangle.
\end{proof}


\subsection{Structural aspects of the obstruction theory}

Here we deduce general results about the obstruction in Chern-Simons theory.  The results are very similar to those presented in section 16 of \cite{Cos3}.  The main result is that we can express the obstruction $O$ as a sum over graphs with at most one loop (wheels and trees).  By decomposing the obstruction into the product of an analytic factor and Lie-theoretic factor, we show that the total obstruction vanishes for one-dimensional Chern-Simons theory with values in a $\L8$-algebra of the form $\fg \oplus \fg^\vee [-2]$. In fact, the obstruction vanishes for {\em two} independent reasons: the analytic factor is zero on the nose, and the Lie-theoretic factor is cohomologically trivial!

Let $\gamma$ be a stable graph and $e$ an edge of $\gamma$ that connects two distinct vertices. We call such an $e$ a \emph{non-loop edge}  Define
\[
W_{\gamma , e} (P_\epsilon^L , K_\epsilon - K_0 , I_{CS} ) \in \sO (\sE)
\]
to be the weight of $\gamma$, where we use $I_{CS}$ to weight vertices, we use $P_\epsilon^L$ to weight all edges of $\gamma$ except $e$, and we use $K_\epsilon - K_0$ to weight $e$. 

\begin{prop}\label{obsexpression} The scale $L$ obstruction can be expressed as
 $$O[L] = \sum_\gamma \sum_{\substack{\text{$e$ a non-}\\ \text{loop edge}}} \frac{1}{\lvert \mathrm{Aut} (\gamma)\rvert} \lim_{\epsilon \to 0} W_{\gamma, e} (P_\epsilon^L , K_\epsilon - K_0 , I_{CS} ) ,$$
 where the sum is over all stable graphs $\gamma$ with at most one loop.
 \end{prop}
 
In order to surmount the notational barrier, we split the proof into a sequence of lemmas. We begin by recalling the compatibility between the RGE and the QME.
 
\begin{lemma}[5.11.1.1 of \cite{Cos1}]\label{RGEQME}
Let $\delta$ be a parameter of cohomological degree $-1$ and satisfy $\delta^2 = 0$.  Fix $\epsilon > 0$. Given a functional $I$, let $I[L]$ denote its image $W(P_\epsilon^L,I)$ under RG flow. Then
\[
Q I[L] + \frac{1}{2} \{ I[L], I[L] \}_L + \hbar \Delta_L I[L] = \frac{d}{d \delta} W \left (P_\epsilon^L , I + \delta \left [ Q I + \frac{1}{2} \{ I , I \}_\epsilon + \hbar \Delta_\epsilon I \right] \right ).
\]
\end{lemma}
 
\begin{lemma}
For any $\epsilon > 0$, we have
\[
\Delta_\epsilon I_{CS} =0 .
\]
\end{lemma}
 
\begin{proof}
This follows from the explicit form of $K_\epsilon$, which, up to a constant, is given by
\[
K_\epsilon = \epsilon^{-1/2} e^{\lvert x-y \rvert^2 / \epsilon } (dx \otimes 1 - 1 \otimes dy ) \ot \Cas_\fg.
\]
Each term in $\Delta_\epsilon I_{CS}$ consists of attaching an edge labeled by $K_\epsilon$ to two tails of a vertex with at least two external tails.  As there is only one vertex, the coordinates for the edge coincide, $x=y$, and this contraction of tensors results in two terms which cancel.
 \end{proof}
 
 \begin{lemma}
 The scale $L$ obstruction is given by
 \[
 O[L] = \hbar^{-1} \lim_{\epsilon \to 0} \frac{d}{d \delta} W \left ( P_\epsilon^L , I_{CS} + \delta \left [ \frac{1}{2} \{I_{CS} , I_{CS} \}_\epsilon - \frac{1}{2} \{ I_{CS} , I_{CS} \}_0 \right ] \right ) ,
 \]
 where $\delta$ is a square zero parameter of cohomological degree -1.
 \end{lemma}
 
 \begin{proof}
 Recall (section \ref{satisfiesCME}) that $I_{CS}$ satisfies the classical master equation
 \[
 Q I_{CS} = - \frac{1}{2} \{ I_{CS} , I_{CS} \}_0.
 \]
 Combining this result with the previous lemma, we see that
 \[
 QI_{CS} + \frac{1}{2} \{I_{CS} , I_{CS}\}_\epsilon + \hbar \Delta_\epsilon I_{CS} = - \frac{1}{2} \{ I_{CS} , I_{CS} \}_0 + \frac{1}{2} \{I_{CS} , I_{CS} \}_\epsilon  .
 \]
Now $I_{naive} [L] = \lim_{\epsilon \to 0} W (P^L_\epsilon , I_{CS})$ and the obstruction is defined as
 \[
O[L] = \hbar^{-1} \left( Q I_{naive}[L] + \frac{1}{2} \{ I_{naive} [L] , I_{naive} [L] \}_L + \hbar \Delta_L I_{naive}[L]  \right), 
\]
so Lemma \ref{RGEQME} completes the proof.
 \end{proof}

 The following lemma completes the proof of Proposition \ref{obsexpression}.
 
 \begin{lemma}
 \[
 \begin{array}{lll}
\lefteqn{\displaystyle \hbar^{-1} \frac{d}{d \delta} W \left ( P_\epsilon^L , I_{CS} + \delta \left [ \frac{1}{2} \{I_{CS} , I_{CS} \}_\epsilon - \frac{1}{2} \{ I_{CS} , I_{CS} \}_0 \right ] \right )} \\[2ex] &\hspace{85pt} & ={\displaystyle \sum_\gamma \sum_{ \substack{\textrm{$e$ a non-}\\ \textrm{loop edge} }} \frac{1}{\lvert \mathrm{Aut} (\gamma)\rvert} W_{\gamma, e} (P_\epsilon^L , K_\epsilon - K_0 , I_{CS} ) .}\end{array}
\]
\end{lemma}

\begin{proof}
Because $\delta^2 =0$, we know the $\delta$-weighted part of the interaction term 
\[ 
\delta \left [ \frac{1}{2} \{I_{CS} , I_{CS} \}_\epsilon - \frac{1}{2} \{I_{CS} , I_{CS} \}_0 \right ]
\]
appears on at most one vertex in any given graph in the computation of the RG flow.  Hence our strategy is to replace that vertex with two vertices, connected by an edge labelled by $K_\epsilon - K_0$.  

We have the equality $ \{ I_{CS}, I_{CS} \}_\epsilon = \Delta_\epsilon (I_{CS} I_{CS})$, where $\Delta_\epsilon (I_{CS} I_{CS})$ is a sum of terms given like that pictured below.
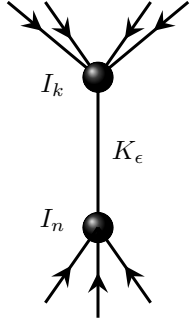
\begin{center}
\begin{tikzpicture}[decoration={markings,
   mark=at position 0.6cm with {\arrow[black,line width=.8mm]{stealth}}}];
\draw[postaction=decorate, line width=.4mm] (-1.2,1) -- (0,0);
\draw[postaction=decorate, line width=.4mm] (-0.7, 1) -- (0,0);
\draw[postaction=decorate, line width=.4mm] (0.7,1) -- (0,0);
\draw[postaction=decorate, line width=.4mm] (1.2,1) -- (0,0);
\draw[line width=.4mm] (0,0) -- (0,-2);
\draw[ball color=black]  (0,0) circle (.2);
\draw (-0.6,-0.1) node {$I_k$};
\draw (0.4,-1) node {$K_\epsilon$};
\draw[ball color=black] (0,-2) circle (.2);
\draw (-0.6, -1.9) node {$I_n$};
\draw[postaction=decorate, line width=.4mm] (-.7,-3) -- (0,-2);
\draw[postaction=decorate, line width=.4mm] (0, -3.2) -- (0,-2);
\draw[postaction=decorate, line width=.4mm] (0.7,-3) -- (0,-2);
\end{tikzpicture}
\end{center}

\noindent
The same is true for $\{I_{CS} , I_{CS} \}_0$ --- that is, $\{I_{CS} , I_{CS} \}_0 = \Delta_0 (I_{CS} I_{CS})$ --- and again we have an expansion as a sum of $I_n$ and $I_k$ connected via the distribution $K_0$. Hence by combining the respective sums we can write $\{I_{CS} , I_{CS} \}_\epsilon - \{I_{CS} , I_{CS} \}_0$ as a sum of terms of the form below.
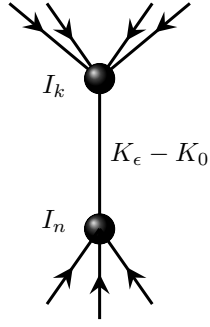
\begin{center}
\begin{tikzpicture}[decoration={markings,
   mark=at position 0.6cm with {\arrow[black,line width=.8mm]{stealth}}}];
\draw[postaction=decorate, line width=.4mm] (-1.2,1) -- (0,0);
\draw[postaction=decorate, line width=.4mm] (-0.7, 1) -- (0,0);
\draw[postaction=decorate, line width=.4mm] (0.7,1) -- (0,0);
\draw[postaction=decorate, line width=.4mm] (1.2,1) -- (0,0);
\draw[line width=.4mm] (0,0) -- (0,-2);
\draw[ball color=black]  (0,0) circle (.2);
\draw (-0.6,-0.1) node {$I_k$};
\draw (.8,-1) node {$K_\epsilon - K_0$};
\draw[ball color=black] (0,-2) circle (.2);
\draw (-0.6, -1.9) node {$I_n$};
\draw[postaction=decorate, line width=.4mm] (-.7,-3) -- (0,-2);
\draw[postaction=decorate, line width=.4mm] (0, -3.2) -- (0,-2);
\draw[postaction=decorate, line width=.4mm] (0.7,-3) -- (0,-2);
\end{tikzpicture}
\end{center}

Hence we see that 
\[
\hbar^{-1} \frac{d}{d \delta} W \left ( P_\epsilon^L , I_{CS} + \delta \left [ \frac{1}{2} \{I_{CS} , I_{CS} \}_\epsilon - \frac{1}{2} \{ I_{CS} , I_{CS} \}_0 \right ] \right )
\]
is given by summing over all graphs $\gamma$ appearing in the RG flow and replacing the $\delta$-weighted part by two vertices connected with an edge labelled by $K_\epsilon - K_0$, exactly as claimed.
 \end{proof}

As a consequence of Proposition \ref{obsexpression}, the obstruction $O = \lim_{L \to 0} O [L]$ can be written as a sum
\[
O = \sum_{\gamma, e} O_{\gamma , e} \overset{\text{def}}{=} \sum_{\gamma, e} \frac{1}{\lvert \mathrm{Aut} (\gamma) \rvert} \lim_{\epsilon \to 0} W_{\gamma, e} (P_\epsilon^1, K_\epsilon - K_0 , I_{CS}) ,
\]
where each term
\[
O_{\gamma , e} : \left ( \Omega^\ast (M) \otimes \fg \right)^{\otimes T(\gamma)} \to \CC
\]
can be decomposed as a product $O_{\gamma , e}^{an} \otimes O_{\gamma , e}^{\fg}$, where the analytic/Lie factor is a linear map on the analytic/Lie factor, respectively. This decomposition lets us eliminate certain factors by showing the analytic factor vanishes. Trees don't contribute to the obstruction (they are never singular), so the $\epsilon$-limit is zero for a tree. Likewise, any one-loop graph looks like a wheel with trees attached, so if the distinguished edge $e$ appears in one of the trees, then the $\epsilon$-limit is zero. Hence the relevant term of the obstruction becomes 
\[
O' = \sum_{n \ge 2} O_n^{an} \otimes \left( \sum_{\substack{\gamma \text{ a wheel with $n$ vertices} \\ e \in \gamma \text{ a non-loop edge}} } O_{\gamma,e}^{\fg} \right) .
\]
By `relevant' we mean that if $O'$ vanishes, then the obstruction $O$ vanishes. Here we view the analytic obstruction as a distribution that only depends on the number of vertices of any given wheel.

\begin{prop}\label{1dvanishprop}
In one-dimensional Chern-Simons theory, for each $n \ge 2$, the sum
 \[
 \sum_{\substack{\gamma \text{ a wheel with n vertices} \\ e \in \gamma \text{ an edge}} } O_{\gamma,e}^{\fg} 
 \]
is zero in $H^1(\Omega^1_{cl}(B\fg) \oplus \Omega^1_{cl}(B\fg)[1])$, the first cohomology group of the deformation obstruction complex. Consequently, in one-dimensional Chern-Simons with values in a $\L8$-algebra $\fg \oplus \fg^\vee [-2]$,  the total obstruction $O$ also vanishes. 
\end{prop}

\begin{proof}
We compute below (Lemma \ref{wheellemma}) that by summing over all wheels $\gamma$ with $n$ vertices, we have
\[
\sum_{\gamma,e} O_{\gamma,e}^{\fg} = n! (-2 \pi i)^n  ch_n (T_{B\fg}) .
\]
Now $ch_n (T_{B \fg}) \in H^{2n-1} (\Omega^1_{cl} (B\fg))$ and by Corollary \ref{obscpxcor} any obstruction lives in $H^1 (\Omega^1_{cl}(B \fg)) \oplus H^2 (\Omega^1_{cl}(B \fg))$.  Therefore, the Lie-theoretic obstruction must vanish.  The total obstruction is just some multiple of the Lie-theoretic obstruction and hence it also vanishes.
\end{proof}

\begin{prop}
In one-dimensional Chern-Simons theory, for each $n \ge 2$, the analytic obstruction $O_n^{an}$ is zero.
\end{prop}

\begin{proof}
For any wheel $\gamma$, the limit $\lim_{\epsilon \to 0} W_{\gamma, e} (P_\epsilon^1, K_\epsilon - K_0, I_{CS})$ is zero because
\[
\lim_{\epsilon \to 0} W_{\gamma, e} (P_\epsilon^1, K_\epsilon, I_{CS}) = \lim_{\epsilon \to 0} W_{\gamma, e} (P_\epsilon^1, K_0, I_{CS})
\]
as distributions, which can be shown by direct computation.
\end{proof}

Both propositions imply that our action functional satisfies the QME.

\begin{cor}\label{1dvanish}
For $\sE = \Omega^\ast (M) \otimes ( \mathfrak{g}[1] \oplus \mathfrak{g}^\vee [-1])$, the pre-theory $\{ I_{naive} [L] \} \in \sO^+ (\sE) [[\hbar]]$ (from definition \ref{naivecor}) is a BV theory. 
\end{cor}

We will denote the resulting theory by $\{I [L]\}$. As we have a one loop quantization, we write
\[
I[L] = I^{(0)} [L] + \hbar I^{(1)} [L],
\]
where the superscript records how many loops appear in the Feynman diagrams of the RG flow from $I[0] = I_{CS}$ to  $I[L]$. 

\section{The Atiyah class and Koszul duality}

In order to provide an elegant presentation of theorem \ref{inftytermBg}, we need to develop a bit of machinery known as the Atiyah class. Its primary role for us is to construct a kind of characteristic class, a process which we take up in the next section. We will elaborate on how these constructions appear in the geometry of smooth manifolds in Part II of this paper, where we extract the usual Chern classes by methods different than the usual approaches with Atiyah classes.

\subsection{The definition}

Let $R = (R^\#, d)$ be a commutative dg algebra over a base ring $k$. The underlying graded algebra is denoted $R^\#$. We denote the K\"ahler differentials of $R$ by $\Omega^1_R$ and let $d_{dR}: R \rightarrow \Omega^1_R$ denote the universal derivation.

\begin{definition}
Let $M$ be an $R$-module that is projective over $R^\#$. A {\em connection} on $M$ is a $k$-linear map $\nabla: M \rightarrow M \ot_R \Omega^1_R$ such that
\[
\nabla(r \cdot m) = (d_{dR} r) m + (-1)^{|r|} r \nabla m, 
\]
for all $r \in R$ and $m \in M$.
\end{definition}

A connection may not be compatible  with the differential $d_M$ on $M$, and the Atiyah class is precisely the obstruction to compatibility between $\nabla$ and the dg $R$-module structure on $M$. 

\begin{definition}
The {\em Atiyah class} of $\nabla$ is the class in $\Omega^1_R \ot_R \End_R(M)$ given by
\[
\At(\nabla) = [\nabla, d] = \nabla \circ d_M - d_{\Omega^1_R \ot_R M} \circ \nabla.
\]
\end{definition}

This definition is quite abstract as stated, but it appears naturally in many contexts, notably in work by Kapranov \cite{Kap}, Markarian \cite{Mark}, Calaque-van den Bergh \cite{CVDB}, Caldararu \cite{Cald1} \cite{Cald2}, Ramadoss \cite{Ram} and Chen-Sti\'enon-Xu \cite{CSX}.

We now explain how the Atiyah class appears in the context of the Koszul duality between Lie and commutative algebras. Eventually we will apply this formalism to give an alternative approach to constructing characteristic classes of vector bundles.

\begin{remark}
Atiyah \cite{At} originally introduced this construction to measure the obstruction to obtaining a holomorphic connection on a holomorphic bundle over a complex manifold. Let $X$ be a complex manifold, $\pi: E \rightarrow X$ a holomorphic vector bundle,  $\Omega^{0,\ast}(X)$ the Dolbeault complex of $X$, and $(\Omega^{0,\ast} (E), \bar{\partial})$ the Dolbeault complex of the bundle. Let 
$$\nabla: \Omega^{0,\ast} (E) \rightarrow \Omega^{1,\ast} (X) \otimes_{\Omega^{0,\ast} (X)} \Omega^{0,\ast} (E)$$ 
be a $\CC$-linear map satisfying
\[
\nabla(f s) = (\partial f) s + f \nabla s
\]
for all $f \in \Omega^{0,\ast} (X)$ and $s \in \Omega^{0,\ast} (E)$ (really it is enough to consider $\nabla$ on $\Omega^{0,0} (E)$). The usual Atiyah class is $[\nabla, \bar{\partial}] \in \Omega^{1,1}( \End(E))$. Notice that if this Atiyah class vanishes, then $\nabla$ is clearly a holomorphic connection. On a compact K\"ahler manifold, Atiyah showed that traces of powers of the usual Atiyah class give the Chern classes of $E$.
\end{remark}

\subsection{Koszul duality and the Atiyah class}\label{atiyahnotation}

In the setting of $\L8$-algebras, one sometimes takes the Chevalley-Eilenberg cochain complex {\em as} the definition of the $\L8$ structure, so it should be no surprise that there is a natural way to strip off the Taylor components from the Chevalley-Eilenberg complex. What we'll show in this section is that
\begin{enumerate} 
\item[(1)] the tangent bundle to $B\fg$ has a natural connection, and 
\item[(2)] by taking derivatives of the Atiyah class for this connection, we recover the brackets $\ell_n$ of the $\L8$-algebra $\fg$. 
\end{enumerate}
This result is interesting from the point of view of deformation theory and Koszul duality: it explains how the Atiyah class fits into this process. Namely, given an augmented commutative dg algebra $\sA$, the Koszul dual $\L8$-algebra $\fg_\sA$ should be the shifted tangent complex $T_{\sA}[-1]$, at the point given by the augmentation, equipped with brackets by taking the Taylor terms of some Atiyah class. (For a modern, careful treatment, we direct the reader to the discussion around proposition 2.47 of \cite{Francis},  where Francis gives broad generalizations of this relationship and connections with several other mathematical themes in this paper. These ideas have a long history, of course, starting at least with Grothendieck and Quillen. The relationship with the Atiyah class already appears in the work of Illusie \cite{Illusie} on the cotangent complex and Schlessinger-Stasheff \cite{SchlessingerStasheff} on tangent cohomology.) 

We will work with an arbitrary $\fg$-module $M$ as it simplifies the formulas to distinguish between $M$ and $\fg$ (for the tangent bundle, $M$ is another copy of $\fg$, which can be distracting). Consider the sections $\cM$ of this module as a sheaf over $B\fg$: it is the $C^*(\fg)$-module $C^*(\fg, M)$. Forgetting the differentials, we see there is a natural trivialization
\[
C^\#(\fg,M) \cong C^\#(\fg) \ot_k M,
\]
as a $C^\#(\fg)$-module. This trivialization equips $\cM$ with a connection
\[
C^\#(\fg) \ot_k M \to \Omega^1_{B\fg} \ot_k M,
\]
\[
f \ot m \mapsto (d_{dR} f) \ot m.
\]
Define $\At(\cM)$ to be the Atiyah class for this connection. 

The Atiyah class lives in $\Omega^1_{B\fg}(\End \, \cM) \cong C^*(\fg, \fg^\vee[-1] \ot \End_k(M) )$. We can thus view it as a map
\[
\At(\cM): T_{B\fg} \ot \cM \to \cM
\]
and ask for the Taylor coefficients as a section of $B\fg$.

\begin{prop}\label{l8atiyah}
Given $x \in \fg$, we obtain a vector field $X$ on $B\fg$,by shifting the degree of $x$. Let $m$ be a section in $\cM$. We find
\[
\At(\cM)(X \ot m) = \ell_2(x,m) + \ell_3(x,x,m) + \cdots + \ell_n(x^{\ot n-1},m) + \cdots.
\]
Alternatively, we say that for $X$ a vector field, $m \in \cM$, and  $x_1, \ldots, x_n, y \in \fg$,
\[
\frac{\partial}{\partial x_1} \cdots \frac{\partial}{\partial x_n} \Bigg|_0 \At(\cM)(X \ot m) = \ell_{n+2}(x_1, \ldots, x_n, x, m) \in \fg,
\]
where $x \in \fg$ is the shift of $x$.
\end{prop}

\begin{remark}
We know $\ell_1$ via the differential on the tangent complex. This proposition tells us how to recover $\ell_2$, $\ell_3$, and so on, but it does not return $\ell_0$. 
\end{remark}

\begin{proof}
We pin down some useful notation that makes the proof straightforward.

By definition, $\sO_{B\fg}$ is the algebra $C^\#(\fg) = \csym(\fg^\vee [-1])$ equipped with a degree 1 derivation\footnote{There are so many $d$'s floating around that we switch notation as an aid to clarity.} $\partial$ whose homogeneous components
\[
\partial_n: \fg^\vee [-1] \to \Sym^n \left( \fg^\vee [-1] \right)
\]
are dual to the $n$-fold brackets $\ell_n$.

The K\"ahler differentials $\Omega^1_{B\fg}$ are thus $\csym(\fg^\vee [-1]) \ot_k \left( \fg^\vee [-1]\right)$ with differential 
\[
d_{dR}: f \ot x \mapsto \partial f \ot x + (-1)^{|f|} f \cdot d_{dR} (\partial x), 
\]
where $f \in \sO_{B\fg}$ and $x \in \fg^\vee[-1]$. Here $d_{dR}: \sO_{B\fg} \to \Omega^1_{B\fg}$ denotes the universal derivation
\[
d_{dR}: x \mapsto 1 \ot x
\]
for $x \in \fg^\vee[-1]$. Note that $d_{dR} \circ d_{dR} = d_{dR} \circ \partial$. From hereon, we will denote $1 \ot x$ by $dx$ and $f \ot x$ by $f \, dx$.

We now need to describe the dg module of sections $\cM$. The underlying module is
\[
\csym(\fg^\vee [-1]) \ot M
\]
and the differential has the form $d = \partial \ot 1_M + d_M$, where $d_M$ encodes the action of $\fg$ on $M$. Fixing a basis $\{x^j\}$ for $\fg$ and the dual basis $\{x_j\}$ for $\fg^\vee$, we can express $d_M$ as
\begin{multline*}
d_M(f \ot m) = \sum_j x_j \cdot f \ot \ell_2(x^j, m) + \sum_{j_1, j_2} x_{j_1} x_{j_2} \cdot f \ot \ell_3(x^{j_1}, x^{j_2}, m) + \cdots \\
\cdots + \sum_{j_1, \ldots, j_n} (x_{j_1} \cdots x_{j_n}) \cdot f \ot \ell_{n+1}(x^{j_1}, \ldots, x^{j_n}, m) + \cdots
\end{multline*}

Similarly, $\Omega^1_{B\fg} \ot_{\sO_{B\fg}} \cM$ consists of
\[
\csym(\fg^\vee [-1]) \ot \fg^\vee[-1] \ot M
\]
with differential $d = d_{dR} \ot 1_{M} + d_M \ot 1_{\fg^\vee[-1]}$.\footnote{This notation is meant to indicate that we use the differential for 1-forms without changing the section of $\cM$ and then we use the differential for the section without changing the 1-form.} For instance,
\[
d(dx \ot m) = d_{dR} \partial x \ot m + \sum_n \sum_{j_1, \ldots, j_n} (x_{j_1} \cdots x_{j_n}) dx \ot \ell_{n+1}(x^{j_1}, \ldots, x^{j_n}, m).
\]

With these definitions in hand, we see that the Atiyah class for $\cM$ is
\begin{align*}
[\nabla, d] &= (d_{dR} \ot 1_{M}) \circ (\partial \ot 1_{M} + d_{M}) - (d_{dR} \ot 1_{M} + d_{M} \ot 1_{\fg^\vee[-1]}) \circ (d_{dR} \ot 1_{M}) \\
 &= \left( d_{dR} \circ \partial - d_{dR} \circ d_{dR} \right) \ot 1_{M} + (d_{dR} \ot 1_{M}) \circ d_{M} - (d_{M} \ot 1_{\fg^\vee[-1]}) \circ (d_{dR} \ot 1_{M}) \\
  &= [d_{dR}, d_{M}],
 \end{align*}
 using the compatibility of $\partial$, $d_{dR}$, and $d_{dR}$.

Observe that the Atiyah class sends $1 \ot m$ to 
\[
d_{dR} \left( \sum_n \sum_{j_1, \ldots, j_n} (x_{j_1} \cdots x_{j_n}) \ot \ell_{n+1}(x^{j_1}, \ldots, x^{j_n}, m) \right). 
\]
If we ``evaluate this sum at zero," this means we take the constant term of the expression above, which is
\[
\sum_j dx_j \ot \ell_2(x^j, m).
\]
This element sends a vector field $X$, given by the shift of an element $x \in \fg$, to $\ell_2(x,m)$. 


Taking higher derivatives of this expression and evaluating at zero recovers all the data of the brackets $\ell_n$.
\end{proof}

\subsection{Useful facts about the Atiyah class}

We establish here several facts that we will find useful later. We now fix notation that we use throughout this section. 

Denote the differential on $R$ by $d_R$. Let $M$ be a free $R^\#$-module\footnote{Our results imply the relevant results for finitely generated projective $R^\#$ modules.} and fix a basis so that the differential $d_M$ has the form $d_R + A$, where $A \in \Hom^1_R(M,M)$ and $d_R A + A^2 = 0$. Let $\nabla$ denote a connection on $M$, which has the form $d_{dR} + B$ with respect to the basis on $M$, where $B \in \Hom^0(M, \Omega^1 \ot_R M)$. We denote the differential on $\Omega^1_R$ by $d_{\Omega^1}$, and hence the differential on $\Omega^1_R \ot_R M$ is $d_{\Omega^1} + A$.

\begin{lemma}\label{atiyah}
The Atiyah class $\At(\nabla)$, with respect to the basis we've fixed on $M$, has the form
\[
d_{dR} A - d_{\Omega^1} B - [A,B].
\]
Alternatively, we express it as
\[
d_{dR} A - d_{\Omega^1 \ot \End M} B.
\]
\end{lemma}

\begin{proof}
This is a straightforward computation.
\begin{eqnarray*}
\At(\nabla) & = & (d_{dR} + B)(d_R + A) - (d_{\Omega^1} + A)(d_{dR} + B) \\
 & = & (d_{dR} \circ d_R + d_{dR} \circ A + B \circ d_R + BA) - (d_{\Omega^1} \circ d_{dR} + d_{\Omega^1} \circ B + A \circ d_{dR} + AB) \\
 & =  & [d_{dR}, d] + d_{dR}(A) - d_{\Omega^1}(B) + [B,A] \\
 & = & d_{dR} A - d_{\Omega^1_R \ot_R \End(M)}(B).
\end{eqnarray*}
Here $d$ denotes the differential either on $R$ or on $\Omega^1_R$, and the commutator $[d_{dR}, d]$ vanishes by construction.
\end{proof}

\begin{cor}
The Atiyah class is closed: $d_{\Omega^1 \ot \End M} \At(\nabla) = 0$.
\end{cor}

\begin{proof}
Recall $d_R A + A^2 = 0$. We compute
\begin{eqnarray*}
d_{\Omega^1 \ot \End M} \At(\nabla) & = & d_{\Omega^1 \ot \End M} d_{dR} A - d_{\Omega^1 \ot \End M}^2 B \\
 & = & d_{\Omega^1} d_{dR} A + [A, d_{dR} A] \\
 & = & d_{dR} d_R A +  [A, d_{dR} A] \\
 & = & d_{dR} \left( -A^2 \right) +  [A, d_{dR} A]\\
 & = & - (d_{dR}A) A + A (d_{dR} A) + [A, d_{dR} A] \\
 & = & 0,
\end{eqnarray*}
as $d_{dR}$ satisfies the Leibniz rule.
\end{proof}

In analogy with geometry, the de Rham differential $d_{dR}$ extends to a complex $\Omega^*_R$, with exterior derivative $d_{dR}: \Omega^k_R \to \Omega^{k+1}_R$ such that $[d,d_{dR}] = 0$, where $d$ denotes the differential on the $R$-modules $\Omega^k_R$. 

We are now led to the following question: what if $\nabla$ equips $M$ with a {\em flat} connection, so that $\nabla^2 = 0$? In that case, $\nabla$ makes $\Omega^\ast_R \ot_R M$ a cochain complex over $R^\#$, the underlying graded algebra. Hence the Atiyah class is the obstruction to making $M$ a ``vector bundle with flat connection" over the space described by $R$. This situation is precisely what appears in our jet-bundle approach to the Chern-Weil construction of characteristic classes in section \ref{charclass}. The Atiyah class will play the same role that the curvature usually does because it will be precisely the obstruction to making the connection flat.

In this situation, we have a natural analogue of the Bianchi identity. Recall that a connection $\nabla$ on $M$ induces a connection $\nabla^{\End}$ on $\End M$. If $\nabla = d_{dR} + B$ in our basis, then $\nabla^{\End} = d_{dR} + [B, -]$.

\begin{prop}\label{flatness}
If $\nabla^2 = 0$, then $\At(\nabla)$ is a horizontal section of $\Omega^\ast_R \ot_R \End M$. More explicitly, 
\[
\nabla^{\End} \At(\nabla) = 0.
\]
\end{prop}

\begin{proof}
We compute
\begin{eqnarray*}
\nabla^{\End} \At(\nabla) & = & d_{dR} \At(\nabla) + [B, \At(\nabla)]\\
 & = & d_{dR}^2 A - d_{dR} d_{\Omega^1 \ot \End M} B + [B, d_{dR} A] - [B, d_{\Omega^1 \ot \End M} B].
\end{eqnarray*}
Now we need some useful cancellations. Clearly, $d_{dR}^2 A = 0$. 

Next, observe that
\begin{eqnarray*}
d_{dR} d_{\Omega^1 \ot \End M} B & = & d_{dR} d_{\Omega^1} B + d_{dR} [A,B]\\
 & = & d_{\Omega^2} d_{dR} B + [d_{dR} A, B] - [A, d_{dR} B],
\end{eqnarray*}
and since $\nabla^2 = d_{dR} B + B^2 = 0$, we continue
\begin{eqnarray*}
 & = & d_{\Omega^2} (-B^2) + [d_{dR} A, B] - [A, d_{dR} B]\\
 & = & -[B, d_{\Omega^1} B] + [d_{dR} A, B] - [A, d_{dR} B].
\end{eqnarray*}

Another computation shows
\begin{eqnarray*}
[B, d_{\Omega^1 \ot \End M} B] & = & [B, d_{\Omega^1} B] + [B, [A,B]]\\
 & = & [B, d_{\Omega^1} B] + \frac{1}{2}[A, [B,B]]
\end{eqnarray*}
by the Jacobi identity.

Putting these computations together, we find
\begin{eqnarray*}
\nabla^{\End} \At(\nabla) & = & d_{dR}^2 A - d_{dR} d_{\Omega^1 \ot \End M} B + [B, d_{dR} A] - [B, d_{\Omega^1 \ot \End M} B]\\
 & = & \left([B, d_{\Omega^1} B] - [d_{dR} A, B] + [A, d_{dR} B] \right) + [B, d_{dR} A] - \left( [B, d_{\Omega^1} B] + [B, [A,B]] \right) \\
 & = & [A, d_{dR} B] - \frac{1}{2} [A,[B,B]] \\
 & = & 0,
\end{eqnarray*}
as $\nabla^2 = 0$.
\end{proof}

\section{Characters in geometry and Lie theory}\label{cherndefn}

In representation theory, the character of a representation is one of the most useful invariants; in geometry, the Chern character of a bundle is likewise one of the most useful invariants. In this section, we want to exhibit how Koszul duality provides an approach to characters that includes both of these cases.

\begin{definition}
The {\em Chern character} of a connection $\nabla$ is $ch(\nabla) := \Tr \exp \left( \frac{\At(\nabla)}{-2\pi i} \right)$.
\end{definition}

We let $ch_k(\nabla)$ denote the homogeneous component of $ch(\nabla)$ in $\Omega^k_R$. Hence $ch_k(\nabla) = \Tr \left(  \frac{1}{k!(-2\pi i)^k} \At(\nabla)^k \right)$.

As stated, the Chern character is an element in $\Omega^\ast_R$ of mixed degree, but it is more natural (as we explain below) to make it homogeneous by forcing $\At(\nabla)$ to be homogeneous as follows. Observe that $\At(\nabla)$ lives in $\Omega^1_R \ot_R \End M$, and it has degree 1. We can identify it with a degree 0 element if we instead view it as living in $\Omega^1_R \ot_R \End M[1]$. In that case, the powers $\At(\nabla)^k$ live in $\Omega^k_R \ot_R \End M [k]$ and have degree 0. The Chern character $ch(\nabla)$ is then a homogeneous, degree 0 element of $\oplus_k \Omega^k_R[k]$, which we will denote as $\Omega^{-\ast}_R$. From the perspective of derived geometry, this setting is more natural since we only access homogeneous elements when we work functorially (cf. Bernstein's discussion of the ``even rules" principle in \cite{IAS}).

There is another conceptual reason to work with the algebra $\Omega^{-\ast}_R$, as explained by To\"en-Vezzosi \cite{TV} and Ben-Zvi--Nadler \cite{BZN}. This algebra is the structure sheaf of the \emph{derived loop space} $\cL X$ for a derived scheme $X$. Given a loop $\gamma$ in $X$, we can pull back the bundle $M$ on $X$ to a bundle on $S^1_{dR}$, which is locally constant by construction. The monodromy around $\gamma$ defines an endomorphism of $M$ and the trace of this monodromy is the value of $ch(M)$ at this point $\gamma \in \cL X$. In fact, this function $ch(M)$ is equivariant under rotation under loops and hence lives in $\sO(\cL X)^{S^1}$, which can be identified with the even de Rham cohomology of $X$.

In this paper, we are giving a construction of $ch(M)$ in the style of Chern-Weil (i.e., via connections), and there is a condition for $ch(\nabla)$ to agree with $ch(M)$. In essence, this condition is that the parallel translation via our connection is locally constant along a loop. This condition is obstructed for a generic choice of connection, since the Chern classes $ch_k(\nabla)$ are always closed under $d_{\Omega^k}$ but not always closed under $d_{dR}$. The following result follows directly from our work in the previous section.

\begin{cor}\label{chernclosed}
If $\nabla^2 = 0$, then the Chern classes $ch_k(\nabla)$ are closed under $d_{dR}$ and $d_{\Omega^k}$.
\end{cor}

\begin{proof}
Both of these follow straightforwardly from our work in the preceding section, and we use the same notation as above. All the work here is about understanding what happens when we pull an operator like $d_{dR}$ or $d_{\Omega^k}$ past trace. The arguments are completely analogous, so we only give one.

Observe that once we fix a basis for $M$, we have a natural way to write endomorphisms as matrices and thus we can define $d_{dR} X$ for $X \in \Omega^k \ot \End M$. In consequence, we see $d_{dR} \Tr X = \Tr d_{dR} X$. Thus we find
\[
\Tr \nabla^{\End} X = \Tr d_{dR} X + \Tr [B,X] = \Tr d_{dR} X = d_{dR} \Tr X.
\]
Here $\nabla^{\End} = d_{dR} + [B, -]$ denotes the induced connection on $\End M$. We also use the fact that trace vanishes on commutators.

We thus find 
\[
d_{dR} \Tr  \At(\nabla)^k = \Tr \nabla^{\End} \At(\nabla)^k = 0,
\]
since $\nabla^{\End}$ satisfies the Leibniz rule and $\nabla^{\End}  \At(\nabla) = 0$ by proposition \ref{flatness}. 
\end{proof}

\subsection{The character in Lie theory}

Although it is not necessary for the rest of the paper, the reader might find it helpful to understand how this notion of character appears in Lie theory. Let $\fg$ be a Lie algebra in the usual sense, such as $\mathfrak{sl}_2$. In that case, the derived loop space of $B\fg$ is the derived adjoint quotient $\fg/\fg$, whose ring of functions is $C^*(\fg, \csym \fg^\vee)$, where $\fg$ acts on $\csym \fg^\vee$ by the coadjoint action and its symmetric powers. The degree 0 cohomology is then $\left( \csym \fg^\vee \right)^\fg$, and (at least when $\fg$ is semisimple) this ring is isomorphic to $\left( \csym \mathfrak{h}^\vee \right)^W$, where $\mathfrak{h}$ is a Cartan and $W$ is the Weyl group. This ring is isomorphic to the (ungraded and completed) cohomology of the classifying space $BG$, and hence is a natural target of the equivariant Chern character. More explicitly, the equivariant K-theory $K_G(\pt)$ is the representation ring of $G$, and we compose the Atiyah-Segal completion map with the Chern character 
\[
K_G(\pt) \ot \QQ \to K(BG) \ot \QQ \to \hat{H}^*(BG, \QQ)
\] 
to define a character for each $G$-representation.

For any $\fg$-module $M$, there is a canonical connection $\nabla_M$ induced by the natural $C^\#(\fg)$ splitting $C^\#(\fg, M) \cong C^\#(\fg) \ot M$. This Chern character $ch(\nabla_M)$ agrees with the Chern character applied to $M$ as a bundle over $BG$, under the isomorphism described above.

\section{The global observables}\label{goingtoinfinity}

Our goal in this section is to provide a conceptual, geometric interpretation of the BV quantization we have constructed for Chern-Simons, and we use the language of observables, which we now discuss, to provide that interpretation. 

\subsection{Reminder on observables}\label{observables}

Since studying a classical field theory amounts to studying the space of solutions to some system of PDE, all the information of the field theory is encoded in the commutative algebra of functions on the space of solutions. Any imaginable measurement of the physical system described by the theory yields an element of this algebra: to some field, it returns the value of the measurement on that field. Thus we call this algebra the {\em classical observables} for the theory. In our case, classical Chern-Simons describes the derived loop space $\cL T^*B\fg$, and the classical observables are the commutative dg algebra
\[
Obs^{cl} := \left( \csym \left( (\Omega^*(S^1) \ot (\fg[1] \oplus \fg^\vee[-1]))^\vee \right),\, Q + \{I_{CS}, - \} \right).
\]
Note that our observables are the power series constructed out of the distributions dual to our fields $\Omega^*(S^1) \ot (\fg[1] \oplus \fg^\vee[-1])$; this has the flavor of formal algebraic geometry, since we only want to study fields that are infinitesimally close to the zero field (which is a solution to our Euler-Lagrange equations). The differential $Q + \{I_{CS}, - \}$ encodes the Maurer-Cartan equation for $\fg \oplus \fg^\vee[-2]$-connections.

The BV quantization leads to a deformation of this algebra into just a cochain complex.\footnote{More accurately, we construct a Beilinson-Drinfeld algebra, which interpolates between a commutative algebra and a plain complex. See \cite{Cos2} and \cite{CG} for discussion of this notion.} In particular, at scale $L$, we have the following cochain complex
\[
Obs^q_L := \left( \csym \left( (\Omega^*(S^1) \ot (\fg[1] \oplus \fg^\vee[-1]))^\vee \right) [[\hbar]],\, Q + \{I_{naive}[L], - \} + \hbar \Delta_L \right).
\]
We call this complex the {\it quantum observables at scale $L$}. The RG flow $W(P_{\ell}^L), -)$ defines a quasi-isomorphism between $Obs^q_\ell$ and $Obs^q_L$. In our case, since our base manifold is closed, we can consider the scale $\infty$ observables and this complex is quasi-isomorphic to the observables at finite scale.\footnote{One should view running the RG flow from scale 0 to scale $\infty$ as taking the full path integral or, more precisely, as integrating out all the fields spanned by the nonzero modes of the free theory.} 

Something striking happens at scale $\infty$: there is a much smaller cochain complex which is homotopy equivalent to $Obs^q_\infty$. We construct it as follows. Due to our choice of a gauge-fixing operator $Q^*$, we get a decomposition of the fields $\sE$ into eigenspaces $\sE_\lambda$ of the operator $D = [Q,Q^*]$, where the eigenvalues $\{\lambda\}$ form a discrete subset of the nonnegative reals (here we use the fact that our base manifold is closed). We call $\cH = \sE_0$ the {\em harmonic fields}, in analogy with Hodge theory. Notice that the operator $\lim_{t \to \infty} e^{-tD}$ is simply projection on $\cH$, since all nonzero eigenfunctions are damped to zero, and thus its kernel $K_\infty$ lives in $\cH \ot \cH$. Hence both the bracket $\{ -,-\}_\infty$ and the BV Laplacian $\Delta_\infty$ vanish except on functions that depend on the harmonic fields. Let $\{-,-\}_\cH$ and $\Delta_\cH$ denote these operators restricted to $\csym (\cH^\vee)[[\hbar]]$. Then we have a homotopy equivalence
\[
\left( \csym (\cH^\vee)[[\hbar]], \{I[\infty], -\}_\cH + \hbar \Delta_\cH \right) \overset{\simeq}{\hookrightarrow} Obs^q_\infty.
\]
We can make the left-hand side more explicit and easier to interpret.

It remains to understand the differential on $\csym (\cH^\vee)[[\hbar]]$. As a first step, we show the following lemma, where we ignore the $\hbar$-dependent terms.

\begin{lemma}\label{HKR}
There is an isomorphism of cochain complexes
\[
\left(\csym (\cH^\vee), \{I^{(0)}[\infty], -\}_\cH \right) \cong \Omega^{-\ast} (T^*B \fg),
\]
as $C^* (\fg \oplus \fg^\vee[-2])$ modules. In other words, the global classical observables are the (negatively-graded) de Rham forms on $T^* B\fg$.
\end{lemma}

\begin{proof}
Fix an isometry class of metric on $S^1$, i.e., fix the length of the circle (say a length $\ell$) and think of $S^1$ as $\RR /\ell \ZZ$ with volume form $d\theta = (1/\ell)dx$, with $x$ the coordinate on $\RR$.  Then by definition,
\[
\cH = \CC [d \theta] \otimes \left( \fg[1] \oplus \fg^\vee[-1] \right),
\]
where $d \theta$ is viewed as a square zero, cohomological degree $+1$ element. Hence
\[
\csym (\cH^\vee) = \csym ( (\CC [d \theta] \ot \fg[1])^\vee )  \ot \csym ( (\CC [d \theta] \ot \fg^\vee[-1])^\vee).
\]
But this complex has a description in the language of dg manifolds.

Observe that $\CC [d \theta] \ot \fg[1]$ is simply $\fg[1] \oplus \fg$, so
\[
\csym ( (\CC [d \theta] \ot \fg[1])^\vee ) = \csym(\fg^\vee[-1]) \ot \csym(\fg^\vee).
\]
Now observe that 
\[
\Omega^1 (B \fg) = \Gamma (B \fg , T^\ast_{B \fg}) \cong C^\sharp(\fg) \otimes \fg^\vee [-1] ,
\] 
as $C^\sharp (\fg)$-modules, and more generally, 
\[
\Omega^k (B \fg) \cong C^\sharp (\fg) \otimes \text{Sym}^k ( \fg^\vee) [-k] .
\]
Thus we have an equivalence of $C^\sharp (\fg)$-modules
\[
\Omega^{-\ast} (B \fg) \cong C^\sharp (\fg) \otimes \csym(\fg^\vee) = \csym ( (\CC[d\theta] \ot \fg[1] )^\vee ).
\]
Extending this argument by viewing $\fg \oplus \fg^\vee[-2]$ as an $\L8$ algebra, we see that 
\[
\csym (\cH^\vee) \cong \Omega^{-\ast} (T^*B \fg)
\]
as $C^\sharp (\fg \oplus \fg^\vee[-2])$ modules.

The interaction term $I^{(0)}[\infty]$ is given by summing over all trees where the vertices are labelled by $I_{CS}$ and the edges are labelled by the propagator $P(0,\infty)$. Since $P(0,\infty) = Q^* K_\infty$, we see that $P(0,\infty) = 0$ because $Q^*$ acts by zero on $\cH$. Any tree with an internal edge must have then weight zero. Thus $I^{(0)}[\infty]\big|_\cH$ is simply $I_{CS}$. Thus we recover the usual Chevalley-Eilenberg complexes.
\end{proof}

\begin{remark}
Recall that the one-loop term $I^{(1)}$ of the quantized action only depends on $\Omega^*(S^1) \ot \fg[1]$ and vanishes on terms depending on $\fg^\vee[-1]$. Hence, by our work above, we see that $I^{(1)}[\infty]$ restricts to a function on $\CC[d\theta] \ot \fg[1]$. In other words, $I^{(1)}[\infty]$ lives in $\Omega^{-\ast}(B\fg)$.
\end{remark}

\subsection{A crucial property of the one-loop Feynman diagrams}

Our goal now is to understand the one-loop interaction term $I^{(1)}[\infty]$ restricted to $\cH$, and here we will see why the Atiyah class is so useful. First, we will show that wheels are the only connected graphs that contribute, and, second, we will show how to organize the sum over wheels into something conceptually meaningful.

\begin{lemma}
For a connected one-loop graph $\gamma$ that is not a wheel (i.e., a graph which consists of a wheel with trees attached), the graph weight $W_\gamma(P(0,\infty), I)$ is zero on $\cH$.
\end{lemma}

\begin{proof}
As shown in the preceding lemma about trees, the propagator $P(0,\infty)$ is zero on fields in $\cH$. Hence any tree with an internal edge will vanish on a harmonic field. Plugging zero into a wheel yields zero as well.
\end{proof}

It remains to understand the graph weight of wheels. As discussed earlier, the weight of every graph decomposes into a product of an analytic and a Lie-theoretic part: $W_\gamma = W^{an}_\gamma \cdot W^{Lie}_\gamma$. We will analyze these two aspects separately. It is crucial to bear in mind that the weight of a wheel can be viewed as the trace of the operator described by the propagator that labels the internal edges. 

The analytic part is particularly easy to understand. In the analytic part of the interaction term $I_{CS}$, the homogeneous degree $k$ component $(I_{CS})_k$ simply consists of wedging $k$ differential forms on $S^1$ and then integrating over $S^1$. Now that we've described the vertices in a wheel, it remains to describe the internal edges. 

Fix a length $\ell$ and identify $S^1$ with $\RR/\ell \ZZ$. Fix the volume form $dz= (1/2 \pi i \ell) dx$, where $x$ denotes the coordinate on $\RR$.\footnote{We make the possibly peculiar-looking choice to give the circle the ``volume" $1/2 \pi i$ because it makes the end result in Theorem \ref{inftytermBg} look the nicest.} Let $d^*$ denote the operator $f \, dx \mapsto df/dx$ on $\RR$ and descend it to $S^1$; then $D = d^2/dx^2$, the usual Laplacian on $\RR$, and again descend it to $S^1$.We let $D^{-1}$ be zero on harmonic fields and the inverse to $D$ on the orthogonal complement.

\begin{lemma}
For any wheel with one internal vertex, the analytic graph weight vanishes. For a wheel $\gamma$ with $n > 1$ internal vertices and $N$ external legs, the analytic graph weight $W^{an}_\gamma$ on a harmonic field $\alpha$ is $\Tr \left(  ( \frac{1}{2 \pi i \ell} \frac{d}{d x} D^{-1})^n \right) \int_{S^1} (\alpha^N)$. Further,
\[
\Tr \left(  \left(\frac{1}{2 \pi i\ell} \frac{d}{dx} D^{-1} \right)^n \right) = \frac{1}{(2 \pi )^{2n}} \sum_{k \in \ZZ,\,  k \neq 0}  \frac{1}{k^n}.
\]
Thus, the trace is $\frac{2}{(2 \pi )^{2n}} \zeta(n)$, where $\zeta$ denotes the Riemann zeta function.
\end{lemma}

Note that for $n$ odd, the analytic weight vanishes as $\sum_{k \in \ZZ \setminus \{0\}} 1/k^n =0$ for $n$ odd.

\begin{proof}
The proof amounts to unpacking our definitions and then using Fourier series. We need to show that the propagator $P(0,\infty)$ corresponds to the operator $\frac{1}{\ell} \frac{d}{d x} D^{-1}$ and that the graph weight separates the input harmonic form term from the trace term.

The propagator $P(0,\infty)$ is an element of $\Omega^*_{S^1} \ot \Omega^*_{S^1}$ of cohomological degree 0, constructed as follows. Let $K_t^{scalar}$ denote the heat kernel corresponding to the operator $e^{-tD}$:
\[
\left( e^{-tD}f \right) (x) = \int_{S^1} K^{scalar}_t(x, y) f(y) \, dy
\]
for a function $f$ on $S^1$. We have an isomorphism
\[
\Phi: C^\infty(S^1) \ot \CC[dz] \to \Omega^*_{S^1},
\]
\[
f \, dz \mapsto \frac{1}{2 \pi i \ell} f\, dx.
\]
The kernel $K_t$ is then
\[
K^{scalar}_t \left(  dz \ot 1 - 1 \ot dz \right) 
\]  
and so
\[
P(0,\infty) = d^* \int_0^\infty K_t \, dt = \frac{1}{2 \pi i \ell}  \frac{d}{d x} \int_0^\infty K_t^{scalar} \, dt.
\]
This kernel defines an operator that we denote $\frac{1}{2 \pi i \ell} \frac{d}{d x} D^{-1}$, as the integral $\int_0^\infty K^{scalar}_t dt$ corresponds to the inverse of $D$ on functions that are not harmonic. By construction, this operator is zero on harmonic functions. More explicitly, the functions $\{ e^{2 \pi ikx/\ell} \}$, with $k \in \ZZ$, are the eigenfunctions of $D$, so
\[
\frac{1}{2 \pi i\ell} \frac{d}{d x} D^{-1} (e^{2 \pi  i k x/ \ell} )= \frac{1}{(2 \pi )^2 k} e^{ikx/\ell}
\]
for $k \neq 0$.

Note that since $\frac{1}{2 \pi i\ell} \frac{d}{d x} D^{-1}$ vanishes on harmonic fields, we know that
\[
\frac{1}{2 \pi i\ell} \frac{d}{d x} D^{-1} (fg) = f \cdot \frac{1}{2 \pi i \ell} \frac{d}{dx} D^{-1}(g)
\]
for $f$ harmonic. Hence, the behavior of the internal edges is independent of the inputs to the external legs. The contribution of the internal edges to the overall weight is computed by orienting the vertices cyclically, viewing the propagator as the operator $\frac{1}{2 \pi i\ell} \frac{d}{dx} D^{-1}$, and taking the trace of the $n$-fold composition of this operator.

The case $n = 1$ is slightly different as $\zeta(1)$ is divergent. However, it is straightforward to show that if we use the propagator $P_\epsilon^L$ for any $L > \epsilon > 0$, the graph weight is zero, so the limit as $\epsilon \to 0$ and $L \to \infty$  is also zero.
\end{proof}

We now consider the Lie-theoretic weight of a wheel. In this case, a vertex with $k+1$ legs corresponds to the bracket $\ell_k$, and so we have no easy simplification analogous to that for the analytic weight. However, recall that the Atiyah class of $T_{B\fg}$ is the endomorphism of $\fg$ given essentially by summing over all $\ell_k$ for $k \geq 2$. Hence we obtain the following useful lemma.

\begin{lemma}\label{wheellemma}
Let $\Cas_\fg$ denote the Lie-theoretic part of the propagator (see \ref{propagator}). Then for $\alpha \in \fg [1] \oplus \fg^\vee[1]$,
\[
\sum_{\gamma \text{ is a wheel with $n$ vertices} } \frac{\hbar}{|\Aut \gamma|} W^{Lie}_\gamma(\Cas_\fg,I_{CS})(\alpha) = \frac{1}{n} \Tr \left( ( \At(T_{B\fg})  )^n \right)(\alpha).
\]
Using our definition of the Chern character in section \ref{cherndefn}, we express this sum of weights as
\[
(n-1)! (-2 \pi i)^n ch_n (T_{B\fg})(\alpha).
\]
\end{lemma}

\begin{proof}
We use the correspondence between the propagator $\Cas_\fg$ and the identity operator. A cyclic ordering of vertices lets us identify the sum of the graph weights (for graphs with this ordering of vertices) with the trace of $( \At(T_{B\fg})  )^n$. We divide by $n$ because the ordering of vertices leads to an $n$-fold overcounting of the actual sum of graph weights that we desire.
\end{proof}

\subsection{The main theorem}\label{sect:mainthm}

We now use the preceding lemmas to identify the image of $I^{(1)} [\infty]$ in $\Omega^{-\ast} (B \fg).$ Recall that (following Hirzebruch \cite{H}) the Todd class can be defined in terms of Chern classes by the power series $Q (x)$ and the $\hat{A}$ class is given in Pontryagin classes via $P (x)$ where
\[
Q(x) = \frac{x}{1- e^{-x}} \text{   \; \; \;    and     \; \; \;      } P(x) = \frac{x/2}{\sinh x/2} .
\]
We define a new power series by $\log (Q(x)) - x/2$ and denote the corresponding characteristic class by $\log (e^{-c_1/2} \text{Td})$. We have an equivalence of power series (see \cite{Cartier} and \cite{BHJ})
\begin{equation}\label{pwrsrs}
\log \left (\frac{x}{1-e^{-x}} \right ) - \frac{x}{2} = \sum_{k \ge 1} 2 \zeta (2k) \frac{x^{2k}}{2k (2 \pi i)^{2k}} ,
\end{equation}
where $\zeta$ is the Riemann zeta function.


\begin{theorem}\label{inftytermBg}
For one dimensional Chern-Simons theory with values in the $\L8$-algebra $\fg \oplus \fg^\vee [-2]$, the scale $\infty$ interaction term $I^{(1)}[\infty]$ encodes the Todd class of $B\fg$ when restricted to harmonic fields $\cH$. More precisely, we have
\[
I^{(1)} [\infty] \big|_{\cH} = \sum_{k \ge 1}  \left \{ \frac{ 2(2k-1)!}{(2 \pi)^{2k}}  \zeta(2k) ch_{2k} (T_{B\fg})  \right \} 
\]
\[
= \log (e^{-c_1(T_{B\fg})/2} \text{Td} (T_{B\fg}) ) \in \Omega^{-\ast} (B\fg).
\]
\end{theorem}

\begin{proof}
First, consider the sum of weights over all wheels with $n$ vertices. We find
\[
 \sum_{\gamma \text{ is a wheel with $n$ vertices} } \frac{1}{|\Aut \gamma |} W_\gamma(P(0,\infty), I_{CS}) = \sum_\gamma \frac{1}{|\Aut \gamma |} W^{an}_\gamma W^{Lie}_\gamma
\]
and by our lemmas above, we obtain
\[
= \frac{2}{(2 \pi i)^{2n}} \zeta(n) \cdot (n-1)! (-2 \pi i)^n ch_n (T_{B\fg}) =  \frac{2 \zeta(n) \cdot (n-1)!}{(2 \pi i)^n}  ch_n(T_{B\fg}).
\]
For $n$ odd, this vanishes, as the analytic weight vanishes. 

We now use standard arguments about characteristic classes. For a sum of complex line bundles $E = L_1 \oplus \cdots \oplus L_n$, the Todd class is 
\[
Td(E) = Q(c_1(L_1)) \cdots Q(c_1(L_n)).
\]
Thus, equation \ref{pwrsrs} tells us
\[
\log (e^{-c_1(E)/2} Td(E) ) = \sum_{k \ge 1} \frac{2 \zeta (2k)}{2k (2 \pi i)^{2k}} (c_1(L_1)^{2k} + \cdots + c_1(L_n)^{2k}).
\] 
As $ch_{2k}(E) = (c_1(L_1)^{2k} + \cdots + c_1(L_n)^{2k})/(2k!)$, we obtain a general formula for an arbitrary bundle $E$,
\[
\log (e^{-c_1(E)/2} Td(E) ) = \sum_{k \ge 1} \frac{2 \zeta (2k)}{2k (2 \pi i)^{2k}} (2k)! ch_{2k}(E).
\]
This formula combines with our computation of the graph weights to yield the theorem.
\end{proof}

At first glance, $e^{-c_1/2} \Td (x)$ does not seem to be an even power series and hence only seems to define a complex genus.  However, as we saw above, it actually is even and thus defines a real genus via standard arguments going back to Hirzebruch.  Further, considered as a characteristic class for real bundles $e^{-c_1 (E) /2} \Td (E)$ agrees with $\hat{A} (E)$ (see \cite{BHJ} or \cite{Gilkey} for an index theoretic explanation). The following is then immediate.

\begin{cor}
The scale infinity effective action functional in one dimensional Chern-Simons theory is given by the logarithm of the $\hat{A}$ class.  That is,
\[
I^{(1)} [\infty] |_\cH = \log (\hat{A} (T_{B\fg})) \in \Omega^{-\ast} (B \fg).
\]
\end{cor}

\part{Topological quantum mechanics}

So far we have studied a field theory arising from a Lie algebra or an $\L8$ algebra -- its homotopical generalization -- but as geometers, we would also like to study nonlinear sigma models. Our goal in this part is to apply our methods to a certain nonlinear sigma model. {\it A priori}, gauge theories and sigma models look quite different, but the Koszul duality between commutative and Lie algebras provides a method for rewriting a certain simple sigma model as a Chern-Simons theory.  As a consequence, we can reinterpret Theorem \ref{inftytermBg} for a sigma model with target smooth manifold $X$.  

Our theorem relates the global observables of our theory to a deformed version of the de Rham complex of $T^* X$. Because our underlying classical fields are related to the loop space of $T^* X$, it is no surprise that we end up working with the negative cyclic homology of $T^*X$, which we identify with $(\Omega^{-*}(T^*X)[[u]], u d)$, where $u$ is a formal variable of cohomological degree 2 and $d$ denotes the exterior derivative with cohomological degree $-1$. Our deformation of the differential involves the $\hat{A}$ class of $X$ in a form modified to work with the negative cyclic homology: let $\hat{A}_u(X)$ denote the element in negative cyclic homology obtained by replacing $ch_{k}(X)$ by $u^k ch_{k}(X)$ wherever it appears in the usual $\hat{A}$ class.

\begin{theorem}
There exists a quantization of a nonlinear sigma model from the circle $S^1$ into $T^*[0] X$, where $X$ is a smooth manifold. Only 1-loop Feynman diagrams appear in the quantization. The solutions to the Euler-Lagrange equations consist of constant maps into $T^*[0] X$. The $S^1$-invariant global quantum observables over $S^1$ form a cochain complex quasi-isomorphic to the following deformation of the negative cyclic homology of $T^*X$:
\[
(\Omega^{-*}(T^*X)[[u]][[\hbar]], ud + \hbar L_\pi + \hbar\{\log (\hat{A}_u(X)), -\}),
\]
where $L_\pi$ denotes the Lie derivative with respect to the canonical Poisson bivector $\pi$ on $T^* X$.
\end{theorem}

\begin{remark}
If we worked with a complex manifold instead of a smooth manifold, in the style of Costello's work on the Witten genus, we would recover the Todd genus instead and could avoid working with cyclic homology.
\end{remark}

We will show that this theorem follows from Theorem \ref{inftytermBg}. First, we show how to encode the smooth manifold $X$ as $B\fg_X$, where $\fg_X$ is an $\L8$ algebra, and hence $T^*[0]X$ as $T^*B\fg_X$. This result lets us apply our results from Part I. Then we explain how this QFT relates to the usual sigma model, known as the ``free particle in $X$." The only work that remains to prove the theorem is to show that $\hat{A}_u(X)$ actually arises by using $B\fg_X$. At the end, we explain how the theorem above relates to Theorem \ref{ManifoldThm} stated in the introduction.

\section{Koszul duality and formal geometry}

We would like to encode the smooth geometry of the manifold $X$ in the language of $\L8$ algebras, as this would allow us to apply Chern-Simons theory. We construct such an $\L8$ algebra from the perspective of dg manifolds.\footnote{The discussion from hereon will use the language of jets, $D$-modules, and dg manifolds quite heavily, so we encourage the reader to skim the appendices for our conventions.} It turns out that this construction fits nicely with the language of formal geometry. In fact, this perspective informed Costello's construction in \cite{Cos3}, from which we draw inspiration.

\subsection{Encoding a smooth manifold as an $\L8$ algebra}

Consider the canonical map of dg manifolds $X \to X_{dR}$ arising from the quotient map of commutative dg algebras $\Omega_X \to \cinf_X$. This map identifies points in $X$ that are infinitesimally close, so a fiber of the map essentially looks like an infinitesimal neighborhood of a point in $X$. Since the functions on an infinitesimal disk look like formal power series, we thus expect that the structure sheaf of $X$, as a space over $X_{dR}$, looks like the sheaf $\sJ$ of $\infty$-jets of smooth functions, or rather the de Rham complex of that $D_X$-module. The sheaf of jets looks huge as a sheaf of vector spaces but it is of manageable size as a $\cinf_X$-module, so we can apply Koszul duality to encode the de Rham complex of jets using an $\L8$ algebra.

The following lemma makes the heuristic picture above precise.\footnote{This result is a direct analogue of a lemma from \cite{Cos3}.}

\begin{lemma}\label{tangentL8}
There is a curved $\L8$ algebra $\fg_X$ over $\Omega_X$, with nilpotent ideal $\Omega^{> 0}_X$, canonical up to a contractible choice, such that
\begin{enumerate} 

\item $\fg_X \cong \cT_X[-1] \otimes_{C^\infty_X} \Omega^\#_X$ as an $\Omega^\#_X$ module;

\item $C^*(\fg_X) \cong dR(\sJ)$ as commutative $\Omega_X$ algebras;

\item $C^*(\fg_X) \simeq C^\infty_X$ as $\Omega_X$ modules.

\end{enumerate}
\end{lemma}

\begin{proof}
We need to show that we can equip $\csym_{\cinf_X}(\cT_X^\vee) \ot_{\cinf_X} \Omega_X$ with a degree $1$ derivation $d$ such that $d^2 = 0$ (this is the curved $\L8$ structure) and such that this Chevalley-Eilenberg complex is quasi-isomorphic to $C^\infty_X$ as an $\Omega_X$ module. In this process we will see the second property explicitly.

We start by working with $D_X$ modules and then use the de Rham functor to translate our constructions to $\Omega_X$ modules. Consider the sheaf $\sJ$ of infinite jets of smooth functions. Observe that there is a natural descending filtration on $\sJ$ by ``order of vanishing." To see this explicitly, note that the fiber of $J$ at a point $x$ is isomorphic (after picking local coordinates $x_1, \ldots, x_n$) to $\CC[[x_1, \ldots, x_n]]$, and we can filter this vector space by powers of the ideal $\mathfrak m = (x_1,\ldots,x_n)$. We define $F^k \sJ$ to be those sections of $\sJ$ which live in $\mathfrak m^k$ for every point. This filtration is not preserved by the flat connection, but the connection does send a section in $F^k \sJ$ to a section of $F^{k-1} \sJ \ot_{\cinf_X} \Omega^1_X$.

Observe that $F^1 \sJ / F^2 \sJ \cong \Omega^1_X$, because the first-order jets of a function encode its exterior derivative. Moreover, $F^k \sJ / F^{k+1} \sJ \cong \Sym^k(\Omega^1_x)$ for similar reasons. Pick a splitting of the map $F^1 \sJ \rightarrow \Omega^1_X$ as $\cinf_X$ modules; we denote the splitting by $\sigma$. (Note that there is a contractible space of such splittings, see the discussion below.) By the universal property of the symmetric algebra, we get a map of non unital $\cinf_X$ algebras that is, in fact, an isomorphism 
\[
 \mathrm{Sym}^{>0}_{\cinf_X}(\Omega^1_X) \xrightarrow{\cong} F^1 \sJ.
\]
Now both $\csym_{\cinf_X} (\Omega^1_x)$ and $\sJ$ are augmented $\cinf_X$ algebras with augmentations
\[
p: \csym_{\cinf_X} (\Omega^1_X) \to \mathrm{Sym}^0 = \cinf_X \text{ and } q: \sJ \to \sJ/ F^1 \sJ \cong \cinf_X .
\]
Further, $\mathrm{Sym}^{>0}_{\cinf_X}(\Omega^1_X) = \ker p$ and $F^1 \sJ = \ker q$, so we obtain an isomorphism of $\cinf_X$ algebras
\[
\csym_{\cinf_X} (\Omega^1_X) \xrightarrow{\cong_\sigma} \sJ
\]
by extending the previous isomorphism by the identity on $\mathrm{Sym}^0_{\cinf_X}$ and $\sJ / F^1 \sJ$.  The preceding discussion is just one instance of the equivalence of categories between commutative non unital $A$ algebras and commutative augmented $A$ algebras for $A$ any commutative algebra.

We then equip $\csym(\Omega^1_X)$ with the flat connection for $\sJ$, via the isomorphism, thus making it into a $D_X$ algebra. Applying the de Rham functor $dR$, we get an isomorphism of $\Omega_X$ algebras
\[
 \csym_{\cinf_X} (\Omega^1_X) \ot_{\cinf_X} \Omega_X \xrightarrow{\cong_\sigma} \sJ \otimes_{\cinf_X} \Omega_X .
\]
Recall that the symmetric algebra is compatible with base change, that is
\[
\csym_{\cinf_X} (\Omega^1_X ) \otimes_{\cinf_X} \Omega^\sharp_X = \csym_{\Omega^\sharp_X} (\Omega^1_X \otimes_{\cinf_X} \Omega^\sharp_X ) \cong \csym_{\Omega^\sharp_X} \left ( (\cT_X [-1] \otimes_{\cinf_X} \Omega^\sharp_X)^\vee [-1]\right ),
\]
where we dualize over $\Omega^\sharp_X$. Via the de Rham functor we have constructed a derivation on this completed symmetric algebra defining the $\L8$ structure over $\Omega_X$. 

Finally, the third property follows immediately from a standard argument that the cohomology of the de Rham complex of jets is concentrated in degree 0 (see \cite{CFT}). 
\end{proof}

That the space of splittings of the jet sequence
\[
\xymatrix{
0 \ar[r]& F^2 \sJ \ar[r] & F^1 \sJ \ar[r] & \Omega^1_X \ar[r] \ar@/_1pc/[l]_\sigma & 0 .}
\]
is non-empty and contractible is proved in \cite{Grady}, see also \cite{CSX}.  Further, in \cite{Grady} it is shown that the assigment $X \mapsto \fg_X$ is in a certain sense functorial.

\subsection{Relation to Gelfand-Kazhdan formal geometry}

The construction in the previous subsection can also be motivated via the approach to formal geometry going back to Gelfand, Kazhdan, and Fuchs (see \cite{GK}, \cite{GKF}, \cite{BR}, or the more recent \cite{BK}, \cite{BF}).  The idea of this approach is to build something over the formal disk and then use so-called {\it Harish-Chandra localization} to glue this construction over a manifold.  

Given an $n$-dimensional smooth manifold $X$, let $X^{coor}$ denote the infinite-dimensional manifold of maps from the parametrized formal $n$-disk into $X$. More explicitly, a point of $X^{coor}$ is given by taking a local diffeomorphism $\phi: U \subset \RR^n \to X$, where $U$ is an open neighborhood of the origin $0 \in \RR^n$, and then taking its $\infty$-jet (aka Taylor expansion) at $0$. We view two such representatives $\phi$, $\psi$ as equivalent if they have the same $\infty$-jet. By evaluating the jet at $0$, we have a projection map $\pi: X^{coor} \to X$. We now explain how $X^{coor}$ is a special kind of principal bundle.

Let $\sW$ be the Lie algebra of formal vector fields:
\[
\sW = \left \{ \sum_{j=1}^n v_j \frac{\partial}{\partial y_j} : v_j \in \RR [[ y_1 , \dotsc , y_n ]] \right \} .
\]
$\sW$ acts infinitesimally on the formal $n$-disk and the action restricted to the subalgebra of vector fields vanishing at the origin can be integrated to an action of the Lie group $G_0$ of formal coordinate transformations of $\RR^n$. The pair $(\sW, G_0)$ is known as a {\it Harish-Chandra pair}.

Observe that $X^{coor}$ has a natural action of this pair $(\sW, G_0)$, because the Lie algebra $\sW$ and the group $G_0$ both act compatibly on the formal $n$-disk. In fact, $\pi: X^{coor} \to X$ is a principal $G_0$-bundle, and the action of the pair makes this bundle into a {\it Harish-Chandra structure} (see \cite{BF} for a discussion of all the necessary conditions). There is an intermediate space that often appears in discussions of formal geometry. The subgroup $\mathrm{GL} (n, \RR) < G_0$ of linear diffeomorphisms acts on $X^{coor}$ freely with quotient denoted by $X^{aff}$. We summarize all of the above with the following sequence of fibrations
\[
\begin{xymatrix}{
\mathrm{GL}(n,\RR) \ar@{^{(}->}[r] & X^{coor} \ar@{->>}[d] \\ G_0 / \mathrm{GL}(n,\RR) \ar@{^{(}->}[r] & X^{aff} \ar@{->>}[d] \\ & X}
\end{xymatrix}
\]
The utility of $X^{aff}$ is that it is an {\em affine} bundle and hence easier to work with.

{\it Harish-Chandra localization} just amounts to performing an associated bundle construction.  That is, given a module $V$ for the pair $(\sW , G_0)$ -- i.e., a vector space with actions of $\sW$ and $G_0$ satisfying certain compatibility relations -- we can build a vector bundle $\sV \to X$ by defining  
\[
\sV := X^{coor} \times_{G_0} V \to X .
\]
For example, let $V = \CC [[x_1 , \dotsc , x_n]]$ be the functions on the formal disk.  Then $\sV = \sJ$ ---Harish-Chandra localization for the module $\CC[[x_1, \dotsc , x_n ]]$ recovers the jet bundle on $X$!  We can go further and use the compatible action of $\sW$ to equip the bundle $\sV$ with a flat connection. In the case $V = \CC[[x_1 , \dotsc , x_n]]$, this yields the standard flat connection on the jet bundle.

In \cite{Cos3}, Costello implicitly uses the following observation and we follow suit in this paper. One way to construct an object living over some space $Y$ is to do a Borel construction on a principal bundle of $Y$.  Harish-Chandra localization is an example of this approach: the Borel construction takes a module $V$ for $G_0$ and gives us a vector bundle $\sV$ over $X$, and if $V$ is actually a module for the Harish-Chandra pair $(\sW , G_0)$, then Harish-Chandra localization eqips the same bundle $\sV \to X$ with a flat connection.  The resulting bundle with flat connection should be the same as directly performing a construction over the quotient of $X^{coor}$ by the Harish-Chandra pair $(\sW, G_0)$; this quotient doesn't exist in manifolds, but it does exist as a dg manifold, namely $X_{dR}$.\footnote{One way to see this is to consider the action of the Lie algebra of vector fields on smooth functions and see that the Chevalley-Eilenberg complex $C^\ast (\cT_X , C^\infty_X)$ is quasi-isomorphic to the de Rham complex.} The upshot is that we can forgo all of the formal geometry constructions by working in dg manifolds and constructing our objects of interest directly over $X_{dR}$.

\subsection{A circle action on this $\L8$ algebra}\label{circleaction}

In what follows, we will consider the dg manifold of maps from $S^1_{\Omega}$ into $T^*[0]X$. Using our techniques from above, we know that this dg manifold, the derived loop space 
\[
\cL T^\ast [0]X \cong T[-1] T^\ast[0]X,
\]
has a description in terms of the $\L8$ algebra 
\[
\Omega^\ast(S^1) \otimes (\fg_X \oplus \fg_X^\vee [-2]), 
\]
which is quasi-isomorphic to the $\L8$ algebra
\[
\CC[\epsilon] \otimes (\fg_X \oplus \fg_X^\vee [-2]),
\]
where $\epsilon$ is a square zero parameter of degree 1.  More precisely, we have seen implicitly in section \ref{observables} that this $\L8$ algebra encodes the structure sheaf of the derived loop space $\cL T^\ast [0]X$. Now, by definition, the functions on this derived space $\sO(T[-1]T^\ast[0]X)$ are $\Omega^{-\ast}_{T^\ast B\fg}$, namely the complex with $\Omega^k_{T^\ast B \fg}$ in degree $-k$.  Note that there is no de Rham differential appearing in this complex, just the internal differential coming from the $\Omega^\ast_X$-algebra structure.  

We somehow need to restore the de Rham differential in order to see the $\hat{A}$ class. Recall from the remark \ref{whythecircleaction} that scalar Atiyah classes vanish in the smooth setting. The standard method is to take advantage of the action of $\CC[\epsilon]$ on $\cL T^\ast[0]X$; the de Rham differential corresponds to the $\L8$ algebra derivation $\partial/\partial \epsilon$. We prefer to think of this as an action of the dg manifold $B \GG_a = (pt, \CC[\epsilon])$. If one views $B\GG_a$ as an avatar of the circle, then this action is a version of ``rotating the loops."

These constructions are well-known (see appendix \ref{Mixedcomplexes}), usually referred to by the name of {\em mixed complexes} or {\em cyclic modules} (see \cite{BZN} and \cite{TV}). If we ask for the $B\GG_a$-invariant functions on $\cL T^\ast (X, \fg_X)$, we obtain the {\em negative cyclic homology} of $T^\ast (X, \fg_X)$. For a thorough discussion of these ideas in the language of derived geometry, see \cite{BZN} and \cite{TV}. We emphasize these circle actions here as they are crucial for actually recovering the $\hat{A}$-class in smooth geometry.

\section{Motivation for our action functional}\label{large volume limit}

In the remainder of this part, we will study Chern-Simons with the $\L8$ algebra $\fg_X$, but before embarking on that study, we want to discuss how this theory relates to other forms of quantum mechanics. On its face, this Chern-Simons action functional does not resemble the usual action functional for a free particle, but one does recover the usual algebra of observables, namely the ring of differential operators on the manifold $X$, so it would be nice to know how the two theories are related.\footnote{Proving this assertion about the observables is one of the main goals of a followup paper.} There is a natural construction that relates the two theories, and it consists of three steps. First, as described below, we'll take the {\em infinite-volume limit} of the action for the usual free particle. Second, we will apply the Batalin-Vilkovisky procedure to this action. Together, these steps yield the AKSZ action functional for maps from a 1-manifold into $T^\ast X$, with the standard symplectic form. In the final step, a form of Koszul duality then lets us re-express the AKSZ action as a version of Chern-Simons theory.\footnote{This procedure appears to work quite well for many nonlinear sigma models: the combination of the infinite-volume limit with the BV procedure yields an AKSZ theory. Rewriting the theory as a gauge theory is useful simply because it allows us to apply the toolkit developed by Costello \cite{Cos1}.} To summarize, our discussion here outlines a proof of the following theorem.

\begin{theorem}
Given a Riemannian manifold $(X,g)$, there is family of classical field theories parametrized by $\ell \in [0,1]$ such that for $\ell > 0$, the theory is equivalent to the sigma model with target $(M, g/\ell)$, and for $\ell = 0$ the theory is equivalent to the one-dimensional Chern-Simons theory with $\L8$ algebra $\fg_X$.
\end{theorem}

\subsection{Step 1: the infinite-volume limit}

Recall that the phrase ``free particle" refers to studying maps of an interval $I$ (or circle) into a Riemannian manifold $(X,g)$. The action $S_{fr}$ of a map $\phi: I \to X$ is
\[
S_{fr}(\phi) = \frac{1}{2} \int_I \phi^\ast g(\partial_t \phi, \partial_t \phi) \, dt,
\]
which is simply the integral of the kinetic energy of the particle over the path traveled. The critical locus of $S_{fr}$ is the space of geodesics in $X$.

There is another action functional with exactly the same critical locus. It arises by including another field that encodes the momentum of the particle (and so is sometimes called the ``first-order formulation" of the theory). Let $\phi \in \Maps(I,X)$ and $\psi \in \Gamma(I, \phi^\ast T^\ast X)$. Then
\[
S_{FO} (\phi, \psi) = \int_I \langle \partial_t \phi, \psi \rangle - \frac{1}{2} \phi^* (g^{-1}) (\psi, \psi) \, dt,
\]
where $\langle -, - \rangle$ denotes the canonical pairing between vector fields and covector fields and $g^{-1}$ denotes the metric on $T^\ast_X$ induced by the metric $g$. The Euler-Lagrange equations for $S_{FO}$ are that $\partial_t \phi = \psi^\vee$ (the vector field dual to $\psi$ using $\phi^\ast g$) and $\partial_t \psi= 0$, and so the critical locus of $S_{FO}$ is again the space of geodesics in $X$. Note that our space of fields has changed to the mapping space $\Maps(I, T^\ast X)$.

We now take the ``infinite-volume limit" of the first-order formulation. This means we scale the metric $g$ by a parameter $1/\ell$ and consider how the theory changes as $\ell \to 0$. Geometrically, making the metric larger corresponds to flattening out the local geometry of $X$. In particular, the volume grows toward infinity as $\ell \to 0$. Algebraically, it means that the second term in $S_{FO}$, depending quadratically on $\psi$, becomes less significant, since it is weighted by $\ell$. Notice that the first term, since it involves the canonical pairing, is unchanged by dilating the metric. Hence, as $\ell \to 0$, the first term comes to dominate, and the infinite-volume limit of the action functional is
\[
S_{IVL}(\phi, \psi) = \int_I \langle \partial_t \phi, \psi \rangle \, dt.
\]
The Euler-Lagrange equations are $\partial_t \phi= 0$ and $\partial_t \psi= 0$, so the critical locus is the space of {\em constant} maps into $T^\ast X$. This limiting behavior should be intuitively reasonable: if we fix an interval $I = [0,1]$ and look at any trajectory $\gamma_\ell$ satisfying the Euler-Lagrange equation for $S_{FO}$ with scale $\ell$, the path grows shorter with respect to the original metric $g$ as $\ell$ grows smaller. In the infinite volume limit, you only remember the point $\gamma(0)$ and its (co)tangent vector. For the infinite volume theory, the solutions are $T^*X$ now viewed as parametrizing the space of geodesics via their initial conditions.

There is another description of this action functional that is probably more familiar. Let $\lambda$ denote the Liouville 1-form on $T^\ast X$; that is, $d\lambda$ is the natural symplectic form on $T^\ast X$. In local coordinates, $\lambda$ would usually be expressed as $\sum p_i \, dq_i$. The action functional is then that
\[
S_{IVL}(f) = \int_{I} f^\ast \lambda,
\]
where $f \in \Maps(I, T^*X)$.

\begin{remark}
Taking a infinite-volume limit is a natural way to turn a theory that depends on the geometry of the target $X$ into a theory that depends only on the smooth topology of $X$. It drastically simplifies the physics. On the other hand, we might hope that we can express the observables for our family of theories, parametrized by $\ell$, in a power series in $1/\ell$, so that we know the physics for large but not infinite metric as a deformation of the topological field theory. This idea is the subject of work in progress joint with Si Li.
\end{remark}

\subsection{Step 2: the BV procedure}

Now we apply the BV procedure. We need to phrase the theory in a way where we can add ``antifields," which will describe the {\em derived} critical locus of our action functional $S_{LVL}$. 

We need to construct a version of the shifted cotangent bundle of $\sM := \Maps(S^1, T^\ast X)$. Observe that at a point $f \in \sM$, the tangent space is
\[
T_f \sM = \Gamma(S^1, f^\ast T(T^\ast X)).
\]
The space 
\[
\Gamma(S^1, T^\ast S^1 \ot f^\ast T^\ast (T^\ast X)) = \Omega^1(S^1,  f^\ast T^\ast (T^\ast X))
\] 
has a natural pairing with $T_f \sM$ by pairing the sections of the dual pullback bundles of $T^\ast X$ and then integrating over $S^1$. We will use it as the version of $T^\ast_f \sM$ appropriate to our purposes. Hence,
\[
T^\ast_f [-1] \sM := \Omega^1_{S^1}(f^\ast T^\ast (T^\ast X))[-1].
\]
Use the symplectic form on $T^\ast X$ to identify its cotangent bundle with its tangent bundle. Then we see
\[
T^\ast[-1] \sM \cong \Maps(S^1_{\Omega^\#}, T^\ast X),
\]
where $\Omega^\#$ denotes the differential forms as a graded algebra. There is a natural extension of $S_{LVL}$ to an action functional
\[
S(f) = \int_{S^1} \langle f, df \rangle
\]
for $f \in \Maps(S^1_{\Omega^\#}, T^\ast X)$, where the brackets denote the symplectic form on $T^\ast X$. Its critical locus is equivalent to equipping $T^\ast [-1]\sM$ with the differential that makes it the dg manifold
\[
\Maps(S^1_{dR}, T^\ast X).
\]
This is the derived critical locus of $S_{LVL}$.

Nots that we have recovered the AKSZ theory with source $S^1_{dR}$ and target $T^* X$. 

\subsection{Concentrating our attention on a neighborhood of the zero section}

So far, we have talked about the space $T^*X$, but our QFT uses $T^*B\fg_X$, which is equivalent to $T^*[0]X$. These spaces are not exactly the same: $T^*[0]X$ is the formal neighborhood of the zero section $X \hookrightarrow T^*X$. Notice, however, that the BV action above depends linearly on $\psi$, just as our discussion in section \ref{actionsection}. If we require our quantization to preserve this symmetry, we obtain a theory that encodes the same data as the theory with $T^*[0]X$. As usual, requiring equivariance under rescaling the cotangent fibers means that we can instead study the formal neighborhood of the zero section.

\section{Characteristic classes via formal geometry}\label{charclass}

The aim of this section is to explain how the characteristic classes of a complex vector bundle $\pi: E \to X$ can be expressed using the language of formal geometry. In essence, we develop a Chern-Weil construction of characteristic classes using Atiyah classes. To be precise, we apply the Atiyah class formalism to $J(E)$, the $\infty$-jet bundle of $E$, and then relate the Atiyah class of $J(E)$ to the curvature of a connection on $E$. Our goal is the following proposition.

\begin{prop}\label{chernX}
Let $\alpha_E$ denote the Atiyah class of $J(E)$, defined below (see section \ref{sec:AtVB}). The Chern class $ch_k(E) \in H^{2k}(X)$ is given by the cohomology class $\frac{1}{k!(-2\pi i)^k} [\Tr(\alpha_E^{\wedge k})]$, under the isomorphism in corollary \ref{cor:horizontal}.
\end{prop}

We will build up to this proposition in stages.\footnote{In the appendix on $D$-modules, we give more background on $D$-modules, jets, and other constructions that are used throughout the following section.} First, we will discuss the algebra of jets of smooth functions $\sJ$ and various useful constructions on it. Then we explain how to construct a connection on $J(E)$, the jets of the vector bundle $E$. Finally, we exploit the Atiyah class of this connection to recover the Chern-Weil construction of Chern classes. 

\subsection{The algebra of jets}

Let $\sJ$ denote the $\infty$-jet bundle of the trivial bundle on $X$. It is a commutative $D_X$-algebra, as can be seen locally by the natural product on Taylor series. By construction, it encodes the smooth geometry of the manifold $X$. In this paper we build everything over the dg manifold $X_{dR}$, so that all constructions automatically come equipped with a flat connection. Hence, in what follows, we work with the commutative $\Omega_X$-algebra $\fJ$, the de Rham complex $dR(\sJ)$ of the jets $\sJ$. 

By construction, the de Rham differential $d_{dR}: \fJ \to \Omega^1_\fJ$ commutes with the differentials $d$ on these $\Omega_X$-modules and by construction $(d_{dR})^2 = 0$, so we get a cochain complex of $\Omega_X$-modules
\[
\fJ \overset{d_{dR}}{\longrightarrow} \Omega^1_\fJ \overset{d_{dR}}{\longrightarrow} \Omega^2_\fJ \to \cdots \to \Omega^{\dim X}_\fJ.
\]
This double complex $(\Omega^*_\fJ, d_{dR})$ provides a description, using $\Omega_X$ modules, of the usual de Rham complex of $X$.\footnote{This construction probably seems tortuous, if not gratuitous, but it arises naturally from our approach to formal geometry.} The following proposition makes this interpretation precise.

\begin{prop}\label{prop:jetquasiiso}
As $\Omega_X$ modules, $\Omega^k_\fJ$ is isomorphic to $dR(J(\Omega^k_X))$.
\end{prop}

\begin{proof}
We explain the case $k = 1$, as the other cases follow straightforwardly. 

The exterior derivative $d: \cinf_X \to \Omega^1_X$ is a differential operator, so by proposition \ref{diff} in the appendix, we see that it induces a map of $D_X$ modules $J(d): \sJ \to J(\Omega^1_X)$. Hence, we obtain a map of $\Omega_X$ modules $J(d): \fJ \to dR(J(\Omega^1_X))$. By construction this map is a derivation, so the universal property of K\"ahler differentials insures that there is a natural map $\theta: \Omega^1_\fJ \to dR(J(\Omega^1_X))$ of $\Omega_X$ modules.

We need to show this map is an isomorphism, and it's easy to do this locally.
\end{proof}

\begin{cor}\label{cor:horizontal}
The horizontal sections of $\Omega^k_\fJ$ are precisely $\Omega^k_X$. Moreover, the map $J$ sending a smooth form to its infinite jet induces a quasi-isomorphism of complexes $J: \Omega^*_X \overset{\simeq}{\hookrightarrow} \Omega^*_{\fJ}$.
\end{cor}

\subsection{The Atiyah class of a vector bundle}\label{sec:AtVB}

Let $J(E)$ denote the $\infty$-jet bundle of $E$. Its sections consist of jets of smooth sections of $E$, and there is a canonical flat connection $\nabla_{J(E)}$ whose kernel is exactly the smooth sections $\sE$ of $E$. There is a natural filtration on $J(E)$ by the order of vanishing:
\[
J(E) = F^0 \supset F^1 \supset \cdots
\]
where $F^k$ consists of those sections of $J(E)$ whose $k$-jets are zero. Observe that there is a canonical isomorphism $J(E)/F^1 \overset{\cong}{\to} \sE$ as a $C^\infty_X$ module. Pick a splitting $\sigma: \sE \to J(E)$ for the canonical quotient $q: J(E) \to \sE$, as discussed in lemma \ref{splitting}.

Let $\fJ(E)$ denote the de Rham complex of $J(E)$. This isomorphism $i_\sigma$ also induces an isomorphism $\fJ(E) \cong \sE \ot_{\cinf_X} \fJ$. Hence $\fJ(E) \ot_\fJ \Omega^1_\fJ \cong \sE \ot_{\cinf_X} \Omega^1_\fJ$, and hence we obtain a natural connection (with respect to the splitting $\sigma$) on $\fJ(E)$:
\[
\nabla_\sigma: \sE \ot_{\cinf_X} \fJ \to \sE \ot_{\cinf_X} \Omega^1_\fJ,
\]
\[
s \ot j \mapsto s \ot d_{dR} j,
\]
where $d_{dR}$ is the de Rham differential on the commutative algebra $\fJ$. Moreover, by construction,  $\nabla_\sigma^2 = 0$, so this connection is {\em flat}! (The de Rham differential $d_{dR}$ on $\fJ$ is flat, and we immediately borrow this fact.) Thus we have a cochain complex of $\Omega^\#_X$-modules
\[
\fJ(E) \overset{\nabla_\sigma}{\longrightarrow} \fJ(E) \ot_{\fJ} \Omega^1_\fJ  \overset{\nabla_\sigma}{\longrightarrow} \fJ(E) \ot_{\fJ} \Omega^2_\fJ \to \cdots
\]
but there is no reason to expect this connection to be compatible with the $\Omega_X$-module structure of these sheaves. 

Let $\alpha_\sigma = \At(\nabla_\sigma)$ denote the Atiyah class for our connection $\nabla_\sigma$. It measures the failure of this connection to be compatible with the differentials defined on these $\Omega_X$-modules $\Omega^k_\fJ(\fJ(E))$.

\subsection{Proving the proposition}

We are now in a position to approach the main proposition. The basic strategy is to relate the Atiyah class $\alpha_\sigma$ to the usual Chern-Weil construction of the Chern classes.\footnote{Alternatively, one could verify the axioms of the Chern classes directly. We hope our approach illustrates the yoga of Gelfand-Kazhdan formal geometry.}

Recall the definition of the Chern character from Section \ref{cherndefn}. This definition obviously bears a close resemblance to the definition of the Chern character in terms of a connection on a vector bundle (simply replace the Atiyah class with the curvature). Hence, our strategy will be to relate $\alpha_\sigma$ to an actual connection on $E$ in such a way that the two definitions of Chern character will coincide. 

In more detail, the argument runs as follows. We show that our Chern character is expressed in terms of elements $\omega_k \in \Omega^*_{\fJ}$ that correspond to {\em closed} forms in $\Omega^{even}_X$ (see lemma \ref{closed} below). Hence, although each $\omega_k$ lives in some kind of jet bundle, it is determined by its ``constant coefficient" part (i.e., its projection onto the bundle ``$J/F^1$"), just the way that a smooth function determines its $\infty$-jet. We then show that this constant coefficient part corresponds to the curvature of a connection on $E$ arising naturally from our choice of splitting $\sigma$. This step will explain the relationship to the usual Chern-Weil construction.

By definition, $\alpha_\sigma$ is an element of cohomological degree 1 in  $\Omega^1_\fJ \ot_\fJ \End_\fJ( \fJ(E))$, so it lives in $\Omega^1_X \left(\Omega^1_\fJ \ot_\fJ \End_\fJ( \fJ(E)) \right)$. Note that $\alpha_\sigma^{\wedge k} \in \Omega^k_X \left(\Omega^k_\fJ \ot_\fJ \End_\fJ( \fJ(E)) \right)$. Hence, we find that the form $\omega_k := \Tr(\alpha_\sigma^{\wedge k})$ lives in $\Omega^k_X \left(\Omega^k_\fJ \right)$. 

\begin{lemma}\label{closed}
$\omega_k$ defines a closed form in $\Omega^{2k}_X$.
\end{lemma}

\begin{proof}
The lemma follows from the useful facts about the Atiyah class that we proved earlier. In particular, we see that
\begin{enumerate}
\item $d_{dR} \omega_k = 0$ and
\item $\omega_k$ is a horizontal section.
\end{enumerate} 
Thus $\omega_k$ is a closed form in the total complex of the double complex $(\Omega^*_\fJ, d_{dR})$. 
\end{proof}

Thus, $ch(\nabla_\sigma) = \sum_k \frac{1}{(-2\pi i)^k k!} \omega_k$. 

\begin{prop}
The splitting $\sigma$ induces a connection $\nabla^E_\sigma$ on the bundle $E$.
\end{prop}

\begin{proof}
There is a natural connection on $E$ arising from our splitting $\sigma$ as follows. It is a composition of three natural maps. First, there is an important map of sheaves\footnote{Notice this is {\em not} a map of $\cinf_X$-modules!} $J: \sE \to J(E) $ sending a smooth section $f$ to its $\infty$-jet. Second, we have the connection defined on $\fJ(E)$ by $\nabla_\sigma$. Finally, we have the quotient map $q: \fJ(E) \ot_\fJ \Omega^1_\fJ \to \sE \ot_{\cinf_X} \Omega^1_X$ that returns the ``constant coefficient" term (i.e., from the filtration by order of vanishing). In sum, there is a map $\nabla^E_\sigma: \sE \to \sE \ot \Omega^1_X$ given by $q \circ \nabla_\sigma \circ J$, and it defines a connection on the bundle $E$.
\end{proof}

We want to describe this connection explicitly in local coordinates so that it is clear how our construction relates  with Chern-Weil theory. Fix coordinates $x_1, \ldots, x_n$ on some small ball $U \subset X$ and fix a basis $e_1, \ldots, e_k$ for the fiber $E_0$. With respect to these choices, the splitting $i_\sigma: \sE \ot \sJ \to J(E)$ over $U$ yields a $\cinf_X$-linear map 
\[
S: \cinf_X(U) \ot_\RR \RR[[x_1,\ldots,x_n]] \ot_\RR E_0 \to \cinf_X(U) \ot_\RR \RR[[x_1,\ldots,x_n]] \ot_\RR E_0,
\] 
which we write as a sum of its homogeneous components $S = 1 + S_1 + S_2 + \cdots$. Since $i_\sigma$ is determined by its behavior on $\sE$, we only need to say how $S$ acts on $E_0$ to fully describe $S$. Hence, $S_k$ denotes the degree $k$ part of $S$ on $E_0$. More explicitly 
\[
S_k = \sum_{|\alpha| = k} s^i_{j, \alpha} x^\alpha (e_j \ot e_i^\vee), \text{ with } s^i_{j,\alpha} \in \cinf_X.
\]
In words, $S_k$ sends a constant section $e_i$ of $\sE|_U$ to a section of $J(E)$ that has degree $k$ in the formal variables $\{x_m\}$. Note that the coefficients $s^i_{j,\alpha}$ are smooth functions on $U$.

The connection $\nabla_\sigma$ on $\fJ(E)$ has the form $S \circ (1_E \ot d_{dR}) \circ S^{-1}$. The lowest order term of this map is $d_{dR} - d_{dR}(S_1)$, and this term gives the connection $\nabla^E_\sigma$.

\begin{lemma}
The curvature of $\nabla^E_\sigma$ on $E$ corresponds to the constant coefficient term of the Atiyah class of $\nabla_\sigma$. Explicitly, in local coordinates on the ball $U \subset M$, we find $q(\alpha_\sigma) = - d d_{dR}(S_1) + d_{dR}(S_1) \wedge d_{dR}(S_1)$. 
\end{lemma}

Hence we know the elements $\omega_k$  can be identified, by taking $\infty$-jets, with the corresponding differential forms arising from the curvature of the connection $\nabla^E_\sigma$. In other words, they agree with the usual Chern-Weil construction.

\begin{proof}
We check this locally. The notational burden becomes heavy, so we describe the approach before the barrage of indices begins.

Recall the expression from lemma \ref{atiyah} for the Atiyah class of a free $R$-module $M$ in terms of a basis:
\[
\At(\nabla) = d_{dR} A - d_{\Omega^1 \ot \End M} B,
\]
where $d_M = d_R + A$ and $\nabla = d_{dR} + B$. We will show that, in our situation, the term $d_{dR} A = 0$ and then that $- q \circ d_{\Omega^1 \ot \End M} B$ will be precisely the curvature of $\nabla^E_\sigma$.

To see that $d_{dR} A = 0$, we need to describe $A$. On an open ball $U \subset X$, once we pick coordinates $\{x_1,\ldots, x_n\}$ on $U$, we get a trivialization of the sections of $J$ as $\cinf(U) \ot \RR [[x_1,\ldots,x_n]]$. The differential on $J$ is
\[
d_J(f\ot x^\alpha) = (df) \ot x^\alpha - \sum_k \alpha_k f dx_k \ot x^{\alpha - e_k}.
\]
Once we pick a trivialization $E|_U \cong U \times E_0$ and coordinates $\{x_1,\ldots, x_n\}$ on $U$, we get a trivialization of the sections of $J(E)$ as $\cinf(U) \ot \RR [[x_1,\ldots,x_n]] \ot E_0$. The differential on $J(E)$ has the form
\[
d_{J(E)}(f \ot x^\alpha \ot v) = d_J(f \ot x^\alpha) \ot v,
\]
so there is no ``connection 1-form" part of the differential (with respect to the basis we're using). Hence $A = 0$ and so $d_{dR} A = 0$.

Now it remains to compute the term $- q \circ d_{\Omega^1 \ot \End M} B$. As we just saw above, the differential on $\fJ(V)$, for any vector bundle $V$, has the form $d = d_{dR} - d_J$, where 
\[
d_{dR} (\omega \ot x^\alpha \ot v) = d\omega \ot x^\alpha \ot v
\] 
and 
\[
d_J (\omega \ot x^\alpha \ot v) = \sum_j \alpha_j dx_j \wedge \omega \ot x^{\alpha - e_j} \ot v.
\] 
Hence the term $d_{dR}$ is, in fact, simply the exterior derivative (i.e., the differential on $\Omega_X$). Writing an element $B$ of this jet bundle $J(V)$ in terms of its homogeneous components $B_k$, we see 
\[
q \circ d_{J(V)} (B_0 + B_1 + B_2 + \cdots) = d_{dR} B_0 - d_J B_1.
\]
We now apply this observation to the ``connection 1-form" $B \in \fJ(\End E)$ for our connection $\nabla_\sigma$.

We use notation from the proof of the previous proposition. Let $S^{-1} = 1 + T_1 + T_2 + \cdots$. Note that $T_1 = -S_1$. Hence, the low order terms of $\nabla_\sigma$ are
\[
(1 + S_1 + S_2 + \cdots) \circ d_{dR} \circ (1 - S_1 + T_2 + \cdots) 
\]
\[
= 
\underbrace{d_{dR} - d_{dR}(S_1)}_{\text{order 0}} +
\underbrace{d_{dR}(T_2) - S_1 \wedge d_{dR}(S_1)}_{\text{order 1}}
+ \cdots
\]
and so the connection 1-form is $B = - d_{dR} (S_1) + d_{dR} (T_2) - S_1 \wedge d_{dR} (S_1) + \cdots$.
Thus we see
\[
q (\At(\nabla_\sigma)) = q \circ d B = - d_{dR} d_{dR}(S_1) +  d_{dR} S_1 \wedge d_{dR} S_1.
\]
This is precisely the curvature of $\nabla^E_\sigma$. 
\end{proof}

\subsection{The characteristic class $\log(\hat{A}_u (X))$}\label{sect:A_u}

In encoding $X$ as $B\fg_X$, we use the formalism of Gelfand-Kazhdan formal geometry to construct $\fg_X$; essentially, we replace smooth functions $C^\infty_X$ by the de Rham complex of jets of smooth functions. As a result, our construction of the global observables involves a cochain complex quasi-isomorphic to (shifted) de Rham forms, and the characteristic classes $ch_{k}(B\fg_X)$ all manifestly have cohomological degree 0 in this construction. Thus the difficulty is in identifying $ch_{k}(B\fg_X)$ with the usual Chern classes $ch_k(X)$, and the negative cyclic homology surmounts this difficulty.

Just as we saw with $\Omega^{-\ast}_{T^\ast B\fg_X}$, the complex $\Omega^{-\ast}_{B \fg_X}$ has a $B\GG_a$ action (here we have no internal differential and view $d_{dR}$ as lowering degree by 1).  Now by construction we have a quasi-isomorphism 
\[
\Omega^{-\ast}_{B\fg_X} \simeq \bigoplus_k \Omega^k_\fJ [k].
\]
Hence, by the discussion in Section \ref{circleaction} and Proposition \ref{prop:jetquasiiso},  we have a quasi-isomorphism of complexes of $\Omega^\ast_X$-modules
\[
 (\Omega^{-\ast}_{B \fg_X})^{B \GG_a} \simeq dR (J(\Omega^{-\ast}_X [[u]], ud)).
\]

%

Recall that the characteristic class $ch_{2k} (\nabla_\sigma)$ lives in $\Omega^{2k}_X (\Omega^{2k}_\fJ)$.  We want to obtain a cohomologous class living in the bottom row of our double complex. As the double complex has acyclic columns, we want to use a zig-zag argument.  That is, by Lemma \ref{closed}, $ch_{2k} (\nabla_\sigma)$ is closed with respect to both the horizontal differential (in this case, the de Rham differential) and the vertical differential (the one coming from the jet bundle), and hence its cohomology class in the total complex is represented by  a class $\alpha_1$ of cohomological degree $2k-1$ in $\Omega^{2k+1}_\fJ$.  Continuing in this manner, we obtain a class $\alpha_{2k} \in H^0 (\Omega^{4k}_\fJ) \cong \Omega^{4k}_X$.  

Now we want to identify the image of the class $\alpha_{2k}$ in the complex $(\Omega^{-\ast}_{B \fg_X})^{B \GG_a}$.  From Proposition \ref{prop:jetquasiiso} we have that $\Omega^k_\fJ \cong dR (J(\Omega^k_X))$ as $\Omega^\ast_X$-modules.  Let $\widetilde{ch}_{2k} (\nabla_\sigma)$ denote the class
\[
\frac{1}{(2k)! (-2\pi i)^{2k}} \mathrm{Tr}(\At (\nabla_\sigma)^{2k}) \in \Omega^{2k}_X (\Omega^{2k}_\fJ) [2k] \subset \Omega^{-\ast}_{B \fg_X} .
\]
In order to enact the zig-zag argument (and hence produce a nontrivial cohomology class), we need the de Rham differential that is obtained on $\Omega^{-\ast}_{B \fg_X}$ by taking homotopy invariants with respect to the action of $B \GG_a$.  As we zig-zag down to row zero (that is, $\Omega^0_X (\Omega^{-\ast}_\fJ)$), we pick up a factor of $u$ at each step.  Therefore, if we denote the resulting class by $\widetilde{\alpha}_{2k}$, we have that
\[
\widetilde{ch}_{2k} (\nabla_\sigma) \simeq \widetilde{\alpha}_{2k} \simeq u^{2k} \alpha_{2k} \simeq u^{2k} ch_{2k} (X) \in \Omega^{0}_X (\Omega^{4k}_\fJ)[[u]][4k] \subset (\Omega^{-\ast}_{B \fg_X})^{B \GG_a} .
\]

 Following the presentation of Section \ref{sect:mainthm}, we define for any smooth manifold $X$ the class $\log (\hat{A}_u (X))$ to be 
\[
\log (\hat{A}_u (X)) \overset{def}{=} \sum_{k \ge 1} \frac{2 (2k-1)!}{(2 \pi i)^{2k}} u^{2k} \zeta(2k) ch_{2k} (X) \in (\Omega^{-\ast}_X [[u]], ud ).
\]
This is the usual logarithm of the $\hat{A}$ class  weighted by powers of $u$.   So far, we have argued that $\widetilde{ch}_{2k} (\nabla_\sigma) \simeq u^{2k} ch_{2k} (X) \in dR (J(\Omega^{-\ast}_X [[u]], ud))$.  We now show that 
\[
dR (J(\Omega^{-\ast}_X [[u]], ud)) \simeq (\Omega^{-\ast}_X [[u]], ud).
\]

We consider $dR (J (\Omega^{-\ast}_X [[u]], ud))$ and $(\Omega^{-\ast}_X [[u]] , ud)$ as differential complexes.  Here a differential complex is a complex of sheaves whose graded terms are $\cinf_X$ modules and whose differentials are differential operators. 

\begin{prop}\label{jetquasiiso}
Let $(\sE,d)$ be a differential complex. The natural map of differential complexes, sending a section to its $\infty$-jet, 
\begin{equation}\label{takeinfinitejets}
(\sE,d) \to dR (J (\sE,d))
\end{equation}
is a quasi-isomorphism.
\end{prop}

This proposition is a straightforward sheaf-theoretic argument that we now provide.

\begin{lemma}
\mbox{}
\begin{enumerate}
\item[(1)] Restricted to a contractible open, the map (\ref{takeinfinitejets}) is a quasi-isomorphism.
\item[(2)] As sheaves of graded vector spaces, $\sE$ and $dR(J(\sE))$ are fine.
\end{enumerate}
\end{lemma}

\begin{proof}
To prove the first claim, one constructs a contracting homotopy as in the standard proof of the Poincar\'e lemma. It is simply the assertion that the horizontal sections of jets come from smooth sections of $\sE$. The second claim just follows from the existence of partitions of unity because $\sE$ is a sheaf of $\cinf_X$-modules.
\end{proof}

\begin{proof}[Proof of Proposition \ref{jetquasiiso}]
Fix a good cover $\fU$ of the manifold $X$.  Since $\sE$ is fine we have a quasi-isomorphism $(\sE,d) \xrightarrow{\simeq} \check{C}(\fU , (\sE,d))$.  Similarly, we have a quasi-isomorphism
\[
dR (J (\sE,d)) \xrightarrow{\simeq} \check{C} (\fU , dR (J(\sE,d))).
\]
Therefore we have a commutative diagram
\[
\begin{xymatrix}{
(\sE,d) \ar[r] \ar[d]_{\simeq} & dR (J(\sE,d)) \ar[d]^{\simeq} \\ \check{C} (\fU , (\sE,d)) \ar[r] & \check{C} (\fU , dR (J(\sE,d))) }\end{xymatrix}
\]
where the vertical arrows are quasi-isomorphisms. The map of interest is the on the top row; we will show that the map of the bottom row is a quasi-isomorphism, which implies that the top row is a quasi-isomorphism.

Consider the spectral sequences associated to $\check{C} (\fU , (\sE,d))$ and $\check{C} (\fU , dR(J(\sE,d)))$ where we filter by \v{C}ech degree.  The map $(\sE,d) \to dR(J(\sE,d))$ induces a map of spectral sequences.  This map is a quasi-isomorphism on the $E_1$ page by part 1 of the preceding lemma. Hence, as the spectral sequences converge, it is an isomorphism on cohomology.
\end{proof}
\section{Recovering $\hat{A} (X)$}

In this section we prove a main theorem of this paper, which computes the effective action for one-dimensional Chern-Simons theory with values in the $\L8$ algebra encoding the smooth manifold $X$.

\begin{theorem}
\mbox{}
\begin{itemize}

\item[(1)] There exists a quantization of a nonlinear sigma model from the circle $S^1$ into $T^*[0] X$, where $X$ is a smooth manifold. The solutions to the Euler-Lagrange equations consist of constant maps into $T^* [0]X$. Only 1-loop Feynman diagrams appear in the quantization.

\item[(2)] The scale $\infty$ interaction term $I^{(1)} [\infty]$ encodes $\hat{A}(X)$ when restricted to the harmonic fields $\cH$.  More precisely, we have
\[
I^{(1)} [\infty] |_{\cH} \simeq \log ( \hat{A}_u (T_X)) \in \Omega^{-*}_X [[u]] .
\]
\end{itemize}
\end{theorem}

Theorem \ref{ManifoldThm}, stated in the introduction, is a corollary of this result.

The existence part of the theorem above follows immediately from the following proposition.

\begin{prop}\label{obsthm} 
For one-dimensional Chern-Simons theory with an $\L8$ algebra encoding the smooth geometry of the manifold $X$, the obstruction group is
\[
H^1(X, \Omega^1_{cl} (X)) \oplus H^2 (X, \Omega^1_{cl} (X)) \cong H^2 (X , \RR) \oplus H^3 (X, \RR).
\]
Further, the obstruction to quantization vanishes.
\end{prop}

\begin{proof}
This is just an application of Corollary \ref{obscpxcor}.  That the obstruction is zero follows exactly as in Proposition \ref{1dvanishprop}, i.e.,  from the vanishing of the total Lie factor, $\sum_{\gamma , e} O^{\fg}_{\gamma, e}$.  \end{proof}

\begin{proof}[Proof of Theorem]
Having proved the existence of a quantization, we need is to identify the scale $\infty$ interaction term. We know from above (Theorem \ref{inftytermBg}) that the scale $\infty$ interaction restricted to harmonic fields can be written
\[
I^{(1)} [\infty] \big|_{\cH} = \sum_{k \ge 1}  \left \{ \frac{ 2(2k-1)!}{(2 \pi)^{2k}}  \zeta(2k) ch_{2k} (\nabla_\sigma)  \right \}. 
\]
Here $ch_{2k} (\nabla_\sigma)$ lives $\Omega^{2k}_X (\Omega^{2k}_\fJ)$.  We proved in Section \ref{sect:A_u} that
\[
\sum_{k \ge 1}  \left \{ \frac{ 2(2k-1)!}{(2 \pi)^{2k}}  \zeta(2k) ch_{2k} (\nabla_\sigma)  \right \} 
\simeq \log (A_u (X)) \in \Omega^{-\ast}_X [[u]]. \qedhere
\]
\end{proof}

\newpage

\part{Appendices}
\appendix

\section{Notational and other conventions}

A tensor product $\ot$ without a subscript (usually) denotes $\ot_\RR$. We denote tensoring over another commutative ring $\sA$ by $\ot_\sA$. 

Given a free $R$-module $V$, we denote the dual by $V^\vee$.

We use $\Sym V$ to denote the symmetric algebra (i.e., the direct sum of symmetric powers) and $\csym V$ to denote the completed symmetric algebra (i.e., the direct product of symmetric powers).

Unless we are working with indices, $k$ typically denotes the ``base ring" over which live our $\L8$-algebras, commutative algebras, and so on. In practice, it might be short-hand for $\RR$, $\CC$, $\Omega_X$, or other things that are hopefully clear from context.

We always work with cochain complexes, so that differentials have degree +1. Likewise, we always employ the Koszul rule of signs.

Given a smooth manifold $X$, we denote the sheaf of smooth functions by $C^\infty_X$, of $n$-forms by $\Omega^n_X$, and of vector fields by $\cT_X$. We use $\Omega_X$ to denote the de Rham complex of $X$ as a sheaf of commutative dg algebras.

Given a commutative dg algebra $\sA$, we denote the underlying graded algebra by $\sA^\#$.

Given a vector bundle $\pi: E \to X$ with a flat connection $\nabla$, we denote by $dR(E)$ the associated de Rham complex $\Omega^*_X(E)$ with differential $\nabla$.

Our space of fields $\sE$ will always be sections of a $\ZZ$-graded vector bundle $E \to M$. Given a quadratic action $Q$ and a fiberwise (degree -1) symplectic pairing $\langle - , - \rangle_{loc} : E \otimes E \to \mathrm{Dens} (M)$,  let $D = [Q,Q^\ast]$ be the generalized Laplacian associated to our classical field theory and  for $t \in \RR_{>0}$, let $K_t \in \sE \otimes \sE$ denote the heat kernel for $D$. Our convention for kernels is that for any $\phi \in \sE$,
\[
\int_M \langle K_t (x,y) , \phi (y) \rangle_{loc} = (e^{-tD} \phi ) (x) .
\]
The associated BV Laplacian at scale $L$ is $\Delta_L$, while the scale $L$ BV bracket is denoted by $\{- , - \}_L$.

Given any functional $I \in \sO (\sE)$ we let $W (P_\epsilon^L , I)$ be the renormalization group flow operator which is expressed as a weighted sum of graph weights $W_\gamma (P_\epsilon^L , I)$.  The graph weight $W_{\gamma ,e} (P_\epsilon^L , \Phi, I)$ is given by equipping the edge $e \in \gamma$ by $\Phi \in \mathrm{Sym}^2 (\sE)$ and all remaining edges by $P_\epsilon^L$.

The obstruction to satisfying the QME at scale $L$ is denoted $O[L]$, while the limit as $L \to 0$ is denoted simply by $O$. We use the notation $O_{\gamma, e}$ to denote the contribution of a graph $\gamma$ with edge $e$ to the obstruction. 

Note that in both $W_{\gamma ,e}$ and $O_{\gamma , e}$, the edge $e$ is assumed {\em not} to be a loop.

\section{$\L8$ algebras and their cyclic versions}\label{app:L8}

An $\L8$ algebra is a homotopy coherent weakening of the idea of a Lie algebra, and there is an extensive literature on them. We will provide a minimal overview targeted at the less-conventional aspects that we use.

Let $R$ denote a commutative dg algebra with a nilpotent ideal $I \subset R$.


\begin{definition}
A {\em curved $\L8$ algebra} over $R$ is a locally free, graded $R^\#$-module $L$ with a degree 1 derivation 
\[
d: \csym(L^\vee[-1]) \to \csym(L^\vee[-1])
\]
satisfying 
\begin{itemize} 

\item $d^2 = 0$; 

\item $d$ makes $\csym(L^\vee[-1])$ into a commutative dg algebra over $R$; 

\item modulo $I$, the derivation must preserve the ideal generated by $L^\vee[-1]$ inside $\csym(L^\vee[-1])$.

\end{itemize}

We call the commutative dg algebra $(\csym(L^\vee[-1]), d)$ the {\em Chevalley-Eilenberg complex} of the $\L8$ algebra $L$. 
\end{definition}

When we speak of Koszul duality, we mean the process of moving between an $\L8$ algebra and a commutative dg algebra.

\begin{remark}
The $n$-fold brackets of $L$ are obtained from $d$ as follows. A derivation is determined by its behavior on $L^\vee[-1]$, thanks to the Leibniz rule. Hence we may view $d$ as simply an $R$-linear map from $L^\vee[-1]$ to $\csym(L^\vee[-1])$. Consider the homogeneous components of $d$, namely the maps $d_n: L^\vee[-1] \to \Sym^n(L^\vee[-1])$. If we take the dual, we get maps
\[
\ell_n : \Sym^n(L^\vee[-1])^\vee \to (L^\vee[-1])^\vee,
\]
which we can consider as degree $0$ maps from $(\wedge^n L)[n-2]$ to $L$. These are the Lie brackets on $L$, and we sometimes call them the {\it Taylor coefficients} of the bracket. The higher Jacobi relations between the $\ell_n$ are encoded by the fact that $d^2 = 0$.
\end{remark}

\begin{remark}
A curious aspect of this definition is the curving, since the uncurved case is typically more familiar. Under Koszul duality, there is a natural ``geometric" source for curved $\L8$ algebras (modulo an issue of completion). Consider a map of commutative dg algebras $f: A \to B$, which we view as a map of derived schemes $\Spec B \to \Spec A$. This map makes $B$ an $A$-algebra and so we can find a semi-free resolution $\Sym_A(M)$ of $B$ as an $A$-algebra. This replacement $\Sym_A(M)$ expresses $B$ as a kind of $\L8$ algebra over $A$, namely $\fg_B = M^\vee[-1]$ (here the completion issue appears). Note that if $f$ factors through a quotient $A/I$ of $A$, however, then $\fg_B$ will be curved. This curving appears because $\Spec B$ really only lives over the subscheme $\Spec A/I \subset \Spec A$, and extending it over the rest of $\Spec A$ is obstructed.
\end{remark}

We say a bilinear pairing of degree $k$ $\langle -, - \rangle: L \ot L \to R[-k]$ is {\em nondegenerate} if the induced pairing on cohomology $H^*(\fg) \ot H^*(\fg) \to H^*(R)[-k]$ is perfect.

\begin{definition}
A {\em cyclic $\L8$ algebra} of degree $k$ consists of an $\L8$ algebra $L$ and a nondegenerate symmetric bilinear pairing $\langle -, - \rangle: L \ot L \to R[-k]$ such that
\[
\langle x_1, \ell_n(x_2, \ldots,x_{n+1}) \rangle = (-1)^{n + |x_{n+1}|(|x_1| + \cdots + |x_n|)}\langle x_{n+1}, \ell_n(x_1, \ldots,x_n) \rangle.
\]
\end{definition}



\section{Complexes with a circle action}\label{Mixedcomplexes}

We define the category of complexes with a circle action to be the category of dg modules over $\CC[\epsilon]$, where $\epsilon$ is square zero of cohomological degree 1. This notion is equivalent to dg modules with a $B\GG_a$ module structure, where $B\GG_a$ denotes the dg group manifold $(\pt, \CC[\epsilon])$. Explicitly, an object is just a triple $(V^\ast , d, \epsilon )$ in which $(V^\ast, d)$ is a cochain complex and $\epsilon$ is a degree $-1$ cochain map.

If $V$ is a cochain complex with a $B\GG_a$ module structure, we want to compute the homotopy fixed points of the $B\GG_a$ action, namely $V^{B\GG_a}$ or, equivalently, $V^{h\CC[\epsilon]}$. In other words, we want to compute $\RR \Hom_{\CC[\epsilon]}(\CC, V)$. Note that this will be a module over
\[
\RR \Hom_{\CC[\epsilon]} (\CC, \CC) \simeq \CC [[u]], \; \deg u =2 .
\]
(To an algebraic topologist, this looks like a completed version of the cohomology of $BS^1$.) To do this, we resolve $\CC$ as a $\CC[\epsilon]$ module:
\[
\dotsb \xrightarrow{\cdot \epsilon} \CC[\epsilon] \xrightarrow{\cdot \epsilon} \CC[\epsilon] \xrightarrow{\cdot \epsilon} \CC [\epsilon] \simeq \CC .
 \]
Thus we can compute the homotopy fixed points as the total complex of a double complex.  Indeed, 
\[
V^{B\GG_a} = (V[[u]], d + u \epsilon) = \left ( \prod_{n\ge 0} V[2n] , d \text{ internal to } V \text{ and } \epsilon \text{ shifts between copies} \right).
\]
For a discussion of these ideas, we recommend \cite{TV} and \cite{BZN}.

\section{Differential graded manifolds and derived geometry}
\label{DGmanifolds}

In this paper we use a limited version of ``derived" geometry adequate to our tasks. In essence, we enhance smooth manifolds by allowing ``formal" directions, which allows us to work with certain kinds of derived quotients and derived intersections. For instance --- and we elaborate below --- we consider the space $X_{dR}$, whose structure sheaf is the de Rham complex of the smooth manifold $X$. Unfortunately, we lack the expertise to explain how this formalism fits inside the deeper formalisms recently developed by To{\"e}n-Vezzosi, Lurie, and others. To some extent, what we lose in generality is redeemed by how concrete and easy it is to work with dg manifolds.\footnote{After this paper was written, Costello introduced an approach to derived smooth geometry in his revision of \cite{Cos3}. In \cite{GGonL8}, we provide an introduction to this formalism and explain how it interacts with Costello's QFT formalism. In particular, we describe several versions of the derived loop space and explain their relations.} 

\begin{definition}
A {\em differential graded manifold} (dg manifold, for short) is ringed space $\cX = (X, \sO_\cX)$ where $X$ is a smooth manifold and $\sO_\cX$ is a sheaf on $X$ of commutative dg algebras over $\RR$ (or $\CC$) such that locally the underlying graded algebra of $\sO_\cX$ has the form $\cinf_X \ot \csym V$ for some topological vector space $V$ over $\RR$ (or $\CC$).
\end{definition}

There is a category of dg manifolds where the morphisms are pairs $(f, f^\#): \cX \rightarrow \cY$, with $f: X\rightarrow Y$ a map of smooth manifolds and $f^\#: f^{-1}\sO_\cY \rightarrow \sO_\cX$ a map of commutative dg algebras over $f^{-1} C^\infty_Y$. 

\subsection{Geometric Examples}

Many constructions from differential geometry and topology can be phrased elegantly using dg manifolds.

\begin{itemize}

\item Given a finite-rank $\ZZ$-graded vector bundle $E$ on a smooth manifold $X$, let $E^\vee$ denote the dual bundle and $\sE^\vee$ the sheaf of smooth sections of $E^\vee$. The dg manifold $(X, \csym_{C^\infty_X}(\sE^\vee))$ describes the formal neighborhood of $X$ inside the total space of $E$. For instance, in this paper we often work with the shifted cotangent bundle $T^*[k]X$, which is precisely the dg manifold $(X, \csym_{C^\infty_X}(\cT_X[-k]))$.

\item Let $f: X \rightarrow \RR$ be a smooth function on a smooth manifold $X$ of dimension $n$. Consider the cochain complex, denoted $\sO_{Crit(f)}$,
\[
\cdots \rightarrow 0 \rightarrow \wedge^n \cT_X [n] \overset{\iota_{df}}{\rightarrow} \cdots \overset{\iota_{df}}{\rightarrow} \cT_X[1]  \overset{\iota_{df}}{\rightarrow} C^\infty_X,
\]
where we simply contract the exterior derivative $df$ with vector fields. Observe that $H^0(\sO_{Crit(f)})$  consists of functions on the critical locus of $f$, in the usual sense. We call $dCrit(f) = (X, \sO_{Crit(f)})$ the {\em derived critical locus} of $f$.

\item Given two submanifolds $M, N$ of a smooth manifold $X$, the {\em derived intersection} $M \cap^d N$ is the dg manifold $(X, C^\infty_M \ot^\LL_{C^\infty_X} C^\infty_N)$. (The derived critical locus is the derived intersection inside $T^*X$ of the zero section and the graph of the 1-form $df$.)

\item For $X$ a smooth manifold, $X_{dR} = (X, \Omega_X)$ is a dg manifold that encodes the topology of $X$\footnote{There is another dg manifold $(\pt, \Omega^*(X))$ that knows the real homotopy type of $X$ but nothing more. By contrast, the module sheaves of $X_{dR}$ are essentially the $D$-modules on $X$, and hence $X_{dR}$ knows much more of the topology (via the cohomologically constructible sheaves) and not just homotopy of $X$.}, since we can view the de Rham complex as a resolution of the constant sheaf $\underline{\RR}_X$ on $X$.

\end{itemize}

\subsection{Main examples for this paper}

The de Rham space makes it easy to discuss certain geometric constructions.

\begin{lemma}
A vector bundle $\pi: E \to X_{dR}$ is a vector bundle $\pi_0: E_0 \to X$ with a flat connection $\nabla$.
\end{lemma}

This result suggests that there might be a dg manifold that acts as a classifying space for bundles with flat connection. For instance, given a finite-dimensional Lie algebra $\fg$, consider the dg manifold $B\fg := (\pt, C^*(\fg))$, whose structure sheaf is the Chevalley-Eilenberg cochain complex.

\begin{prop}
The space of flat connections on the trivial $G$-bundle over a smooth manifold $X$ is equivalent to the space of maps from $X_{dR}$ to $B\fg$. 
\end{prop}

We construct a {\em space} of maps, rather than merely a set, by enriching over simplicial sets in the standard way. We define $\Maps(\cX,\cY)$ to be the simplicial set whose $n$-simplices are pairs of a smooth map $f:X \to Y$ and a map of commutative dg algebras $f^\sharp: f^* \sO_\cY \to \sO_\cX \ot \Omega^*(\triangle^n)$. Throughout this paper, however, we will never explicitly use this notion of mapping space, instead working directly with algebras.

It is natural to consider as well {\em families} of $\L8$ algebras over a dg manifold. This description is just an alternative way to discuss a dg manifold. For instance, in this paper, we encode a smooth manifold $X$ as a dg manifold $B\fg_X = (X, C^*(\fg_X))$, where $\fg_X$ is a sheaf of curved $\L8$ algebras over the sheaf of commutative dg algebras $\Omega^*_X$. Hence the dg manifold $B\fg_X$ lives over $X_{dR}$. This is a central construction in the text.

Our final example is the {\em derived loop space} we use throughout the paper. Let $B\fg$ denote a sheaf of curved $\L8$ algebras over $X_{dR}$ (possibly the example above, the holomorphic version in \cite{Cos3}, or something else). Then the derived loop space $\cL B\fg$ is the dg manifold $(X, C^* (\Omega^*(S^1) \ot \fg ))$. There are other dg manifolds that might deserve the name ``derived loop space," but this version is the most relevant for our purpses. When $B\fg$ just lives over a point, this definition essentially coincides with the definition in derived algebraic geometry. More generally, our version plays nicely with the AKSZ construction.

\section{Differential operators, $D$-modules, and $\Omega$-modules}

In this paper we will make use of $D$-modules, jets, and modules over the de Rham complex, so we will provide a rapid overview of the simple technology that we need. We will use nothing deep or difficult in this paper; this appendix is merely a collection of definitions and examples. In fact, it just provides several different ways to talk about differential operators, but given their central role in geometry, this proliferation of language is perhaps not too surprising. 

\subsection{$D$-modules}

For $X$ a smooth manifold, let $D_X$ denote the ring of smooth differential operators on $X$. There are many ways to define this ring. For instance, $D_X$ is the subalgebra of $\End_\CC (C^\infty_X, C^\infty_X)$ generated by left multiplication by $C^\infty_X$ and by smooth vector fields $T_X$. Locally, every differential operator $P$ has the form 
\[
P = \sum_\alpha a_\alpha(x) \partial^\alpha,
\] 
where the $a_\alpha$ are smooth functions and $\partial^\alpha$ is the multinomial notation for a partial derivative.

A left $D_X$ module $M$ is simply a left module for this algebra. One natural source of left $D_X$ modules is given by smooth vector bundles with flat connections. Let $E$ be a smooth vector bundle over $X$ and let $\sE$ denote its smooth sections. If $\sE$ is a left $D_X$ module, then every vector field acts on $\sE$: we have $X \cdot s \in \sE$ for every vector field $X \in T_X$ and every smooth section $s \in \sE$. Equipping $\sE$ with an action of vector fields is equivalent to putting a connection $\nabla$ on $E$. Moreover, we have $[X, Y] \cdot s = X \cdot (Y \cdot s) - Y \cdot (X \cdot s)$ for all $X, Y \in T_X$ and $s \in \sE$. To satisfy the bracket relation, this connection $\nabla$ must be flat.

There is a forgetful functor $F: D_X-mod \rightarrow C^\infty_X-mod$, where we simply forget about how vector fields act on sections of the sheaf. As usual, there is a left adjoint to $F$ given by tensoring with $D_X$:
\[
D_X \ot_{C^\infty_X} -: M \mapsto D_X \ot_{C^\infty_X} M.
\]
Using the forgetful functor, we can equip the category of left $D_X$ modules with a symmetric monoidal product. Namely, we tensor over $C^\infty_X$ and equip $M \ot_{C^\infty_X} N$ with the natural $D_X$ structure
\[
X \cdot (m \ot n) = ( X \cdot m) \ot n + (-1)^{|m|} m \ot (X \cdot n),
\]
for any $X \in T_X$, $m \in M$, and $n \in N$. By construction, $C^\infty_X$ is the unit object in the symmetric monoidal category of left $D_X$ modules. We will write $M \ot N$ to denote $M \ot_{C^\infty_X} N$ unless there is a possibility of confusion.

\begin{remark}
Right $D_X$ modules also appear in this paper and throughout mathematics. For instance, distributions and the sheaf of densities $\Dens_X$ are naturally a right $D_X$ modules, since distributions and densities pair with functions to give numbers. Since we are working with smooth manifolds, however, it is easy to pass back and forth between left and right $D_X$ modules.
\end{remark}

\subsection{Jets}

There is another, beautiful way to relate vector bundles and $D_X$ modules, and we will use it extensively in our constructions. Given a finite rank vector bundle $E$ on $X$, the infinite jet bundle $J(E)$ is naturally a $D_X$ module, as follows. Recall that for a smooth function $f$, the $\infty$-jet of $f$ at a point $x \in X$ is its Taylor series (or, rather, the coordinate-independent object that corresponds to a Taylor series after giving local coordinates around $x$). We can likewise define the $\infty$-jet of a section $s$ of $E$ at a point $x$. The bundle $J(E)$ is the infinite-dimensional vector bundle whose fiber at a point $x$ is the space of $\infty$-jets of sections of $E$ at $x$. This bundle has a tautological connection, since knowing the Taylor series of a section at a point automatically tells us how to do infinitesimal parallel transport. Nonetheless, it is useful to give an explicit formula. Let $x$ be a point in $X$ and pick local coordinates $x_1, \ldots, x_n$ in a small open neighborhood $U$ of $x$. Pick a trivialization of $E$ over $U$ so that
\[
\Gamma(U, J(E)) \cong C^\infty(U) \ot_\RR \RR[[x_1, \ldots, x_n]] \ot_\RR E_x.
\]
We write a monomial $x_1^{a_1} \cdots x_n^{a_n}$ using multinomial notation: for $\alpha = (a_1, \ldots, a_n) \in \NN^n$, $x^\alpha$ denotes the obvious monomial. Hence, given a section $f \ot x^\alpha \ot e \in C^\infty(U) \ot_\RR \RR[[x_1, \ldots, x_n]] \ot_\RR E_x$ and vector field $\partial_j = \partial/\partial x_j$, the connection is
\[
\partial_j \cdot f \ot x^\alpha \ot e = (\partial_j f) \ot x^\alpha \ot e - f \ot (\alpha_j x^{\alpha - e_j}) \ot e.
\]
We are just applying the vector field in the natural way first to the function and then to the monomial. We leave it to the reader to verify that this defines a flat connection.

The following proposition gives a striking reason for the usefulness of jet bundles. Let $\Diff(\sE, \sF)$ denote the differential operators from the $C^\infty_X$-module  $\sE$ to the $C^\infty_X$-module $\sF$.

\begin{prop}\label{diff}
For vector bundles $E$ and $F$ on $X$, $\Diff(\sE,\sF) \cong \Hom_{D_X}(J(E), J(F))$.
\end{prop}

\begin{remark}
A differential operator $P$ is characterized by the fact that, for any point $x \in X$, the linear functional $\lambda: C^\infty(X) \rightarrow \RR, f \mapsto Pf(x)$ is purely local. It is a distribution with support at $x$, and hence $\lambda$ is a finite linear combination of the delta function $\delta_x$ and its derivatives $\partial^\alpha \delta_x$. But this means $\lambda$ depends only on the $\infty$-jet of a function $f$ at $x$.
\end{remark}

What makes this construction useful is that it allows one to translate questions about geometry into questions about $D_X$ modules. There is a rich literature explaining how to exploit this translation, and the usual name for this area of mathematics is (Gelfand-Kazhdan) {\em formal geometry}.

There is another way to construct the sheaf of sections of $J(E)$. Let $\sJ$ denote the sheaf of sections of $J$, the jet bundle for the trivial rank 1 bundle over $X$. Observe for any point $p \in X$,
\[
\sJ_p = \lim_{\leftarrow} \cinf_X/\mathfrak{m_p}^k,
\]
where $\mathfrak{m}_p$ denotes the maximal ideal of functions vanishing at $p$. This equips $\sJ$ with a canonical filtration by ``order of vanishing."  Now let $\sE$ denote the sheaf of smooth sections of $E$, which is a module over $\cinf_X$. Then the sheaf $J(E)$ (we conflate the bundle with its sheaf of sections) has stalk
\[
J(E)_p = \lim_{\leftarrow} \sE/\mathfrak{m_p}^k \sE,
\]
and hence also has a natural filtration by order of vanishing.  Moreover, this shows that $J(E)$ is a module over $\sJ$. We will use the following lemma repeatedly in our constructions.

\begin{lemma}\label{splitting}
A splitting $\sigma: \sE \to J(E)$ of the canonical quotient map $q: J(E) \to \sE$ induces an isomorphism $i_\sigma: J(E) \cong \sE \ot_{C^\infty_X} \sJ$ as $\sJ$-modules.
\end{lemma}

\begin{proof}
Observe that $J(E)$ is a $\sJ$-module just as $\sE$ is a $C^\infty_X$-module. Thus we obtain a map
\[
\sJ \ot \sE \to J(E)
\]
\[
j \ot s \mapsto j \cdot \sigma(s).
\]
We need to show this map is an isomorphism of $\cinf_X$ modules. It is enough to check this locally, so notice that for any small ball $B \subset X$, if we pick coordinates $x_1, \ldots, x_n$ on $B$, we get trivializations
\[
\sE|_B \cong \cinf_X(B) \ot E_0,\; \sJ|_B \cong \cinf_X(B) \ot \RR[[x_1,\ldots,x_n]], \text{ and }
\]
\[
J(E)|_B \cong \cinf_X(B) \ot \RR[[x_1,\ldots,x_n]] \ot E_0,
\]
where $E_0$ denotes the fiber of $E$ over the point $0 \in B$. Let $\{e_i\}$ denote a basis for $E_0$; the ``constant" sections $\{1 \ot e_i\}$ in $\sE$ then form a frame for $\sE$ over $B$. Let $s_i = \sigma(e_i)$. Notice that under the map $J(E)/F^1 \to \sE$, $s_i$ goes to $e_i$, and so the $s_i$ are linearly independent in $J(E)$. By linear algebra over $\sJ$, one obtains that the map $i_\sigma$ is an isomorphism.
\end{proof}

\subsection{$\Omega$-modules}

Let $\Omega_X$ denote the de Rham complex of $X$ and $\Omega^\#_X$ the underlying graded algebra. An $\Omega_X$ module is a graded module $M^*$ over $\Omega^\#_X$ with a differential $\partial$ that satisfies
\[
\partial (\omega \cdot m) = (d\omega) \cdot m + (-1)^{|\omega|} \omega \cdot \partial m,
\]
where $\omega \in \Omega_X$ and $m \in M$. A natural source of examples is (again!) vector bundles with flat connection. Let $E$ be a vector bundle. Differential forms with values in $E$, $\Omega^\#_X(E)$, naturally form a graded module over $\Omega^\#_X$. Equipping $\Omega^\#_X(E)$ with a differential is exactly the same data as a flat connection $\nabla$ on $E$. We call it the {\em de Rham complex} of $(E,\nabla)$.

The category of $\Omega_X$ modules is symmetric monoidal in the obvious way. Given two $\Omega_X$ modules $M$ and $N$, then $M \ot_{\Omega_X} N$ is, as a graded module, the tensor product $M \ot_{\Omega_X^\#} N$ equipped with differential 
\[
\partial (m \ot n) = \partial_M m \ot n +(-1)^{|m|} m \ot \partial_N n.
\]
Of course, it is better to work with the derived tensor product in most situations.

Since $\Omega_X$ is commutative, there is a dg manifold $X_{dR} = (X, \Omega_X)$. It clearly captures the smooth topology of the manifold $X$. Many of our constructions in this paper involve $X_{dR}$. Moreover, many classical constructions in differential geometry (e.g., the Fr\"olicher-Nijenhuis bracket) appear most naturally as living on $X_{dR}$.

\subsection{The de Rham complex of a left $D$-module}

Earlier, we explained how a vector bundle with flat connection $(E, \nabla)$ is a left $D$-module and how to use the connection to make $\Omega^*(E)$ into an $\Omega_X$ module. We now extend this construction to all left $D$-modules. 

Let $M$ be a left $D$-module. The {\em de Rham complex} $dR(M)$ of $M$ consists of the graded $\cinf_X$-module $\Omega^\#_X \ot_{\cinf_X} M$ equipped with the differential
\[
d_M: \omega \ot m \mapsto d\omega \ot m + (-1)^{|\omega|} \sum_i dx_i \wedge \omega \ot \frac{\partial}{\partial x_i} m.
\]
By construction, $dR(M)$ is an $\Omega$-module.




\section{Feynman diagram computation: the proof of Proposition \ref{noctterms}}\label{diagrams}\label{nocttermproof}

Recall that a one-loop graph is called a {\it wheel} if it cannot be disconnected by the removal of a single edge.

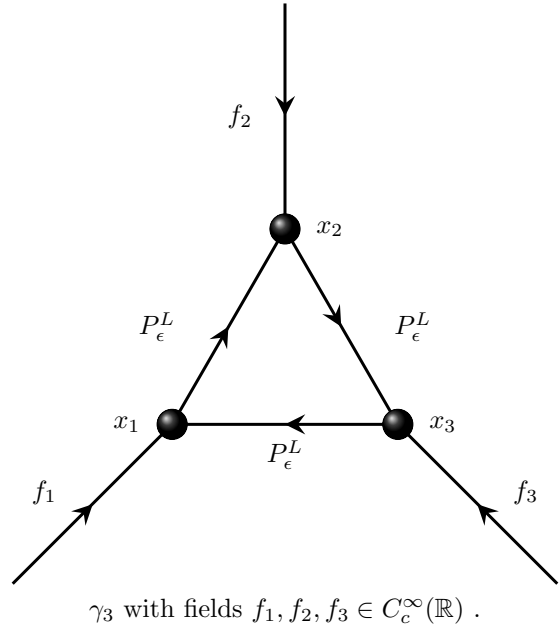
\begin{center}
\begin{tikzpicture}[decoration={markings,
   mark=at position 1.5cm with {\arrow[black,line width=.8mm]{stealth}}}];
\draw[postaction=decorate, line width=.4mm] (-1.5,0) -- (0,2.6);
\draw[postaction=decorate, line width=.4mm] (0,2.6) -- (1.5,0);
\draw[postaction=decorate, line width=.4mm] (1.5,0) -- (-1.5,0);
\draw[postaction=decorate, line width=.4mm] (0,5.6) -- (0,2.6);
\draw[postaction=decorate, line width=.4mm] (-3.62,-2.12) -- (-1.5,0);
\draw[postaction=decorate, line width=.4mm] (3.62,-2.12) -- (1.5,0);
\draw[ball color=black]  (-1.5,0) circle (.2);
\draw[ball color=black]  (1.5,0) circle (.2);
\draw[ball color=black]  (0,2.6) circle (.2);
\draw (-3.2,-.9) node {$f_1$};
\draw (-.6,4.1) node {$f_2$};
\draw (3.2,-.9) node {$f_3$};
\draw (-1.7,1.3) node {$P_\epsilon^L$};
\draw (1.7,1.3) node {$P_\epsilon^L$};
\draw (0,-.4) node {$P_\epsilon^L$};
\draw (-2.1,0) node {$x_1$};
\draw (.6,2.6) node {$x_2$};
\draw (2.1,0) node {$x_3$};
\end{tikzpicture}
\\$\gamma_3$ with fields $f_1,f_2,f_3 \in C^\infty_c (\RR)$ .
\end{center}

\vspace{24pt}

Any one-loop graph is a wheel with trees attached.  As trees don't contribute any singularities (see chapter 2 section 5 of \cite{Cos1}), it is sufficient to prove that the $\epsilon \to 0$ limit exists for the analytic factor $W^{an}_\gamma (P_\epsilon^L I_{CS})$, where $\gamma$ is a wheel. Further, if the limit exists for trivalent wheels, then it exists for wheels with greater valency, since the higher valence vertices simply multiply the incoming functions and hence behave just like trivalent vertices.

Let $\gamma_n$ be a trivalent wheel with $n$ vertices and pick $f_1 , \dotsc , f_n \in C^\infty_c (\RR)$. We then have an explicit integral for $W^{an}_{\gamma_n}(P_\epsilon^L,I_{CS})$:
\begin{equation}\label{analweightequation}
W^{an}_{\gamma_n} (P_\epsilon^L, I_{CS}) (f_1 , \dotsc , f_n) = \int_{x_1 , \dotsc x_n \in \RR} \prod_{i=1}^n f_i (x_i) P_\epsilon^L (x_i , x_{i+1 \text{ mod }n}) \prod_{i=1}^n dx_i .
\end{equation}
The analytic piece of the propagator is given by
\[
P_\epsilon^L = \int_{\epsilon}^L \frac{d}{dx_1} K_t dt,
\]
with $K_t \in C^\infty (\RR \times \RR)$ given (up to a scalar) by
\[
K_t (x_1 , x_2) = t^{-1/2} e^{- \lvert x_1 - x_2 \rvert^2 / t} .
\]

We view the graph weight as a distribution on $\RR^{n}$ and from hereon replace $\prod f_i (x_i)$ by a generic test function (i.e. compact support) $\phi (x)$ on $\RR^{n}$.  Note that from step to step the actual test function may change e.g., as a result of an integration by parts, but for notational convenience (and because the resulting function will again be sufficiently nice) we continue to use the notation $\phi (x)$. Now the graph weight is given by the integral
\[
\lim_{\epsilon \to 0} \int_{\vec{t} \in [\epsilon , L]^n} \int_{\vec{x} \in \RR^{n}} \phi (\vec{x}) \prod_{i=1}^n t_i^{-1/2} \frac{d}{dx_i} e^{- \lvert x_i -x_{i+1 \text{ mod }n} \rvert^2 / t_i } d^n x d^n t ,
\]
where $\vec{t} = (t_1 , \dotsc , t_n)$ and $\vec{x} = (x_1 , \dotsc , x_n)$.  Note that integrand is symmetric in the $t_i$ so if the limit exists then the corresponding limit will exist for any permutation of the $t_i$.  Hence, it is sufficient to integrate the time variables over the $n$-simplex (as opposed to the $n$-cube) given by
\[ 
\epsilon \le t_1 \le t_2 \le \dotsb \le t_n \le L
\]
which we denote by $\Delta^n (\epsilon , L)$.

Note that, aside from $\phi$, the integrand is invariant under translation along the ``small diagonal." In other words, if we change all the $x_i$ by the same amount, the integrand is unchanged. Foliate $\RR^n$ by hyperplanes orthogonal to the small diagonal. Any test function $\phi$ can be approximated by a sum of products $\phi_d \phi_a$, where $\phi_a$ only depends on the anti-diagonal coordinates and $\phi_d$ depends on the diagonal.  As integration along the small diagonal is against a compactly supported function, it is sufficient to consider a test function $\phi$ which is only a function of the anti-diagonal coordinates and show the following is well defined

\begin{equation}\label{elimit}
\lim_{\epsilon \to 0} \int_{\vec{t} \in \Delta^n (\epsilon ,L)} \int_{\stackrel{\vec{x} \in \RR^n}{\sum x_i =0}} \phi (\vec{x}) \prod_{i=1}^n t_i^{-1/2} \frac{d}{dx_i} e^{-\lvert x_i -x_{i+1 \text{ mod }n} \rvert^2/t_i} d^n x d^n t.\end{equation}

We proceed (separately) to show this limit exists in the case $n \ge 2$ and $n=1$.

\subsection{The case $n=1$}
Let $\gamma_1$ be a one vertex wheel (i.e., the hangman's noose), then
\[
\lim_{\epsilon \to 0} W_{\gamma_1} (P_\epsilon^L,  I_{CS} )=0.
\]
Indeed, as there is just one vertex there is a $\frac{d}{dx} K_l (x,x)$ in the integrand, which clearly vanishes as the heat kernel reaches a maximum on the diagonal.

\subsection{The case $n \ge 2$}

We begin by a change of coordinates; let $u_i = (x_i -x_{i+1})$ for $i =1 , \dotsc , n-1$.  The integral in equation \ref{elimit} becomes
\[ 
\int_{\vec{t} \in \Delta^n (\epsilon, L)} \int_{\vec{u} \in \RR^{n-1}} \phi(\vec{u}) \left (  \prod_{i=1}^{n-1} t_i^{-3/2} u_i e^{-\lvert u_i \rvert^2/t_i}  \right ) \left ( t_n^{-3/2} \sum_{i=1}^{n-1} u_i e^{-\lvert \sum  u_i \rvert^2/t_n}  \right )d^{n-1} u d^n t.
\]
This integral is bounded, in absolute value, by
\[ 
\int_{\vec{t} \in \Delta^n (0, L)} \int_{\vec{u} \in \RR^{n-1}} \left (  \prod_{i=1}^{n-1} t_i^{-3/2} \lvert u_i \rvert e^{-\lvert u_i \rvert^2/t_i}  \right ) \left ( t_n^{-3/2} \sum_{i=1}^{n-1} \lvert u_i \rvert  \right ) )d^{n-1} u d^n t.
\]

Now let $v_i = t_i^{-1/2} u_i$ for $i =1 , \dotsc , n$.  Our absolute bound then becomes
\[
\int_{\vec{t} \in \Delta^n (0, L)} \int_{\vec{v} \in \RR^{n-1}}\left ( \prod_{i=1}^{n-1} t_i^{-1/2} \lvert v_i \rvert e^{-\lvert v_i \rvert^2}  \right ) \left ( t_n^{-3/2}  \sum_{i=1}^{n-1} t_i^{1/2} \lvert v_i \rvert  \right ) )d^{n-1} v d^n t.
\]
Using the fact that $t_i \le t_n$ for $i =1 , \dotsc , n-1$ we that the integral is bounded by
\[
\left ( \int_{\vec{t} \in \Delta^n (0, L)} \left( \prod_{i=1}^{n-1} t_i^{-1/2} dt_i \right) t_n^{-1} dt_n \right ) \left ( \int_{\vec{v} \in \RR^{n-1}} P( \lvert v_1 \rvert, \dotsc , \lvert v_{n-1} \rvert) e^{-\sum \lvert v_i \rvert^2} \prod dv_i \right),
\]
for $P(v)$ some polynomial in the variables $\lvert v_i \rvert$.  Note that the second term in parantheses is bounded since $e^{-x^2}$ decays faster than any polynomial in $x$ grows. Thus it suffices to show the first term in parantheses is also bounded. 

Observe that $\int_a^b t^{-1/2} dt = 2(b^{1/2} - a^{1/2})$ for $b > a > 0$. Hence we find
\[
0 \leq \int_{\vec{t} \in \Delta^n (0,t_n)} \prod_{i=1}^{n-1} t_i^{-1/2} dt_i \leq \left( \int_{0< t < t_n} t^{-1/2} dt \right)^{n-1} \leq 2^{n-1} t_n^{(n-1)/2}
\]
and so 
\[
\int_{\vec{t} \in \Delta^n (0,L)} \left( \prod_{i=1}^{n-1} t_i^{-1/2} dt_i \right) t_n^{-1} dt_n \leq \int_0^L 2^{n-1} t_n^{(n-3)/2} dt_n.
\]
When $n > 1$, this integral is clearly bounded.

\bibliographystyle{amsalpha}
\bibliography{chernsimons}

\newcommand{\etalchar}[1]{$^{#1}$}
\def\cprime{$'$} \def\cprime{$'$} \def\cprime{$'$}
\providecommand{\bysame}{\leavevmode\hbox to3em{\hrulefill}\thinspace}
\providecommand{\MR}{\relax\ifhmode\unskip\space\fi MR }
\providecommand{\MRhref}[2]{%
  \href{http://www.ams.org/mathscinet-getitem?mr=#1}{#2}
}
\providecommand{\href}[2]{#2}
\begin{thebibliography}{WMLI92}

\bibitem[ASZK97]{AKSZ}
M.~Alexandrov, A.~Schwarz, O.~Zaboronsky, and M.~Kontsevich, \emph{The geometry
  of the master equation and topological quantum field theory}, Internat. J.
  Modern Phys. A \textbf{12} (1997), no.~7, 1405--1429. \MR{1432574
  (98a:81235)}

\bibitem[Ati57]{At}
M.~F. Atiyah, \emph{Complex analytic connections in fibre bundles}, Trans.
  Amer. Math. Soc. \textbf{85} (1957), 181--207. \MR{0086359 (19,172c)}

\bibitem[BGV92]{BGV}
Nicole Berline, Ezra Getzler, and Mich{\`e}le Vergne, \emph{Heat kernels and
  {D}irac operators}, Grundlehren der Mathematischen Wissenschaften
  [Fundamental Principles of Mathematical Sciences], vol. 298, Springer-Verlag,
  Berlin, 1992. \MR{1215720 (94e:58130)}

\bibitem[BK04]{BK}
R.~Bezrukavnikov and D.~Kaledin, \emph{Fedosov quantization in algebraic
  context}, Mosc. Math. J. \textbf{4} (2004), no.~3, 559--592, 782. \MR{2119140
  (2006j:53130)}

\bibitem[BNT02]{BNT}
P.~Bressler, R.~Nest, and B.~Tsygan, \emph{Riemann-{R}och theorems via
  deformation quantization. {I}, {II}}, Adv. Math. \textbf{167} (2002), no.~1,
  1--25, 26--73. \MR{1901245 (2003i:53131)}

\bibitem[BR73]{BR}
I.~N. Bern{\v{s}}te{\u\i}n and B.~I. Rosenfel{\cprime}d, \emph{Homogeneous
  spaces of infinite-dimensional {L}ie algebras and the characteristic classes
  of foliations}, Uspehi Mat. Nauk \textbf{28} (1973), no.~4(172), 103--138.
  \MR{0415633 (54 \#3714)}

\bibitem[BZN]{BZN}
D.~Ben-Zvi and D.~Nadler, \emph{Loop {S}paces and {C}onnections}, available at
  \href{http://front.math.ucdavis.edu/1002.3636}{arXiv:1002.3636}.

\bibitem[C{\u{a}}l05]{Cald2}
Andrei C{\u{a}}ld{\u{a}}raru, \emph{The {M}ukai pairing. {II}. {T}he
  {H}ochschild-{K}ostant-{R}osenberg isomorphism}, Adv. Math. \textbf{194}
  (2005), no.~1, 34--66. \MR{2141853 (2006a:14029)}

\bibitem[CFT02]{CFT}
Alberto~S. Cattaneo, Giovanni Felder, and Lorenzo Tomassini, \emph{From local
  to global deformation quantization of {P}oisson manifolds}, Duke Math. J.
  \textbf{115} (2002), no.~2, 329--352. \MR{1944574 (2004a:53114)}

\bibitem[CG]{CG}
Kevin Costello and Owen Gwilliam, \emph{Factorization algebras in quantum field
  theory}, book-in-progress available at
  \url{http://math.northwestern.edu/~costello/renormalization}.

\bibitem[Cosa]{Cos2}
Kevin Costello, \emph{A geometric construction of the {W}itten genus, {I}},
  available at \href{http://front.math.ucdavis.edu/1006.5422}{arXiv:1006.5422}.

\bibitem[Cosb]{Cos3}
\bysame, \emph{A geometric construction of the {W}itten genus, {II}}, available
  at \href{http://front.math.ucdavis.edu/1112.0816}{arXiv:1112.0816}.

\bibitem[Cos11]{Cos1}
\bysame, \emph{Renormalization and effective field theory}, Mathematical
  Surveys and Monographs, vol. 170, American Mathematical Society, Providence,
  RI, 2011. \MR{2778558}

\bibitem[CR11]{CalaqueRossi}
Damien Calaque and Carlo~A. Rossi, \emph{Lectures on {D}uflo isomorphisms in
  {L}ie algebra and complex geometry}, EMS Series of Lectures in Mathematics,
  European Mathematical Society (EMS), Z\"urich, 2011. \MR{2816610
  (2012j:53121)}

\bibitem[CSX]{CSX}
Zhuo Chen, Mathieu Stienon, and Ping Xu, \emph{From {A}tiyah classes to
  homotopy {L}eibniz algebras}, available at
  \href{http://front.math.ucdavis.edu/1204.1075}{arXiv:1204.1075}.

\bibitem[CVdB10]{CVDB}
Damien Calaque and Michel Van~den Bergh, \emph{Hochschild cohomology and
  {A}tiyah classes}, Adv. Math. \textbf{224} (2010), no.~5, 1839--1889.
  \MR{2646112 (2011i:14037)}

\bibitem[CW10]{Cald1}
Andrei C{\u{a}}ld{\u{a}}raru and Simon Willerton, \emph{The {M}ukai pairing.
  {I}. {A} categorical approach}, New York J. Math. \textbf{16} (2010), 61--98.
  \MR{2657369 (2011g:18012)}

\bibitem[DEF{\etalchar{+}}99]{IAS}
Pierre Deligne, Pavel Etingof, Daniel~S. Freed, Lisa~C. Jeffrey, David Kazhdan,
  John~W. Morgan, David~R. Morrison, and Edward Witten (eds.), \emph{Quantum
  fields and strings: a course for mathematicians. {V}ol. 1, 2}, American
  Mathematical Society, Providence, RI, 1999, Material from the Special Year on
  Quantum Field Theory held at the Institute for Advanced Study, Princeton, NJ,
  1996--1997. \MR{1701618 (2000e:81010)}

\bibitem[Dol06]{Dolg}
Vasiliy Dolgushev, \emph{A formality theorem for {H}ochschild chains}, Adv.
  Math. \textbf{200} (2006), no.~1, 51--101. \MR{2199629 (2006m:16010)}

\bibitem[FBZ04]{BF}
Edward Frenkel and David Ben-Zvi, \emph{Vertex algebras and algebraic curves},
  second ed., Mathematical Surveys and Monographs, vol.~88, American
  Mathematical Society, Providence, RI, 2004. \MR{2082709 (2005d:17035)}

\bibitem[Fed94]{FedDQ}
Boris~V. Fedosov, \emph{A simple geometrical construction of deformation
  quantization}, J. Differential Geom. \textbf{40} (1994), no.~2, 213--238.
  \MR{1293654 (95h:58062)}

\bibitem[Fed96]{FedBook}
Boris Fedosov, \emph{Deformation quantization and index theory}, Mathematical
  Topics, vol.~9, Akademie Verlag, Berlin, 1996. \MR{1376365 (97a:58179)}

\bibitem[Fra13]{Francis}
John Francis, \emph{The tangent complex and {H}ochschild cohomology of
  {E}[n]-rings}, Compos. Math. \textbf{149} (2013), no.~3, 430--480.
  \MR{3040746}

\bibitem[Get83]{Getzler}
Ezra Getzler, \emph{Pseudodifferential operators on supermanifolds and the
  {A}tiyah-{S}inger index theorem}, Comm. Math. Phys. \textbf{92} (1983),
  no.~2, 163--178. \MR{728863 (86a:58104)}

\bibitem[GG]{GGonL8}
Ryan Grady and Owen Gwilliam, \emph{${L}_\infty$ spaces in derived smooth
  geometry and quantum field theory}, available at
  \url{http://math.berkeley.edu/\~gwilliam}.

\bibitem[Gil95]{Gilkey}
Peter~B. Gilkey, \emph{Invariance theory, the heat equation, and the
  {A}tiyah-{S}inger index theorem}, second ed., Studies in Advanced
  Mathematics, CRC Press, Boca Raton, FL, 1995. \MR{1396308 (98b:58156)}

\bibitem[GK71]{GK}
I.~M. Gel{\cprime}fand and D.~A. Ka{\v{z}}dan, \emph{Some problems of
  differential geometry and the calculation of cohomologies of {L}ie algebras
  of vector fields}, Soviet Math. Doklady \textbf{12} (1971), 1367--1370.

\bibitem[GKF72]{GKF}
I.~M. Gel{\cprime}fand, D.~A. Ka{\v{z}}dan, and D.~B. Fuks, \emph{Actions of
  infinite-dimensional {L}ie algebras}, Funkcional. Anal. i Prilo\v zen.
  \textbf{6} (1972), no.~1, 10--15. \MR{0301767 (46 \#922)}

\bibitem[Graa]{Grady2}
Ryan Grady, \emph{The $\hat{A}$-genus as a projective volume form on the
  derived loop space}, available at
  \href{http://front.math.ucdavis.edu/1211.6816}{arXiv:1211.6816}.

\bibitem[Grab]{Grady}
\bysame, \emph{On geometric aspects of topological quantum mechanics}, PhD
  Dissertation, University of Notre Dame, 2012.

\bibitem[HBJ92]{BHJ}
Friedrich Hirzebruch, Thomas Berger, and Rainer Jung, \emph{Manifolds and
  modular forms}, Aspects of Mathematics, E20, Friedr. Vieweg \& Sohn,
  Braunschweig, 1992, With appendices by Nils-Peter Skoruppa and by Paul Baum.
  \MR{1189136 (94d:57001)}

\bibitem[Hir95]{H}
Friedrich Hirzebruch, \emph{Topological methods in algebraic geometry},
  Classics in Mathematics, Springer-Verlag, Berlin, 1995, Translated from the
  German and Appendix One by R. L. E. Schwarzenberger, With a preface to the
  third English edition by the author and Schwarzenberger, Appendix Two by A.
  Borel, Reprint of the 1978 edition. \MR{1335917 (96c:57002)}

\bibitem[Ill71]{Illusie}
Luc Illusie, \emph{Complexe cotangent et d\'eformations. {I}}, Lecture Notes in
  Mathematics, Vol. 239, Springer-Verlag, Berlin, 1971. \MR{0491680 (58
  \#10886a)}

\bibitem[Kap99]{Kap}
M.~Kapranov, \emph{Rozansky-{W}itten invariants via {A}tiyah classes},
  Compositio Math. \textbf{115} (1999), no.~1, 71--113. \MR{1671737
  (2000h:57056)}

\bibitem[Mar09]{Mark}
Nikita Markarian, \emph{The {A}tiyah class, {H}ochschild cohomology and the
  {R}iemann-{R}och theorem}, J. Lond. Math. Soc. (2) \textbf{79} (2009), no.~1,
  129--143. \MR{2472137 (2010d:14020)}

\bibitem[PPT10]{PPT}
M.~J. Pflaum, H.~Posthuma, and X.~Tang, \emph{Cyclic cocycles on deformation
  quantizations and higher index theorems}, Adv. Math. \textbf{223} (2010),
  no.~6, 1958--2021. \MR{2601006 (2011d:58056)}

\bibitem[Ram08]{Ram}
Ajay~C. Ramadoss, \emph{The big {C}hern classes and the {C}hern character},
  Internat. J. Math. \textbf{19} (2008), no.~6, 699--746. \MR{2431634
  (2010h:14028)}

\bibitem[SS85]{SchlessingerStasheff}
Michael Schlessinger and James Stasheff, \emph{The {L}ie algebra structure of
  tangent cohomology and deformation theory}, J. Pure Appl. Algebra \textbf{38}
  (1985), no.~2-3, 313--322. \MR{814187 (87e:13019)}

\bibitem[Tsy99]{Tsygan}
B.~Tsygan, \emph{Formality conjectures for chains}, Differential topology,
  infinite-dimensional {L}ie algebras, and applications, Amer. Math. Soc.
  Transl. Ser. 2, vol. 194, Amer. Math. Soc., Providence, RI, 1999,
  pp.~261--274. \MR{1729368 (2001g:53161)}

\bibitem[TV09]{TV}
Bertrand To{\"e}n and Gabriele Vezzosi, \emph{Chern character, loop spaces and
  derived algebraic geometry}, Algebraic topology, Abel Symp., vol.~4,
  Springer, Berlin, 2009, pp.~331--354. \MR{2597742 (2011d:14031)}

\bibitem[Wil11]{Will}
Thomas Willwacher, \emph{Formality of cyclic chains}, Int. Math. Res. Not. IMRN
  (2011), no.~17, 3939--3956. \MR{2836399}

\bibitem[Win84]{Windey}
P.~Windey, \emph{Supersymmetric quantum mechanics and the {A}tiyah-{S}inger
  index theorem}, Acta Phys. Polon. B \textbf{15} (1984), no.~5, 435--452.
  \MR{757583 (86g:58132)}

\bibitem[Wit82]{Wit82}
Edward Witten, \emph{Constraints on supersymmetry breaking}, Nuclear Phys. B
  \textbf{202} (1982), no.~2, 253--316. \MR{668987 (84j:81131)}

\bibitem[WMLI92]{Cartier}
M.~Waldschmidt, P.~Moussa, J.~M. Luck, and C.~Itzykson (eds.), \emph{From
  number theory to physics}, Springer-Verlag, Berlin, 1992, Papers from the
  Meeting on Number Theory and Physics held in Les Houches, March 7--16, 1989.
  \MR{1221099 (93m:11001)}

\end{thebibliography}

\end{document}